\pdfoutput=1
\RequirePackage{ifpdf}
\ifpdf 
\documentclass[pdftex]{sigma}
\else
\documentclass{sigma}
\fi

\numberwithin{equation}{section}

\newtheorem{Theorem}{Theorem}[section]
\newtheorem{Corollary}[Theorem]{Corollary}
\newtheorem{Lemma}[Theorem]{Lemma}
\newtheorem{Proposition}[Theorem]{Proposition}
 { \theoremstyle{definition}
\newtheorem{Definition}[Theorem]{Definition}
\newtheorem{Example}[Theorem]{Example}
\newtheorem{Remark}[Theorem]{Remark} }

\usepackage{etoolbox}
\apptocmd{\sloppy}{\hbadness 10000\relax}{}{} 

\usepackage{mathtools}

\usepackage{textcomp,mathrsfs,stmaryrd,bm,braket}

\usepackage{tikz}
\tikzstyle{vertex}=[circle,draw,inner sep=0pt,minimum size=12pt] 
\newcommand{\vertex}{\node[vertex]}
\tikzset{
every picture/.style={thick,>=latex,->,node distance=2cm} 
}

\DeclareMathOperator{\Der}{Der}

\DeclareMathOperator{\GL}{GL}
\DeclareMathOperator{\SL}{SL}

\DeclareMathOperator{\cycl}{cycl}
\DeclareMathOperator{\diag}{diag}

\DeclareMathOperator{\Rep}{Rep}

\DeclareMathOperator{\Sym}{Sym}
\DeclareMathOperator{\Tens}{Tens}
\DeclareMathOperator{\gr}{gr}

\DeclareMathOperator{\Id}{Id}
\DeclareMathOperator{\Rees}{Rees}
\DeclareMathOperator{\IMD}{IMD}
\DeclareMathOperator{\Hom}{Hom}
\DeclareMathOperator{\End}{End}
\DeclareMathOperator{\Ker}{Ker}
\DeclareMathOperator{\Tr}{Tr}

\DeclareMathOperator{\dR}{dR}

\DeclareMathOperator{\Aut}{Aut}

\DeclareMathOperator{\Lie}{Lie}

\DeclareMathOperator{\Card}{Card}

\begin{document}
\allowdisplaybreaks

\newcommand{\arXivNumber}{1905.07713}

\renewcommand{\PaperNumber}{103}

\FirstPageHeading

\ShortArticleName{Symmetries of the Simply-Laced Quantum Connection}

\ArticleName{Symmetries of the Simply-Laced Quantum\\ Connections and Quantisation of Quiver Varieties}

\Author{Gabriele REMBADO}

\AuthorNameForHeading{G.~Rembado}

\Address{Hausdorff Centre for Mathematics, Endenicher Allee 62, D-53115, Bonn, Germany}
\Email{\href{mailto:gabriele.rembado@hcm.uni-bonn.de}{gabriele.rembado@hcm.uni-bonn.de}}

\ArticleDates{Received May 02, 2020, in final form October 13, 2020; Published online October 17, 2020}

\Abstract{We will exhibit a group of symmetries of the simply-laced quantum connections, generalising the quantum/Howe duality relating KZ and the Casimir connection. These symmetries arise as a quantisation of the classical symmetries of the simply-laced isomonodromy systems, which in turn generalise the Harnad duality. The quantisation of the classical symmetries involves constructing the quantum Hamiltonian reduction of the representation variety of any simply-laced quiver, both in filtered and in deformation quantisation.}

\Keywords{isomonodromic deformations; quantum integrable systems; quiver varieties; deformation quantisation; quantum Hamiltonian reduction}

\Classification{81R99}

\section{Introduction and main results}\label{sec:intro}

In a previous paper~\cite{rembado_2019_simply_laced_quantum_connections_generalising_kz} a large class of quantum connections was constructed (the simply-laced quantum connections) that contains as special cases the KZ connection~\cite{knizhnik_zamolodchikov_1984_current_algebra_and_wess_zumino_model_in_two_dimensions}, the Casimir connection~\cite{millson_toledanolaredo_2005_casimir_operators_and_monodromy_representations_of_generalised_braid_groups} and the FMTV connection~\cite{felder_markov_tarasov_varchenko_2000_differential_equations_compatible_with_kz_equations}. They arose by quantising the simply-laced isomonodromy systems of~\cite{boalch_2012_simply_laced_isomonodromy_systems}, each of which is a flat \textit{nonlinear} connection, controlling the isomonodromy deformations of certain meromorphic connections on the Riemann sphere. It was shown in~\cite{boalch_2012_simply_laced_isomonodromy_systems} that the simply-laced isomonodromy systems admit an $\SL_2(\mathbb{C})$ symmetry group, obtained by working with modules for the one-dimensional Weyl algebra, rather than meromorphic connections on the Riemann sphere, and then considering the natural symmetries of the Weyl algebra.

The aim of this paper is to show that the $\SL_2(\mathbb{C})$ symmetry group may be lifted to the simply-laced quantum connections.

In more detail, a nonautonomous integrable Hamiltonian system $H \colon \mathbb{M} \times \mathbf{B} \to \mathbb{C}^I$ is defined on a certain trivial symplectic fibration in~\cite{boalch_2012_simply_laced_isomonodromy_systems}, and shown to control the isomonodromic deformations of meromorphic connections on holomorphically trivial vector bundles $U \times \mathbb{C}P^1 \to \mathbb{C}P^1$ on the Riemann sphere, where $U$ is a finite dimensional complex vector space.

The meromorphic connections we consider can be written
\begin{equation}\label{eq:meromorphic_connections_intro}
 \nabla = {\rm d} - \left(Az + X + \sum_{i \in I'} \frac{R_i}{z - t_i}\right){\rm d}z,
\end{equation}
where $A,X,R_i \in \End(U)$, $z$ is the standard holomorphic coordinate on $\mathbb{C} \subseteq \mathbb{C}P^1$, and $t_i \in \mathbb{C}$ for $i \in I'$ (a finite set). This connections has a simple pole at~$t_i$ with residue~$R_i$, and a pole of order three at infinity when $A \neq 0$. One can then deform~\eqref{eq:meromorphic_connections_intro} by varying the positions of the simple poles (regular singularities); but importantly deformations of the principal part of the connection at infinity can be introduced as follows.
Assume the space $U = \bigoplus_{j \in J'} W^j$ is graded by a finite set $J'$, decompose $X = T + B$, where $T$ and $B$ are the diagonal and off-diagonal part of $X$ in the grading, and suppose the components $T^j \in \End\big(W^j\big)$ of $T$ be diagonalisable (allowing for $A$ and $T$ to have repeated eigenvalues). Now vary: (i)~the position of the simple poles, so that distinct ones do not coalesce; and (ii)~the spectrum of $T^j$, so that the eigenspace decomposition of $W^j$ is not perturbed. Accordingly, one looks for deformations of $B$ and the $R_i$ such that the extended monodromy (Stokes data) of~\eqref{eq:meromorphic_connections_intro} stays fixed. This is by definition an isomonodromic deformation of~$\nabla$, and is controlled by a system of nonolinear differential equations for the endomorphisms~$B$, $R_i$ as functions of $t_i$ and the spectral type of~$T$: the \textit{isomonodromy equations}.

A geometric interpretation can be given by introducing the complex vector space $\mathbb{M}$ paramet\-ri\-sing the linear maps $(B,R_i)_{i \in I'}$, as well as a space of (isomonodromy) times $\mathbf{B} \subseteq \mathbb{C}^I$ corresponding to the admissible deformations of the spectrum of $T$ and the configuration of simple poles. Then the isomonodromy equations are differential equations for local sections $\Gamma$ of the trivial fibration $\mathbb{M} \times \mathbf{B} \to \mathbf{B}$, and they admit a Hamiltonian interpretation: there exists a time-dependent Hamiltonian system $H \colon \mathbb{M} \times \mathbf{B} \to \mathbb{C}^I$ so that the isomonodromy equations along the deformation of~$t_i$ become
\begin{equation*}
 \frac{\partial \Gamma_k}{\partial t_i} = \{H_i,\Gamma_k\},
\end{equation*}
where $\Gamma_k$ and $H_i$ are components of $\Gamma$ and $H$ respectively, and $\{\cdot,\cdot\} = \omega^{-1}$ is the Poisson bracket on $\mathbb{M}$ corresponding to a complex symplectic structure $\omega \in \Omega^2(\mathbb{M},\mathbb{C})$. The same holds for the deformations along the spectral type of $T$.

Since the space $\mathbb{M}$ can be realised as the space of representation of a simply-laced graph, the system $H$ is called the \textit{simply-laced isomonodromy system}. The integral manifolds of $H$ define a flat symplectic nonlinear/Ehresmann connection in the symplectic fibration $\mathbb{M} \times \mathbf{B} \to \mathbf{B}$, whose leaves are precisely isomonodromic families of the meromorphic connections~\eqref{eq:meromorphic_connections_intro}: it is thus called the \textit{isomonodromy connection}, and constitutes an irregular version of the nonabelian Gau\ss--Manin connection~\cite{boalch_2001_symplectic_manifolds_and_isomonodromic_deformations}.

After taking symplectic quotients the isomonodromy system and the isomonodromy connection descend to a bundle of Nakajima quiver varieties over the space of admissible deformations~$\mathbf{B}$, whose fibre parametrise the isomorphism classes of the connections~\eqref{eq:meromorphic_connections_intro}.

Note that this symplectic interpretation of isomonodromic deformations holds for arbitrary polar divisors, surface genera and complex reductive gauge groups~\cite{boalch_2011_geometry_and_braiding_of_stokes_data_fission_and_wild_character_varieties,boalch_yamakawa_2015_twisted_wild_character_varieties}. The base $\mathbf{B}$ corresponds in general to space of admissible deformations of \textit{wild Riemann surface} structures, to be defined in Section~\ref{sec:moduli_spaces} for the case which is relevant to this article. Loosely speaking, such a structure prescribes both the positions of the poles and a local normal form for the principal part of the meromorphic connections at the poles, and reduces to the choice of the complex numbers $t_i$ and the semisimple endomorphisms $T^j$ in the case treated here.

Hence in brief the fibration $\mathbb{M} \times \mathbf{B} \to \mathbf{B}$ with its flat Hamiltonian system generalises after reduction to local systems $\widetilde{\mathcal{M}}^*_{\dR} \to \mathbf{B}$ of symplectic moduli spaces of meromorphic connections over spaces of admissible deformations of wild Riemann spheres. When a preferred global trivialisation of the local system of moduli spaces is given, then it makes sense to try to integrate the isomonodromy connection into a nonautonomous Hamiltonian system, as is the case of the simply-laced isomonodromy systems.

An example of simply-laced isomonodromy system is the Schlesinger system~\cite{schlesinger_1905_ueber_die_loesungen_gewisser_linearer_differentialgleichungen_als_funktionen_der_singularen_punkte}~-- when $A = X = 0$ in~\eqref{eq:meromorphic_connections_intro}~-- controlling the isomonodromic deformations of Fuchsian systems on the sphere. Moving beyond regular singularities, the Harnad duality~\cite{harnad_1994_dual_isomonodromic_deformations_and_moment_maps_to_loop_algebras} turns a logarithmic connection into a connection with a pole order two at infinity and a unique simple pole at zero~-- when $A = B = 0$ and $\lvert I' \rvert = 1$ in~\eqref{eq:meromorphic_connections_intro}. Then the simply-laced isomonodromy system is a dual version of the Schlesinger system, where deformations of the positions of the simple poles are turned into deformations of the spectrum of the irregular type at infinity, that is the spectrum of the rational function $T\,{\rm d}z$. Generalising this example to arbitrarily many simple poles yields the isomonodromy system of Jimbo--Miwa--M\^ori--Sato~\cite{jimbo_miwa_mori_sato_1980_density_matrix_of_an_impenetrable_bose_gas_and_the_fifth_painleve_transcendent}, in which the Harnad duality swaps over the two sets of isomonodromy times, and this article deals with a quantum extension of this involution.

Indeed more recently the quantisation of this rich geometric picture was pursued. The local system $\mathbb{M} \times \mathbf{B} \to \mathbf{B}$ of symplectic spaces dually corresponds to a bundle of commutative Poisson algebras $A_0 \times \mathbf{B} \to \mathbf{B}$, where $A_0 \cong \Sym(\mathbb{M}^*)$ is the ring of polynomial functions on the complex affine space $\mathbb{M}$ equipped with the symplectic Poisson bracket $\{\cdot,\cdot\}$. The simply-laced Hamiltonians $H_i$ define algebraic sections of this bundle, and one can ask to construct a deformation quantisation $\big(\widehat{A},\ast\big)$ of $(A_0,\{\cdot,\cdot\})$, together with algebraic sections $\widehat{H}_i$ of the bundle of noncommutative $\mathbb{C} \llbracket \hslash \rrbracket$-algebras $\widehat{A} \times \mathbf{B} \to \mathbf{B}$ projecting back to the simply-laced Hamiltonians when composed with the semiclassical limit $\widehat{A} \to A_0$.

This can indeed be done~\cite{rembado_2019_simply_laced_quantum_connections_generalising_kz}, and moreover in such a way that the connection
\begin{equation}\label{eq:slqc_intro}
 \widehat{\nabla} = {\rm d} - \widehat{\varpi}, \qquad \text{with} \quad \widehat{\varpi} \coloneqq \sum_{i \in I} \widehat{H}_i \, {\rm d}t_i,
\end{equation}
is strongly flat, where the $\widehat{H}_i$ acts on the fibre $\widehat{A}$ via its commutator. The connection~\eqref{eq:slqc_intro} is the \textit{universal simply-laced quantum connection}, a strongly integrable deformation quantisation of the simply-laced isomonodromy system, and the equations for a local horizontal section of the bundle $\widehat{A} \times \mathbf{B}$ are the corresponding \textit{quantum isomonodromy equations}.

Hence in brief one replaces a local system of symplectic manifolds with a flat vector bundle, constructed via fibrewise deformation quantisation of the rings of functions. In this viewpoint~\eqref{eq:slqc_intro} is a genus zero irregular analogue of the Hitchin connection~\cite{axelrod_dellapietra_witten_1991_geometric_quantisation_of_chern_simons_gauge_theory,hitchin_1990_flat_connections_and_geometric_quantisation} in deformation quantisation, where the base space of the quantum bundle is extended from the ordinary Riemann moduli space~-- which would be trivial in genus zero~-- to a space of admissible deformations of wild Riemann spheres, and where importantly one quantises moduli spaces of meromorphic connections instead of nonsingular ones. For the same reasons, it is also an irregular analogue of the connection of Witten~\cite{witten_1991_quantization_of_chern_simons_gauge_theory_with_complex_gauge_group}, where meromorphic connections replace holomorphic connections/Higgs fields.

Instead of trying to construct other flat quantum connections out of the deformation quantisation of isomonodromy systems, this article further inspects the symmetries of the simply-laced isomonodromy systems, which involve a higher viewpoint on the aforementioned Harnad duality~\cite{harnad_1994_dual_isomonodromic_deformations_and_moment_maps_to_loop_algebras} (see Example~\ref{example:harnad_duality}).
To introduce it let $W = \bigoplus_{i \in I'} V_i$ be an $I'$-graded vector space and $V \coloneqq W \oplus U$, where $U$ is as in the beginning. Then matrices of differential operators $M \in \End(V) \otimes \mathbb{C} [\partial_z,z]$ yields differential equations for functions $v \colon \mathbb{C} \to V$, which can in some cases be read as meromorphic connections. For example if
\begin{equation*}
 M \coloneqq \begin{pmatrix}
		z - T & P \\
		Q & \partial_z \end{pmatrix}, \qquad \text{with} \quad Q \colon W \leftrightarrows U  \ {\colon}\!P \quad \text{and} \quad T \in \End(W),
\end{equation*}
then the equation $Mv = 0$ expresses the flatness of the $U$-component of $v$ as a local section of $U \times \mathbb{C}P^1 \to \mathbb{C}P^1$, with respect to the logarithmic connection $\nabla_U = {\rm d} - Q(z - T)^{-1}P\,{\rm d}z$. This still corresponds to the case $A = X = 0$ of~\eqref{eq:meromorphic_connections_intro}, with residues $R_i = Q_iP_i$ given by the graded components $Q_i \colon V_i \leftrightarrows U \ {\colon}\! P_i$, and where $\{t_i\}_{i \in I'}$ is the spectrum of~$T$.

The Fourier--Laplace transform $\partial_z \mapsto -z$, $z \mapsto \partial_z$ turns $M$ into a new matrix $M'$, and $M'v = 0$ now expresses the flatness of the $W$-component of $v$ as a local section of $W \times \mathbb{C}P^1 \to \mathbb{C}P^1$, with respect to the connection $\nabla_W = {\rm d} - \big(T - \frac{PQ}{z}\big){\rm d}z$. This corresponds to the case $A = B = 0$ and $\lvert I' \rvert = 1$ of~\eqref{eq:meromorphic_connections_intro}, and thus the Fourier--Laplace transform recovers the Harnad duality that swaps over the rational differential operators $\nabla_U$ and $\nabla_W$~-- and relates the isomonodromic deformations of connections living of different bundles.

This viewpoint makes transparent that more general transformations are allowed: the whole of $\SL_2(\mathbb{C})$ acts on the Weyl algebra $\mathbb{C} [\partial_z,z]$ by a symplectic reparametrisation of $\mathbb{C}^2 \cong T^*\mathbb{C}$, and it turns out that the reduced simply-laced isomonodromy systems are invariant under the full action~\cite{boalch_2012_simply_laced_isomonodromy_systems}. Hence the Poisson structures on the bundles of classical algebras are split in orbits for the $\SL_2(\mathbb{C})$-action. Further, the isomorphisms along one orbit become flat when the classical Hamiltonian reduction of both the algebra~$A_0$ and the isomonodromy connection is taken. We thus consider the quantisation of this statement, and ask:
\begin{center}
{\it can one quantise the classical $\SL_2(\mathbb{C})$-action to construct isomorphisms\\ of $($quantum$)$ algebras at different choices of $\ast$-products,\\ so that the reduced quantum isomonodromy equations are invariant?}
\end{center}

This article shows that this is indeed possible, proving the following.

\begin{theorem*}[Section~\ref{sec:quantisation_of_action}]
There exists a natural quantisation of the $\SL_2(\mathbb{C})$-symmetries of the commutative Poisson algebras $(A_0,\{\cdot,\cdot\})$, resulting in $\SL_2(\mathbb{C})$-symmetries of the topologically free $\mathbb{C}\llbracket \hslash \rrbracket$-algebras $\big(\widehat{A},\ast\big)$.
\end{theorem*}

\begin{theorem*}[Sections~\ref{sec:quantum_reduction_quiver_varieties} and~\ref{sec:quantum_action_reduction_compatibility}]
For any choice of coadjoint orbit the quantum Hamiltonian reductions $R\big(\widehat{A},\ast\big)$ of the algebras $\big(\widehat{A},\ast\big)$ can be defined in such a way that both the simply-laced quantum connection~\eqref{eq:slqc_intro} and the quantum $\SL_2(\mathbb{C})$-action can be reduced.
\end{theorem*}

This theorem is obtained as a particular case of a more general construction, namely the quantum Hamiltonian reduction of symplectic varieties of representations of simply-laced quivers. This construction is carried out in Section~\ref{sec:quantum_reduction_quiver_varieties}, and should in
principle follow from~\cite{BezrukavnikovLoseu}.

\begin{theorem*}[Section~\ref{sec:quantum_invariance}]
At sufficiently generic coadjoint orbit the reduced quantum $\SL_2(\mathbb{C})$-action yields flat isomorphisms of the bundles $R\big(\widehat{A},\ast\big) \times \mathbf{B}$, equipped with the reduced $($universal$)$ simply-laced quantum connection.
\end{theorem*}

These results generalise the use of the quantum/Howe duality~\cite{baumann_1999_the_q_weyl_group_of_a_q_schur_algebra} to relate the KZ and the Casimir connection for $\mathfrak{gl}_n(\mathbb{C})$, as was considered in~\cite{toledanolaredo_2002_a_kohno_drinfeld_theorem_for_quantum_weyl_groups}. Further, since the quantum Hamiltonian reduction of Section~\ref{sec:quantum_reduction_quiver_varieties} applies to any simply-laced quiver it generalises the reduction of the simply-laced quantum connections constructed in~\cite{rembado_2019_simply_laced_quantum_connections_generalising_kz}~-- where only the cases relevant to the KZ connection, the Casimir connection and the dynamical connection were worked out.

\section*{Layout of the article}

In Section~\ref{sec:classical_symmetries} we introduce the action of the group $\SL_2(\mathbb{C})$ on a vector space $\mathbb{M}$ of presentations of modules for the one-dimensional Weyl algebra. Then in~\ref{sec:symplectic_classical_action} we introduce a family of symplectic structures $\omega_{\bm{a}}$ on $\mathbb{M}$ parametrised by embeddings $\bm{a} \colon J \hookrightarrow \mathbb{P}\big(\mathbb{C}^2\big)$ of a finite set $J$ into the complex projective line, and recall how the $\SL_2(\mathbb{C})$-action yields simplectomorphisms compatible with the action on the embeddings. This finally yields a Poisson action on the bundle of commutative Poisson algebras $A_0 =\mathbb{C}[\mathbb{M}] \cong \Sym(\mathbb{M}^*)$ of functions on the symplectic vector space $(\mathbb{M},\omega_{\bm{a}})$~-- over the base space of choices of symplectic structure/Poisson bracket.

In Sections~\ref{sec:filtered_quantisation} and~\ref{sec:rees_construction} we construct a topologically free $\mathbb{C}\llbracket\hslash\rrbracket$-algebra $\widehat{A}$, a formal deformation quantisation of the commutative Poisson algebra $A_0$ coming from a filtered quantisation via the Rees construction; the filtered quantisation itself is the noncommutative Weyl algebra $(A,\ast_{\bm{a}})$ of the Poisson space $\big(\mathbb{M}^*,\omega_{\bm{a}}^{-1}\big)$. Then in Section~\ref{sec:quantisation_of_action} we construct a quantisation of the $\SL_2(\mathbb{C})$-action on the bundle of deformation algebras $\big(\widehat{A},\ast_{\bm{a}}\big)$~-- over the base space of choices of symplectic structure/Poisson bracket.

In Section~\ref{sec:slims} we recall the definition of the simply-laced isomonodromy systems, which are strongly integrable nonautonomous Hamiltonian systems $H_i \colon \mathbb{M} \times \mathbf{B} \to \mathbb{C}^I$. In Section~\ref{sec:moduli_spaces} we explain how to reduce this system to obtain the isomonodromy connection inside a symplectic bundle $\widetilde{\mathcal{M}}^*_{\dR} \to \mathbf{B}$ of moduli spaces of meromorphic connections on the sphere, defined over a~space of variations of wild Riemann surfaces of genus zero.

In Section~\ref{sec:classical_hamiltonian_reduction} we recall the definition of the classical Hamiltonian reduction of a commutative Poisson algebra with respect to a Lie algebra action, and define both the reduction of the simply-laced isomonodromy systems and of the (classical) $\SL_2(\mathbb{C})$-action. For this we introduce the viewpoint of representation spaces for a particular class of simply-laced quivers. Then in Section~\ref{sec:classical_invariance} we recall that the reduced simply-laced Hamiltonians are shifted by constants under the reduced action, and interpret this result as the existence of a family of flat isomorphism of bundles~-- equipped with (flat) reduced isomonodromy connections.

In Section~\ref{sec:slqc} we recall the definition of the universal simply-laced quantum connection and the simply-laced quantum connection (the latter is obtained from the former by specifying a~level for the deformation parameter $\hslash$). In Section~\ref{sec:action_quantisation_compatibility} we show that the quantisation of the simply-laced isomonodromy systems is compatible with the quantum $\SL_2(\mathbb{C})$-action constructed in Section~\ref{sec:quantum_action}.

In Sections~\ref{sec:quantum_hamiltonian_reduction_filtered} and~\ref{sec:quantum_hamiltonian_reduction_formal} we recall the definition of quantum Hamiltonian reduction both in filtered and deformation quantisation. In Section~\ref{sec:quantum_reduction_quiver_varieties} we construct the quantum Hamiltonian reduction of the symplectic variety of representations of a simply-laced quiver, and in Section~\ref{sec:quantum_action_reduction_compatibility} we show that the quantum action of Section~\ref{sec:quantisation_of_action} is compatible with the quantum Hamiltonian reduction, thereby defining the reduced quantum $\SL_2(\mathbb{C})$-action.

Finally, in Section~\ref{sec:quantum_invariance} we put together all the previous results to prove the last theorem: the reduced simply-laced quantum connection is invariant under the reduced quantum $\SL_2(\mathbb{C})$-action, at sufficiently generic choices of coadjoint orbit. We interpret this result as the existence of a family of projectively flat isomorphisms of bundles~-- equipped with (flat) reduced simply-laced quantum connections.

All vector spaces and associative/Lie algebras are tacitly defined over $\mathbb{C}$; all associative algebras are unitary and finitely generated. All gradings and filtrations of algebras are over $\mathbb{Z}_{\geq 0}$ and all filtrations are exhaustive.

\section{Classical symmetries}\label{sec:classical_symmetries}

The moduli space of isomorphism classes of meromorphic connections~\eqref{eq:meromorphic_connections_intro} can be realised as the complex symplectic quotient of a vector space parametrising presentations of modules for the one-dimensional Weyl algebra. This is the higher viewpoint that makes the global $\SL_2(\mathbb{C})$-symmetries explicit.

\subsection{Modules for the Weyl algebra}
\label{sec:weyl_modules}

Let $V$ be a finite-dimensional vector space, and choose endomorphisms $\alpha, \beta, \gamma \in \End(V)$ such that $\alpha$ and $\beta$ are simultaneously diagonalisable with kernels intersecting only at the origin. Set then $\partial = \frac{\partial}{\partial z}$, where $z$ is the standard holomorphic coordinate on the complex plane, and consider the differential operator
\begin{equation*}
 M = M(\alpha,\beta,\gamma) \coloneqq \alpha \partial + \beta z - \gamma \in \End(V) \otimes A_1,
\end{equation*}
where $A_1$ is the one-dimensional Weyl algebra, that is the quotient of the free algebra on the set $\{\partial,z\}$ modulo the canonical commutation relations:
\begin{equation}\label{eq:weyl_algebra_dimension_one}
 A_1 \coloneqq \mathbb{C} \langle \partial,z \rangle \big\slash (\partial z - z \partial = 1).
\end{equation}
The space of solutions of the system $Mv = 0$, where $v$ is a local $V$-valued holomorphic function on the complex plane, is naturally related to the quotient (left) module
\begin{equation}\label{eq:weyl_modules}
 \mathcal{N} = \mathcal{N}(\alpha,\beta,\gamma) \coloneqq V \otimes A_1 \big\slash (V \otimes A_1 \cdot M).
\end{equation}
Hence the endomorphisms $\alpha$, $\beta$, $\gamma$ parametrise presentations for the $A_1$-modules~\eqref{eq:weyl_modules}, and the description can be simplified after a diagonalisation and a normalisation.

To this end choose a basis of simultaneous eigenvectors for $\alpha$ and $\beta$, so that $V \cong \mathbb{C}^n$ for some positive integer $n$ and $\alpha = \diag(\alpha_1,\dotsc,\alpha_n)$, $\beta = \diag(\beta_1,\dotsc,\beta_n)$ for some complex numbers~$\alpha_i$,~$\beta_i$. Since $\Ker(\alpha) \cap \Ker(\beta) = (0)$ the points $a_j \coloneqq [-\beta_j : \alpha_j]$ are well defined in the complex projective line $\mathbb{P}\big(\mathbb{C}^2\big)$, which is identified to $\mathbb{C} \cup \{\infty\}$ by letting $\mathbb{C}$ be the image of the standard affine chart sending $[0 : 1]$ to the origin:
\begin{equation*}
[-\beta_j : \alpha_j ] \longmapsto \begin{cases}
 \infty, & \alpha_j = 0, \\
 -\dfrac{\beta_j}{\alpha_j}, & \text{else}.
 \end{cases}
\end{equation*}
Denote $J$ the index set for the points $\bm{a} \coloneqq \{a_j\}_j \subseteq \mathbb{C} \cup \{\infty\}$, which we equivalently think of as an embedding $\bm{a} \colon J \hookrightarrow \mathbb{C} \cup \infty$. Now define $W^j \subseteq V$ as the joint eigenspace for $\alpha$ and $\beta$ such that the corresponding eigenvalues map to~$a_j$:
\begin{equation*}
 W^j \coloneqq \big\{v \in V \,|\, \alpha(v) = \lambda v, \, \beta(v) = \mu v, \, [-\mu : \lambda] = a_j\big\}.
\end{equation*}
There is thus a $J$-grading $V = \bigoplus_{j \in J} W^j$, and assume further that the diagonal part of $\gamma$ in the decomposition
\begin{equation*}
 \End(V) = \bigoplus_{i \neq j \in J} \Hom\big(W^i,W^j\big) \oplus \bigoplus_{j \in J} \End\big(W^j\big)
\end{equation*}
be diagonalisable.

Now one can normalise $M$ by left multiplication with a (constant) diagonal $n$-by-$n$ matrix $N$ so that the following two distinct cases arise according to whether $\infty \in a(J)$ or not.

\begin{example*}[degenerate normal form]If the infinity does belong to the image of embedding we denote $\infty \in J$ the element $\bm{a}^{-1}(\infty)$, so that $W^{\infty} = \Ker(\alpha) \subseteq V$. Then decompose $V = W^{\infty} \oplus U^{\infty}$, with $U^{\infty} \coloneqq \bigoplus_{j \neq \infty} W^j$, and accordingly $\alpha = \left(\begin{smallmatrix}
			 0 & \\
			 & \alpha_{II} \end{smallmatrix}\right)$, $\beta = \left(\begin{smallmatrix}
						 \beta_I & \\
						 & \beta_{II}
						 \end{smallmatrix}\right)$~-- where an empty nondiagonal entry signals vanishing coefficients. If one sets $N = \left(\begin{smallmatrix}
								 \beta_I^{-1} & \\
								 & \alpha_{II}^{-1}
								 \end{smallmatrix}\right)$ then						
\begin{equation}\label{eq:normal_degenerate_form}
 NM = \begin{pmatrix}
 0 & \\
 & \Id_{U^{\infty}}
 \end{pmatrix}\partial + \begin{pmatrix}
				\Id_{W^{\infty}} & \\
				& - A
			 \end{pmatrix}z - \begin{pmatrix}
						T^{\infty} & P \\
						Q & B + T
					 \end{pmatrix} \in \End(V) \otimes A_1,
\end{equation}
where $A = \bigoplus_{j \neq \infty} a_j\Id_{W^j}$ is the endomorphism acting on $W^j$ via the scalar $a_j$, the diagonalisable endomorphisms $T^{\infty} \in \End(W^{\infty})$ and $T \in \End(U^{\infty})$ constitute the diagonal part of $\gamma$, and $\left(\begin{smallmatrix}
	 0 & P \\
	 Q & B
	 \end{smallmatrix}\right)$ the off-diagonal part of $\gamma$. Hence $P \in \Hom(U^{\infty},W^{\infty})$, $Q \in \Hom(W^{\infty},U^{\infty})$ and
	 \begin{equation*}
B = \bigoplus_{i \neq j \in J \setminus \{\infty\}} B^{ij}, \qquad \text{with} \quad B^{ij} \in \Hom\big(W^j,W^i\big).	 \end{equation*}
The components $P^j \colon W^j \to W^{\infty}$ and $Q^j \colon W^{\infty} \to W^j$ could be written $P^j = B^{\infty j}$, $Q^j = P^{j \infty}$ for the sake of a uniform notation. The customary notation for positions and momenta variables is instead used with a view towards the symplectic pairing~\eqref{eq:symplectic_form}.
\end{example*}

\begin{example*}[generic normal form]If instead $\infty \not\in \bm{a}(J)$ then put $N = \alpha^{-1}$ to find
\begin{equation}\label{eq:normal_generic_form}
 NM = \Id_V \partial - A z - \big(B + T\big) \in \End(V) \otimes A_1,
\end{equation}
with $A$, $B$, $T$ defined as above.
\end{example*}

Thus in brief the elements $A, \gamma \in \End(V)$ parametrise the normal forms~\eqref{eq:normal_degenerate_form} and~\eqref{eq:normal_generic_form} for the presentation of the Weyl modules~\eqref{eq:weyl_modules}; more precisely~\eqref{eq:normal_degenerate_form} is a \textit{degenerate normal form} and~\eqref{eq:normal_generic_form} a \textit{generic normal form}. The choice of the endomorphism $A$ will be encoded in a (linear) complex symplectic form $\omega_{\bm{a}} \in \Omega^2(\mathbb{M},\mathbb{C})$ on the vector space $\mathbb{M} \coloneqq \bigoplus_{i \neq j \in J} \Hom\big(W^i,W^j\big)$ in Section~\ref{sec:symplectic_classical_action}, whereas the spectral types of $T^{\infty}$, $T$ will become the regular and irregular times of the simply-laced isomonodromy systems, respectively, in Section~\ref{sec:slims}.

\subsection{Action on presentations for Weyl algebra modules}\label{sec:action_on_weyl_modules}

An element $g = \left(\begin{smallmatrix}
 a & b \\
 c & d
 \end{smallmatrix}\right) \in \SL_2(\mathbb{C})$ acts on the Weyl algebra $A_1$ on the left by transforming the generators via
\begin{equation*}
 g.\begin{pmatrix}
		\partial \\
		z
		\end{pmatrix} \coloneqq \begin{pmatrix}
					 a\partial + b z \\
					 c\partial + d z
					\end{pmatrix}.
\end{equation*}
This turns $M = \alpha \partial + \beta z - \gamma$ into
\begin{equation}\label{eq:action_on_matrices}
 g.M = (a \alpha + c\beta)\partial + (b\alpha + d\beta)z - \gamma,
\end{equation}
also by a left action.

It follows from~\eqref{eq:action_on_matrices} that the point $a_j$ transforms as
\begin{equation}\label{eq:action_on_embedding}
 a_j.g = [-\beta_j : \alpha_j].g = [-(b \alpha_j + d\beta_j) : a\alpha_j + c\beta_j],
\end{equation}
that is via the inverse of the standard action of $(\text{P})\SL_2(\mathbb{C})$ as the group of automorphisms of the Riemann sphere~-- since $g^{-1} = \left(\begin{smallmatrix}
 d & -b \\
 -c & a
\end{smallmatrix}\right)$. In particular $\infty.g = [-d : c]$, and thus $\infty$ is fixed if $c = 0$. Conversely, one has $a_j.g = \infty$ for $a_j \neq \infty$ if and only if $a = ca_j$.

Hence $\SL_2(\mathbb{C})$ acts on the space of embeddings $\bm{a} \colon J \hookrightarrow \mathbb{C} \cup \{\infty\}$ on the right (but not on the set $J$). There is then an induced action on $\gamma$, defined from the change of normal forms. Namely, one passes from $NM \in \End(V) \otimes A_1$ to $N'(g.M) \in \End(V) \otimes A_1$, where $N'$ is a suitable (block diagonal) matrix putting $g.M$ in one of the normal forms~\eqref{eq:normal_degenerate_form},~\eqref{eq:normal_generic_form} by left multiplication. The overall action on $\gamma$ is thus given by the left multiplication by the diagonal matrix $E \coloneqq N'N^{-1}$. We will have $E = \bigoplus_{i \in I} \eta_i \Id_{W^i}$, with the numbers $\eta_i \in \mathbb{C}^{\times}$ depending on $g$ and $\bm{a}$, so that the multiplication $g.\gamma = E\gamma$ yields the transformation
\begin{equation}\label{eq:action_on_fibration}
 g \colon \ B^{ij} \longmapsto \eta_i B^{ij}, \qquad T^i \longmapsto \eta_i T^i, \qquad \text{for} \quad i \neq j \in J.
\end{equation}
This restricts to the linear map $\varphi(\bm{a}) \colon \mathbb{M} \to \mathbb{M}$ given by
\begin{equation*}
 \varphi(\bm{a}) = \bigoplus_{i \in J} \bigoplus_{j \neq i} \eta_i I_{ij} \in \GL(\mathbb{M}),
\end{equation*}
where $I_{ij} \coloneqq \Id_{\Hom(W^j,W^i)}$. Thus the pull-back $\varphi^*_g(\bm{a}) \coloneqq  (\varphi(\bm{a}) )^*$ sends linear functions to linear functions: it is the diagonalisable endomorphism of $\mathbb{M}^*$ which multiplies linear functions on $\Hom(W^j,W^i) \subseteq \mathbb{M}$ by $\eta_i$, i.e.,
\begin{equation}\label{eq:classical_action_decomposition}
 \left.\varphi^*_g(\bm{a})\right\vert_{\mathbb{M}^*} = \bigoplus_{i \in J} \bigoplus_{j \neq i} \eta_i I^*_{ij} \in \GL(\mathbb{M}^*),
\end{equation}
where $I^*_{ij} \coloneqq \Id_{\Hom(W^j,W^i)^*}$.

We will now compute the numbers $\eta_i$ with a case-by-case discussion.

\begin{example*}[switching between generic forms]
Assume first that $\alpha$ is nonsingular and that \mbox{$a \neq ca_j$} for all $j \in J$. Then $\{a_j\}_j$ and $\{a_j.g\}_j$ are both contained in $\mathbb{C}$, and the matrices~$M$ and~$g.M$ can both be put in generic normal forms~\eqref{eq:normal_generic_form} by suitable matrices~$N$,~$N'$. As explained above $N = \alpha^{-1}$, and then $N' = (a\alpha + c\beta)^{-1}$ by looking at~\eqref{eq:action_on_matrices}. Hence the action of~$g$ multiplies~$\gamma$ on the left by
\begin{align*}
 E &= (a\alpha + c\beta)^{-1}\alpha = \diag\left(\frac{\alpha_1}{a\alpha_1 + c\beta_1},\dotsc,\frac{\alpha_n}{a\alpha_n + c\beta_n}\right) \\
 &= \diag\left(\left(\frac{a\alpha_1 + c\beta_1}{\alpha_1}\right)^{-1},\dotsc,\left(\frac{a\alpha_n + c\beta_n}{\alpha_n}\right)^{-1}\right) = \bigoplus_{i \in J} (a - ca_i)^{-1}\Id_{W^i}.
\end{align*}
So $\eta_i = (a - ca_i)^{-1} \in \mathbb{C}^{\times}$ in this case.
\end{example*}

\begin{example*}[from a generic form to a degenerate one]
Assume now $\alpha$ to be nonsingular, but that $a = ca_j$ for some (unique) $j \in J$. Then $\{a_j\}_j$ is contained in $\mathbb{C}$, and $a_j.g = \infty$. Decompose $V = W^j \oplus U^j$, with $U^j = \bigoplus_{i \neq j} W^i$, and accordingly $\alpha = \left(\begin{smallmatrix}
				 \alpha_I & \\
				 & \alpha_{II}
				 \end{smallmatrix}\right)$, $\beta = \left(\begin{smallmatrix}
							 \beta_I & \\
							 & \beta_{II}
							 \end{smallmatrix}\right)$.
Then after the action the space $W^j$ becomes $W^{\infty}$, $U^j$ becomes $U^{\infty}$ and $\alpha$, $\beta$ become
\begin{equation}\label{eq:generic_to_degenerate}
 \alpha' = \begin{pmatrix}
 0 & \\
 & a\alpha_{II} + c\beta_{II}
 \end{pmatrix}, \qquad \beta' = \begin{pmatrix}
					 b\alpha_I + d\beta_I \\
					 & b\alpha_{II} + d\beta_{II}
					 \end{pmatrix}.
\end{equation}
Thus the normal forms are
\begin{equation*}
 NM = \Id_V \partial - A z - (B + T),
\end{equation*}
before the action, and
\begin{equation*}
 N'(g.M) = \begin{pmatrix}
 0 & \\
 & \Id_{U^{\infty}}
 \end{pmatrix} \partial + \begin{pmatrix}
				 \Id_{W^{\infty}} & \\
				 & -A'
				 \end{pmatrix} z - \begin{pmatrix}
							T^{\infty} & P \\
							Q & B' + T'
						 \end{pmatrix}
\end{equation*}
after the action, where $A' = \bigoplus_{i \neq j} (a.g_i) \Id_{W^i}$.

Here again $N = \alpha^{-1}$, but $N' = \left(\begin{smallmatrix}
 (b\alpha_I + d\beta_I)^{-1} & \\
 & (a\alpha_{II} + c\beta_{II})^{-1}
 \end{smallmatrix}\right)$, by looking at~\eqref{eq:generic_to_degenerate}. Hence on the whole
\begin{equation*}
 E = \begin{pmatrix}
 (b\alpha_I + d\beta_I)^{-1}\alpha_I & \\
 & (a\alpha_{II} + c\beta_{II})^{-1}\alpha_{II}
 \end{pmatrix},
\end{equation*}
and one finds
\begin{equation*}
 \eta_i = \begin{cases}
 (b - da_j)^{-1}, & i = j, \\
 (a - ca_i)^{-1}, & i \neq j.
 \end{cases}
\end{equation*}
Note that $b - da_j \neq 0$, since $a = ca_j$ and $b = da_j$ together imply $ad - bc = 0$.
\end{example*}
	
\begin{example*}[from a degenerate form to a generic one]
Assume here $\infty \in \bm{a}(J)$, $c \neq 0$ and $a \neq ca_j$ for all $j \neq \infty$. Then $\infty \not\in \bm{a}.g(J)$, and one can decompose again $V = W^{\infty} \oplus U^{\infty}$, $\alpha = \left(\begin{smallmatrix}
					 0 & \\
					 & \alpha_{II}
					 \end{smallmatrix}\right)$, $\beta = \left(\begin{smallmatrix}
							 \beta_I & \\
							 & \beta_{II}
							 \end{smallmatrix}\right)$. After the action the infinity is mapped to $-\frac{d}{c} \in \mathbb{C}$, and $\alpha$, $\beta$ transform to
\begin{equation*}
 \alpha' = \begin{pmatrix}
 c\beta_I & \\
 & a\alpha_{II} + c\beta_{II}
 \end{pmatrix}, \qquad \beta' = \begin{pmatrix}
					 d\beta_I & \\
					 & b\alpha_{II} + d\beta_{II}
					 \end{pmatrix}.
\end{equation*}
The normalisation matrices are thus
\begin{equation*}
 N = \begin{pmatrix}
 \beta_I^{-1} & \\
 & \alpha_{II}^{-1}
 \end{pmatrix}, \qquad N' = \begin{pmatrix}
				 (c\beta_I)^{-1} & \\
				 & (a\alpha_{II} + c\beta_{II})^{-1}
				\end{pmatrix},
\end{equation*}
so that
\begin{equation*}
 E = \begin{pmatrix}
 c^{-1}\Id_{W^{\infty}} & \\
 & (a\alpha_{II} + c\beta_{II})^{-1}\alpha_{II} \end{pmatrix}.
\end{equation*}
Then one has
\begin{equation*}
 \eta_i = \begin{cases}
 c^{-1}, &i = \infty, \\
 (a - ca_i)^{-1}, &i \neq \infty.
 \end{cases}
\end{equation*}
\end{example*}

\begin{example*}[switching between degenerate forms]
Assume again $\infty \in \bm{a}(J)$, and consider the following two exhaustive subcases:
\begin{enumerate}\itemsep=0pt
 \item[1)] $c = 0$ (whence $ad = 1$),
 \item[2)] $c \neq 0$ and $a = ca_j$ for some $j \in \infty$.
\end{enumerate}

Suppose first $c = 0$, and decompose again $V = W^{\infty} \oplus U^{\infty}$, $\alpha = \left(\begin{smallmatrix}
					 0 & \\
					 & \alpha_{II}
					 \end{smallmatrix}\right)$, $\beta = \left(\begin{smallmatrix}
							 \beta_I & \\
							 & \beta_{II}
							 \end{smallmatrix}\right)$. Then~$\alpha$,~$\beta$ become
\begin{equation*}
 \alpha' = \begin{pmatrix}
 0 & \\
 & a\alpha_{II}
 \end{pmatrix}, \qquad \beta' = \begin{pmatrix}
					 d\beta_I & \\
					 & b\alpha_{II} + d\beta_{II}
					 \end{pmatrix},
\end{equation*}
whence
\begin{equation*}
 N = \begin{pmatrix}
 \beta_I^{-1} & \\
 & \alpha_{II}^{-1}
 \end{pmatrix}, \qquad N' = \begin{pmatrix}
				(d\beta_I)^{-1} & \\
				& (a\alpha_{II})^{-1}
				\end{pmatrix}.
\end{equation*}
This yields the simple
\begin{equation*}
 E = \begin{pmatrix}
 a\Id_{W^{\infty}} & \\
 & d\Id_{U^{\infty}}
 \end{pmatrix},
\end{equation*}
using $a^{-1} = d$; thus in this case
\begin{equation*}
 \eta_i = \begin{cases}
 a, & i = \infty, \\
 d, & i \neq \infty.
 \end{cases}
\end{equation*}

Suppose finally $c \neq 0$ and $a = ca_j$ for some $j \in J$. Consider then the finer decomposition $V = W^{\infty} \oplus W^j \oplus U^{\infty,j}$, where $U^{\infty,j} \coloneqq \bigoplus_{i \in J \setminus \{\infty,j\}} W^i$, and accordingly
\begin{equation*}
 \alpha = \begin{pmatrix}
 0 & & \\
 & \alpha_{II} & \\
 & & \alpha_{III}
 \end{pmatrix}, \qquad \beta = \begin{pmatrix}
					 \beta_I & & \\
					 & \beta_{II} & \\
					 & & \beta_{III}
					\end{pmatrix}.
\end{equation*}
After transforming, the roles of $W^{\infty}$ and $W^j$ are swapped; moreover
\begin{equation*}
 \alpha' = \begin{pmatrix}
 c\beta_I & & \\
 & 0 & \\
 & & a\alpha_{III} + c\beta_{III}
 \end{pmatrix}, \qquad \beta' = \begin{pmatrix}
					 d\beta_I & & \\
					 & b\alpha_{II} + d\beta_{II} & \\
					 & & b\alpha_{III} + d\beta_{III}
					 \end{pmatrix},
\end{equation*}
since indeed $W^j = \Ker(\alpha')$. The normalisation matrices are
\begin{equation*}
 N = \begin{pmatrix}
 \beta_I^{-1} & & \\
 & \alpha_{II}^{-1} & \\
 & & \alpha_{III}^{-1}
 \end{pmatrix}, \qquad N' = \begin{pmatrix}
				 (c\beta_I)^{-1} & & \\
				 & (b\alpha_{II} + d\beta_{II})^{-1} & \\
				 & & (a\alpha_{III} + c\beta_{III})^{-1}
				\end{pmatrix},
\end{equation*}
whence
\begin{equation*}
 E = \begin{pmatrix}
 c^{-1} & & \\
 & (b\alpha_{II} + d\beta_{II})^{-1}\alpha_{II} & \\
 & & (a\alpha_{III} + c\beta_{III})^{-1}\alpha_{III}
 \end{pmatrix}.
\end{equation*}

Therefore
\begin{equation*}
 \eta_i = \begin{cases}
 c^{-1}, & i = \infty, \\
 (b - da_j)^{-1}, & i = j, \\
 (a - ca_i)^{-1}, & i \not\in \{\infty,j\}.
 \end{cases}
\end{equation*}
\end{example*}

Denoting $\mathbf{A} \coloneqq \big\{\{a_j\}_{j \in J} \,|\, a_j \neq a_k \text{ for } j \neq k \big\} \subseteq (\mathbb{C} \cup \{\infty\})^J$ the space of possible embeddings $\bm{a} \colon J \hookrightarrow \mathbb{C} \cup \{\infty\}$, the matrices $E$ defined by~\eqref{eq:action_on_fibration} assemble into a function
\begin{equation*}
 E \colon \ \SL_2(\mathbb{C}) \times \mathbf{A} \longrightarrow \GL(V),
\end{equation*}
satisfying the cocycle identities
\begin{equation*}
 E(gg',\bm{a}) = E(g',\bm{a}.g) \cdot E(g,\bm{a}), \qquad \text{for} \quad g, g' \in \SL_2(\mathbb{C}), \  \bm{a} \in \mathbf{A},
\end{equation*}
where $\bm{a}.g$ is the action~\eqref{eq:action_on_embedding} on the embedding. Similarly the numbers $\eta_i$ of~\eqref{eq:action_on_fibration} become functions $\eta_i \colon \SL_2(\mathbb{C}) \times \mathbf{A} \to \mathbb{C}^{\times}$ which satisfy
\begin{equation}\label{eq:action_cocycle}
 \eta_i(gg',\bm{a}) = \eta_i(g',\bm{a}.g) \cdot \eta_i(g,\bm{a}), \qquad \text{for} \quad i \in J.
\end{equation}

\begin{example*}[generic case]
If $\infty \not\in \bm{a}(J) \cup \bm{a}.g(J)$ and if one writes $g' = \left(\begin{smallmatrix}
 a' & b' \\
 c' & d'
 \end{smallmatrix}\right) \in \SL_2(\mathbb{C})$ then the first case of~\eqref{eq:action_on_fibration} yields
\begin{equation*}
 \eta_i(gg',\bm{a}) = (aa' + bc') - (ca' + dc')a_i, \qquad \eta_i(g,\bm{a}) \cdot \eta_i(g',\bm{a}.g) = (c - c a_i) \cdot (a' - c'a_i.g),
\end{equation*}
and in this generic case~\eqref{eq:action_cocycle} follows from
\begin{equation*}
 a_i.g = -\frac{b - da_i}{a - c a_i} = -\frac{b\alpha_i + d\beta_i}{a\alpha_i + c\beta_i},
\end{equation*}
using $a_i = -\frac{\beta_i}{\alpha_i} \in \mathbb{C}$.
\end{example*}

\begin{example*}[Harnad duality]\label{example:harnad_duality}

It should helpful to the reader to work out the particular case of the $\SL_2(\mathbb{C})$-action which corresponds to the Harnad duality for rational differential operators~\cite{harnad_1994_dual_isomonodromic_deformations_and_moment_maps_to_loop_algebras}, as mentioned in the Introduction.

Consider then the particular case of a degenerate form~\eqref{eq:normal_degenerate_form} with $A = B = T = 0$, that is $\bm{a}(J) = \{\infty,0\}$ and
\begin{equation*}
M = \begin{pmatrix}
 0 & \\
 & \Id_U
 \end{pmatrix} \partial + \begin{pmatrix}
 \Id_W & \\
 & 0
 \end{pmatrix} z - \begin{pmatrix}
 T^{\infty} & P \\
 Q & 0
 \end{pmatrix},
\end{equation*}
where $V = W \oplus U$ with $U \coloneqq U^{\infty}$ and $W \coloneqq W^{\infty}$. This differential operator acts on a local $V$-valued holomorphic function $v = v(z)$ as
\begin{equation*}
 Mv = \begin{pmatrix}
 (z - T^{\infty}) w - P u \\
 \partial_z u - Q w
 \end{pmatrix},
\end{equation*}
in the vector decomposition $v(z) = \left(\begin{smallmatrix}
 w(z) \\
 u(z)
 \end{smallmatrix}\right) \in W \oplus U$. Then the equation $Mv = 0$ reads
\begin{equation*}
 \begin{cases}
 Pu = (z - T^{\infty})w, \\
 \partial_z u = Qw,
 \end{cases}
\end{equation*}
which can be turned into the following system of ordinary, first-order linear differential equations for the component $u$ of $v$:
\begin{equation*}
 \partial_z u = Q(z - T^{\infty})^{-1} P u.
\end{equation*}
If $W = \bigoplus_i V_i$ is the eigenspace decomposition of $T^{\infty}$ then one has $T^{\infty} = \bigoplus_i t_i \Id_{V_i}$, and
\begin{equation*}
 Q(z - T^{\infty})^{-1}P = \sum_i \frac{R_i}{z - t_i},
\end{equation*}
where $R_i \coloneqq Q_iP_i$ and $Q_i \colon V_i \rightleftarrows U \ {\colon}\! P_i$ are components of~$Q$ and~$P$.

Hence in this case one recognises a Fuchsian system on the sphere, which can be extended to a logarithmic connection on the trivial holomorphic vector bundle $U \times \mathbb{C}P^1 \to \mathbb{C}P^1$. In general there is an identification with meromorphic connections on the Riemann sphere with irregular singularities at infinity (of Poincar\'e rank~2 if $A \neq 0$, see Section~\ref{sec:slims}).

Now choose the element $g = \left(\begin{smallmatrix}
 0 & 1 \\
 -1 & 0
 \end{smallmatrix}\right) \in \SL_2(\mathbb{C})$, which acts on the generators of the Weyl algebra as $\left(\begin{smallmatrix}
 \partial \\
 z
 \end{smallmatrix}\right) \mapsto \left(\begin{smallmatrix}
 z \\
 -\partial
 \end{smallmatrix}\right)$, i.e., as the Fourier--Laplace transform.
Acting this way on the differential operator $M$ and permuting the direct summands $U$ and $W$ yields
\begin{equation*}
 g.M = \begin{pmatrix}
 0 & \\
 & -\Id_W
 \end{pmatrix} \partial + \begin{pmatrix}
 \Id_U & \\
 & 0
 \end{pmatrix} z - \begin{pmatrix}
 0 & Q \\
 P & T^{\infty}
 \end{pmatrix},
\end{equation*}
so the roles of $U$ and $W$ are swapped. The new degenerate normal form is achieved by changing the sign, i.e., via the diagonal matrix $N' = \left(\begin{smallmatrix}
 \Id_U & \\
 & -\Id_W
 \end{smallmatrix}\right)$. In this case $N = \Id_V$, as $M$ was already taken in normal form, and indeed the overall action reads
\begin{equation*}
 g \colon \ Q \longmapsto -Q, \qquad P \longmapsto P, \qquad T^{\infty} \longmapsto -T^{\infty},
\end{equation*}
as prescribed by~\eqref{eq:action_on_fibration}.

The new equations are $Qw = zu$ and $\partial w = T^{\infty}w - Pu$, which can be expressed for the component $w$ as
\begin{equation*}
 \partial w = \left(T^{\infty} + \frac{R}{z}\right) w,
\end{equation*}
where $R \coloneqq - PQ$. Hence the Harnad-dual of the differential operator ${\rm d} - Q(z - T^{\infty})^{-1} P\,{\rm d}z$ appears as claimed, and it can be thought as a connection with a pole of order two at infinity and a simple pole at zero on the vector bundle $W \times \mathbb{C}P^1 \to \mathbb{C}P^1$.
\end{example*}

The duality of Remark~\ref{example:harnad_duality} thus corresponds to one element inside a 3-dimensional complex group acting on a space of presentation for Weyl algebra modules.

\subsection{Symplectic structure and symplectic action}\label{sec:symplectic_classical_action}

Recall from~\cite{boalch_2012_simply_laced_isomonodromy_systems} that the space
\begin{equation*}
 \mathbb{M} = \bigoplus_{i \neq j \in J} \Hom\big(W^i,W^j\big)
\end{equation*}
is endowed with a complex symplectic structure $\omega_{\bm{a}}$ depending on the embedding $\bm{a} \in \mathbf{A}$. Its formula reads
\begin{equation}\label{eq:symplectic_form}
 \omega_{\bm{a}} = \sum_{i \neq j \in J \setminus \{\infty\}} \frac{\Tr\big({\rm d}B^{ij} \wedge {\rm d}B^{ji}\big)}{2(a_i - a_j)} + \sum_{i \neq \infty} \Tr\big({\rm d}Q^i \wedge {\rm d}P^i\big).
\end{equation}
Hence $\omega_{\bm{a}}$ only pairs nontrivially maps going in opposite directions. Moreover, it coincides with the canonical symplectic form on the subspace
\begin{equation*}
 \bigoplus_{i \neq \infty} \big(\Hom\big(W^{\infty},W^i\big) \oplus \Hom\big(W^i,W^{\infty}\big)\big) \cong T^*\left(\bigoplus_{i \neq \infty} \Hom\big(W^{\infty},W^i\big)\right),
\end{equation*}
where the dualities $\Hom\big(W^{\infty},W^i\big) \cong \Hom\big(W^i,W^{\infty}\big)^*$ are provided by the canonical nondegenerate trace pairing
\begin{equation*}
 \Hom\big(W^{\infty},W^i\big) \otimes \Hom\big(W^i,W^{\infty}\big) \longrightarrow \mathbb{C}, \qquad \big(Q^i,P^i\big) \longmapsto \Tr\big(Q^iP^i\big).
\end{equation*}

The space $\mathbb{M}$ does not depend on $\bm{a}$, but the symplectic form does. Thus an element $g \in \SL_2(\mathbb{C})$ defines a linear map $\varphi(\bm{a}) \colon (\mathbb{M},\omega_{\bm{a}}) \to (\mathbb{M},\omega_{\bm{a}.g})$ between symplectic spaces according to~\eqref{eq:action_on_fibration}, where $a.g$ is the right action~\eqref{eq:action_on_embedding} on the embedding.

\begin{Proposition}[{\cite[Proposition~3.1]{boalch_2012_simply_laced_isomonodromy_systems}}]\label{prop:equivariant_symplectic_form}
The map $\varphi(\bm{a}) \colon (\mathbb{M},\omega_{\bm{a}}) \to (\mathbb{M},\omega_{\bm{a}.g})$ is symplectic.
\end{Proposition}

We will prove this explicitly, and differently from~\cite{boalch_2012_simply_laced_isomonodromy_systems}, thereby checking the computations of the numbers $\eta_i$ of Section~\ref{sec:action_on_weyl_modules}. To this end write~\eqref{eq:symplectic_form} as
\begin{equation}\label{eq:symplectic_form_alternative}
 \omega_{\bm{a}} = \sum_{i \neq j} \frac{\varepsilon_{ij}(\bm{a})}{2} \Tr\big({\rm d}B^{ij} \wedge {\rm d}B^{ji}\big),
\end{equation}
where the (algebraic) functions $\varepsilon_{ij} \colon \mathbf{A} \to \mathbb{C}^{\times}$ are defined by $\varepsilon_{ij} + \varepsilon_{ji} = 0$, and
\begin{equation}\label{eq:symplectic_form_weights}
 \varepsilon_{ij}(\bm{a}) \coloneqq
\begin{cases}
 \dfrac{1}{a_i - a_j}, \quad & a_i,  a_j \neq \infty, \\
 1, &a_j = \infty.
\end{cases}
\end{equation}

\begin{proof}[Proof of Proposition~\ref{prop:equivariant_symplectic_form}]
One must show that $\varphi^*_g(\bm{a}) \omega_{\bm{a}.g} = \omega_{\bm{a}}$ for $g \in \SL_2(\mathbb{C})$ and $\bm{a} \in \mathbf{A}$, i.e.,
\begin{equation*}
 \sum_{i \neq j} \varepsilon_{ij}(\bm{a}) \Tr \big({\rm d}B^{ij} \wedge {\rm d}B^{ji}\big) = \sum_{i \neq j} \varepsilon_{ij}(\bm{a}.g) \varphi^*_g(\bm{a}) \Tr\big({\rm d}B^{ij} \wedge {\rm d}B^{ji}\big).
\end{equation*}

Now~\eqref{eq:classical_action_decomposition} gives an explicit formula for the pull-back along $\varphi(\bm{a})$, from which it follows that the pull-back of the 2-form $\Tr\big({\rm d}B^{ij} \wedge {\rm d}B^{ji}\big)$ on $\mathbb{M}$ equals
\begin{equation*}
 \varphi^*_g(\bm{a}) \Tr\big({\rm d}B^{ij} \wedge {\rm d}B^{ji}\big) = \eta_i(g,\bm{a}) \cdot \eta_j(g,\bm{a}) \Tr\big({\rm d}B^{ij} \wedge {\rm d}B^{ji}\big).
\end{equation*}
Hence Lemma~\ref{lemma:action_cocycle_symplectic} concludes the proof of Proposition~\ref{prop:equivariant_symplectic_form}.
\end{proof}

\begin{Lemma}\label{lemma:action_cocycle_symplectic}
The following identity holds for $i \neq j \in J$, $\bm{a} \in \mathbf{A}$ and $g \in \SL_2(\mathbb{C})$:
\begin{equation}
 \varepsilon_{ij}(\bm{a}) = \varepsilon_{ij}(\bm{a}.g) \cdot \eta_i(g,\bm{a}) \cdot \eta_j(g,\bm{a}).
\end{equation}
\end{Lemma}

\begin{Remark}
This proof of Proposition~\ref{prop:equivariant_symplectic_form} is given since Lemma~\ref{lemma:action_cocycle_symplectic} is of separate interest; it expresses the compatibility of the cocycles $\eta_i$ with the symplectic structures on $\mathbb{M}$, and will imply the compatibility of the $\SL_2(\mathbb{C})$-action with the comoment maps for the Hamiltonian reduction of $\mathbb{M}$ (classical in Section~\ref{sec:classical_hamiltonian_reduction}; quantum in Section~\ref{sec:quantum_hamiltonian_reduction}).
\end{Remark}

\begin{proof}[Proof of Lemma~\ref{lemma:action_cocycle_symplectic}]

Write $g = \left(\begin{smallmatrix}
 a & b \\
 c & d
 \end{smallmatrix}\right) \in \SL_2(\mathbb{C})$ and as usual $\bm{a}(j) = a_j \in \mathbb{C} \cup \{\infty\}$, and consider the same cases as in Section~\ref{sec:action_on_weyl_modules}, using the following identities where necessary (plus $ad - bc =1$):
\begin{alignat}{3}
& a_j.g= -\frac{b\alpha_j + d\beta_j}{a\alpha_j + c\beta_j}, \qquad && \text{if } a \neq ca_j,& \nonumber \\
& a_j.g= -\frac{b - da_j}{a - ca_j}, \qquad && \text{if further } a_j \neq \infty.& \label{eq:algebraic_identities}
\end{alignat}

\subsubsection*{Switching between generic forms}

Assume $a_j \neq \infty \neq a_j.g$ for $j \in J$. In this case one needs to verify that
\begin{equation*}
 a_i - a_j = (a - ca_i)(a - ca_j)(a.g_i - a_j.g).
\end{equation*}
Indeed using~\eqref{eq:algebraic_identities} the right-hand side becomes
\begin{gather*}
 (b - da_j)(a - ca_i) - (b - da_i)(a - ca_j)\\
 \qquad{} = ad(a_i - a_j) - bc(a_i - a_j)  = (ad - bc)(a_i - a_j) = a_i - a_j.
\end{gather*}

\subsubsection*{From a generic form to a degenerate one}

Assume $a_j.g = \infty$ and $a \neq ca_i$ for $i \neq j$. Then $B^{ij}$ and $B^{ji}$ become $B^{i \infty} = Q^i$ and $B^{\infty i} = P^i$ for $i \neq j$, respectively, and apart from the previous identities one must verify that
\begin{equation*}
 (a - ca_i)(b - da_j) = a_i - a_j.
\end{equation*}
Indeed using $a = ca_j$ the left-hand side equals
\begin{gather*}
 ab - ada_j  - bca_i + cd a_ia_j = bca_j - ada_j - bca_i + ada_i
  = (ad - bc)(a_i - a_j) = a_i - a_j.
\end{gather*}

\subsubsection*{From a degenerate form to a generic one}

Assume here $\infty \in \bm{a}(J)$, $c \neq 0$ and $a \neq ca_i$ for all $i \in J \setminus \{\infty\}$. Then $\infty.g = -\frac{d}{c}$, and if $j \coloneqq \bm{a}^{-1}(-d/c) \in J$ then $B^{i \infty} = Q^i$ and $B^{\infty i} = P^i$ become $B^{ij}$ and $B^{ji}$~-- respectively. Then apart from the previous identities one must show that
\begin{equation*}
 a.g_i + d/c = c^{-1}(a - ca_i)^{-1}.
\end{equation*}
Indeed using~\eqref{eq:algebraic_identities} the left-hand side becomes
\begin{equation*}
 \frac{d}{c} - \frac{b - da_i}{a - ca_i} = \frac{d(a - ca_i) - c(b - da_i)}{c(a - ca_i)} = \frac{ad - bc}{c(a - ca_i)} = c^{-1}(a - ca_i)^{-1}.
\end{equation*}

\subsubsection*{Switching between degenerate forms}

Assume first $\infty \in \bm{a}(J)$ and $c = 0$. Then $\infty.g = \infty$, and in this case the further identities required follow from $\eta_{\infty}(g,\bm{a}) \cdot \eta_i(g,\bm{a}) = ad = 1$ for $i \neq \infty$.

Assume finally $\infty \in \bm{a}(J)$, $c \neq 0$ and $a = ca_j$ for some (unique) $j \in J \setminus \{\infty\}$. Then one has $\infty.g = -d/c$ and $a_j.g = \infty$; also $Q^j$ and $P^j$ are swapped, whereas $B^{ij}$ and $B^{ji}$ are exchanged with $Q^i$ and $P^i$ for all $i \in J \setminus \{\infty,j\}$~-- respectively. Hence apart from the previous identities one must establish that
\begin{equation*}
 c(b - da_j) = -1,
\end{equation*}
which follows from $ca_j = a$.
\end{proof}

Thus the action of $\SL_2(\mathbb{C})$ on the embeddings $\bm{a} \colon J \hookrightarrow \mathbb{C} \cup \{\infty\}$ is lifted to an action on the symplectic vector bundle $\widetilde{\mathbb{M}} \to \mathbf{A}$, whose fibre over $\bm{a}$ is the space $(\mathbb{M},\omega_{\bm{a}})$. This action is explicitly given by the assignment $\varphi \colon \SL_2(\mathbb{C}) \times \mathbf{A} \to \GL(\mathbb{M})$ sending $(g,\bm{a})$ to $\varphi(\bm{a})$, together with the cocycle identities
\begin{equation}\label{eq:cocycle_classical_action}
 \varphi_{gg'}(\bm{a}) = \varphi_{g'}(\bm{a}.g) \circ \varphi(\bm{a}), \qquad \text{for} \quad g,g' \in \SL_2(\mathbb{C}), \  \bm{a} \in \mathbf{A},
\end{equation}
which follow from~\eqref{eq:action_cocycle}. In particular the inverse of $\varphi(\bm{a})$ is $\varphi_{g^{-1}}(\bm{a}.g) \colon (\mathbb{M},\omega_{\bm{a}.g}) \to (\mathbb{M},\omega_{\bm{a}})$.

\subsection{Action on classical algebras}\label{sec:action_on_algebras}

Consider the commutative Poisson algebra of polynomial functions on the complex vector space~$\mathbb{M}$~-- considered as an affine complex space:
\begin{equation*}
 A_0 \coloneqq \mathbb{C}[\mathbb{M}] \cong \Sym(\mathbb{M}^*).
\end{equation*}
The Poisson bracket is $\{\cdot,\cdot\}_{\bm{a}} = \omega_{\bm{a}}^{-1} \in \bigwedge^2 \mathbb{M}$, uniquely determined from its restriction to linear functions which defines a nondegenerate alternating bilinear form on $\mathbb{M}^*$. The commutative algebra structure of $A_0$ only depends on $\mathbb{M}$, whereas the Poisson structure also depends on the embedding $\bm{a}$.

As a corollary of Proposition~\ref{prop:equivariant_symplectic_form} and~\eqref{eq:cocycle_classical_action} one gets the following.

\begin{Proposition}\label{prop:classical_action}
The pull-back of polynomial functions along the $\SL_2(\mathbb{C})$-action defines isomorphisms of Poisson algebras
\begin{equation*}
 \varphi^*_g(\bm{a}) \colon \  (A_0,\{\cdot,\cdot\}_{\bm{a}.g}) \longrightarrow (A_0,\{\cdot,\cdot\}_{\bm{a}}),
\end{equation*}
satifying
\begin{equation*}
 \varphi^*_{gg'}(a) = \varphi^*_g(\bm{a}) \circ \varphi^*_{g'}(\bm{a}.g), \qquad \text{for}\quad g, g' \in \SL_2(\mathbb{C}), \  \bm{a} \in \mathbf{A}.
\end{equation*}
\end{Proposition}

In conclusion, the action of $\SL_2(\mathbb{C})$ on the embeddings $a \colon J \hookrightarrow \mathbb{C} \cup \{\infty\}$ is lifted to an action on the bundle of Poisson algebras $\mathcal{A}_0 \to \mathbf{A}$ whose fibre over $a$ is by definition $(A_0,\{\cdot,\cdot\}_{\bm{a}})$. This is a trivial bundle of (commutative) associative algebras, which is not trivial as bundle of Poisson algebras, and the assignment $(g,\bm{a}) \mapsto \varphi^*_g(\bm{a})$ for $g \in \SL_2(\mathbb{C})$ and $\bm{a} \in \mathbf{A}$ defines an action on the total space which covers the one on the base.

This statement will be provided with a quantum analogue in the next section.

\section{Quantisation of the action}\label{sec:quantum_action}

\subsection{Filtered quantisation}\label{sec:filtered_quantisation}

The following material about the filtered quantisation of the commutative Poisson algebra $A_0 = \Sym(\mathbb{M}^*)$ is standard (see, e.g.,~\cite{schedler_2012_deformations_of_algebras_in_noncommutative_geometry}).

\begin{Definition}\label{def:weyl_algebra}
The Weyl algebra of the vector space $\mathbb{M}^*$ equipped with the alternating bilinear form $\{\cdot,\cdot\}_{\bm{a}} \colon \bigwedge^2 \mathbb{M}^* \to \mathbb{C}$ is the quotient
\begin{equation}\label{eq:weyl_algebra}
 A_{\bm{a}} = W(\mathbb{M}^*,\{\cdot,\cdot\}_{\bm{a}}) \coloneqq \Tens(\mathbb{M}^*) \big\slash I_1(\bm{a}),
\end{equation}
where $\Tens(\mathbb{M}^*)$ is the tensor algebra and $I_1(\bm{a}) \subseteq \Tens(\mathbb{M}^*)$ the two-sided ideal generated by the elements
\begin{equation*}
 f \otimes g - g \otimes f - \{f,g\}_a, \qquad \text{for} \quad f,g \in \mathbb{M}^*.
\end{equation*}
\end{Definition}

Hence $A_{\bm{a}}$ admits a set of algebra generators which do not depend on $\bm{a}$, but the relations among them do. In what follows we will drop the explicit dependence of $A_{\bm{a}}$ from $\bm{a}$ in the notation, but keep it for the noncommutative associative product $\ast_{\bm{a}} \colon A \otimes A \to A$.

\begin{Remark}Definition~\ref{def:weyl_algebra} is the dual description of the Weyl algebra of the symplectic vector space $(\mathbb{M},\omega_{\bm{a}})$, generalising the one-dimensional Weyl algebra $A_1$ of Section~\ref{sec:weyl_modules}. Using \eqref{eq:weyl_algebra} instead of $W(\mathbb{M},\omega_{\bm{a}})$ is preferred as one wants to quantise $\Sym(\mathbb{M}^*)$ rather then $\Sym(\mathbb{M})$, so this saves a canonical symplectic identification $\mathbb{M} \cong \mathbb{M}^*$.
\end{Remark}

To make~\eqref{eq:weyl_algebra} into a filtered quantisation of $(A_0,\{\cdot,\cdot\}_{\bm{a}})$ one must choose a grading on $A_0$ and a filtration on $A$, and provide an isomorphism $\gr(A) \cong A_0$ of graded (commutative) Poisson algebras.

\begin{Definition}The additive/Bernstein grading $\mathcal{B}_0$ on $A_0 = \Sym(\mathbb{M}^*)$ is defined by the family of subspaces
\begin{equation*}
 (\mathcal{B}_0)_k \coloneqq \Sym^k (\mathbb{M}^*).
\end{equation*}
\end{Definition}

This is the grading by the global degree of polynomial functions on $\mathbb{M}$. It corresponds to the quotient grading on $\Sym(\mathbb{M}^*) = \Tens(\mathbb{M}^*) \big\slash I_0$ induced by the additive grading
\begin{equation*}
 \Tens(\mathbb{M}^*) = \bigoplus_{k \geq 0} (\mathbb{M}^*)^{\otimes k},
\end{equation*}
where $I_0 \subseteq \Tens(\mathbb{M}^*)$ is the (homogeneous) two-sided ideal generated by elements
\begin{equation*}
 f \otimes g - g \otimes f, \qquad \text{for} \quad f,g \in \mathbb{M}^*.
\end{equation*}

\begin{Remark}A different grading is obtained by specifying a symplectic identification $\mathbb{M} = T^*\mathbb{L}$, where $\mathbb{L} \subseteq \mathbb{M}$ is a Lagrangian subspace. Then one can give degree zero to the linear coordinates on $\mathbb{L}$ (the positions) and degree one to the linear coordinates on the cotangent fibres $T^*_q \mathbb{L}$, where $q \in \mathbb{L}$ (the momenta).

This yields the geometric grading on $A_0$, which will not be used since the choice of a Lagrangian subspace is noncanonical and breaks the $\SL_2(\mathbb{C})$ symmetries (cf.\ Remark~\ref{remark:lagrangian_splitting}).
\end{Remark}

Similarly there is a quotient filtration on the Weyl algebra~\eqref{eq:weyl_algebra} modulo the nonhomogeneous ideal $I_1(\bm{a})$.

\begin{Definition}The additive/Bernstein filtration $\mathcal{B}$ on $A$ is the quotient of the filtration associated to the additive grading on $\Tens(\mathbb{M}^*)$, modulo the ideal $I_1(\bm{a})$.
\end{Definition}

This means the Bernstein filtration on $A$ is defined by the subspaces
\begin{equation*}
 \mathcal{B}_{\leq k} \coloneqq \pi_1 \Bigg(\bigoplus_{m \leq k} (\mathbb{M}^*)^{\otimes m}\Bigg) \subseteq A,
\end{equation*}
where $\pi_1 \colon \Tens(\mathbb{M}^*) \to A$ is the canonical projection.

Now a standard argument shows that the associated graded of the filtered associative algebra~$(A,\mathcal{B})$ is isomorphic to $(A_0,\mathcal{B}_0)$ as graded Poisson algebra, so that by definition $(A,\ast_{\bm{a}},\mathcal{B})$ is a~filtered quantisation of $(A_0,\{\cdot,\cdot\}_{\bm{a}},\mathcal{B}_0)$.

In more detail, the associated graded of $(A,\mathcal{B})$ is the graded vector space
\begin{equation*}
 \gr(A) \coloneqq \bigoplus_{k \geq 0} \mathcal{B}_{\leq k} \big\slash \mathcal{B}_{\leq k-1},
\end{equation*} with product defined on representatives. There are then canonical projections
\begin{equation*}
 \sigma_k \colon \ \mathcal{B}_{\leq k} \longrightarrow \mathcal{B}_{\leq k} \big\slash \mathcal{B}_{\leq k-1},
\end{equation*}
together with identifications $\mathcal{B}_{\leq k} \big\slash \mathcal{B}_{\leq k-1} \cong (\mathcal{B}_0)_k \subseteq A_0$,
and the compatibility with the Poisson bracket is expressed by the identity
\begin{equation*}
 \big\{\sigma_k(x),\sigma_l(y)\big\}_{\bm{a}} = \sigma_{k+l-2}\big(x \ast_{\bm{a}} y - y \ast_{\bm{a}} x\big), \qquad \text{for} \quad x \in \mathcal{B}_{\leq k}, \  y \in \mathcal{B}_{\leq l}.
\end{equation*}
Moreover there is a universal embedding $\mathbb{M}^* \hookrightarrow A$, provided by the fact that $I_1(\bm{a}) \cap \mathbb{M}^* = (0)$ inside $\Tens(\mathbb{M}^*)$.

An analogous construction can be carried out for a Lie algebra $\mathfrak{g}$. The symmetric algebra $\Sym (\mathfrak{g}) \cong \mathbb{C} [\mathfrak{g}^*]$ carries the Poisson--Lie bracket and the grading by the global degree of polynomial functions on $\mathfrak{g}^*$. Then the universal enveloping algebra $U(\mathfrak{g})$ equipped with the quotient of the filtration defined by the subspaces $\bigoplus_{m \leq k} \mathfrak{g}^{\otimes m} \subseteq \Tens(\mathfrak{g})$~-- for $k \geq 0$~-- is a filtered quantisation of $\Sym(\mathfrak{g})$ (this is one way of stating the Poincar\'e--Birkhoff--Witt theorem).

In what follows the filtration on $\mathfrak{g}$ will be denoted $\mathcal{B}_{\mathfrak{g}}$.

\subsection{Rees construction and deformation quantisation}\label{sec:rees_construction}

Filtered quantisation is a particular instance of deformation quantisation, as there is a universal construction to introduce a formal deformation parameter $\hslash$. This will be used to formalise the semiclassical limit, and it allows for the deformation of nonhomogeneous ideals.

\begin{Definition}\label{def:rees_algebra}The Rees algebra of the filtered associative algebra $(A,\mathcal{B})$ is the $\hslash$-graded ring
\begin{gather}\label{eq:rees_algebra}
 \Rees(A,\mathcal{B}) \coloneqq \bigoplus_{k \geq 0} \mathcal{B}_{\leq k} \cdot \hslash^k \subseteq A[\hslash].
\end{gather}
\end{Definition}

To allow for nonconverging power series one can embed the Rees algebra into
\begin{gather}\label{eq:completion_rees}
 \widehat{A} \coloneqq \Rees(A,\mathcal{B})\llbracket \hslash \rrbracket
 = \Bigg\{ \sum_{k \geq 0} f_k \cdot \hslash^k \bigg| f_k \in \mathcal{B}_{\leq k}, \, \lim_{k \longrightarrow +\infty} \big(k - \lvert f_k \rvert\big) = +\infty\Bigg\} \subseteq A\llbracket \hslash \rrbracket,
\end{gather}
where $\lvert x \rvert = \min \big\{k \in \mathbb{Z}_{\geq 0} \, \big\vert \, x \in \mathcal{B}_{\leq k}\big\}$ is the order of the element $x \in A$ (see~\cite{etingof_schiffmann_1998_lectures_on_quantum_groups}).

Then there is a surjective algebra morphism $\sigma \colon \widehat{A} \to \gr(A) \cong A_0$ defined by
\begin{equation}\label{eq:semiclassical_limit}
 \sigma \colon \ \sum_k f_k \cdot \hslash^k \longmapsto \sum_k \sigma_k(f_k).
\end{equation}
This map is well defined (that is, the sum on the right-hand is finite), vanishes on the ideal $\hslash\widehat{A}$, and is surjective. This yields an identification $\widehat{A} \slash \hslash\widehat{A} \cong \gr(A)$, meaning the topologically free $\mathbb{C}\llbracket \hslash \rrbracket$-algebra $\widehat{A}$ is a (one-parameter) formal deformation of $\gr(A) \cong A_0$. The compatibility with the Poisson bracket makes it into a formal deformation quantisation, and is expressed by the identity
\begin{equation*}
 \sigma\big([f,g] \cdot \hslash^{-2}\big) = \big\{\sigma(f),\sigma(g)\big\}_{\bm{a}} \in A_0, \qquad \text{for} \quad f,g \in \widehat{A}.
\end{equation*}
Indeed $\big[\mathcal{B}_{\leq k}, \mathcal{B}_{\leq l}\big] \subseteq \mathcal{B}_{k+l-2}$ implies that $[f,g] \in \hslash^2 \widehat{A}$ for $f,g \in \widehat{A}$.

Because of this, the morphism $\sigma$ is called the \textit{semiclassical limit}.

\begin{Remark}\label{remark:star_product}Equivalently, there is an isomorphism $\widehat{A} \cong A_0\llbracket \hslash \rrbracket$ of $\mathbb{C}\llbracket \hslash \rrbracket$-modules such that the product of two elements $f,g \in A_0$ expands in $\widehat{A}$ as
\begin{equation*}
 f \ast_{\bm{a}} g = \sum_{k \geq 0} c_{k,\bm{a}}(f,g) \cdot \hslash^k,
\end{equation*}
where the $c_{k,\bm{a}} \colon A_0 \otimes A_0 \to A_0$ are bilinear maps satisfying the conditions imposed by the associativity of the product of $\widehat{A}$.

In this notation the alternating bilinear map $\left.c_{2,\bm{a}}\right\vert_{A_0 \wedge A_0}$ is the Poisson structure $\{\cdot,\cdot\}_{\bm{a}}$ of $A_0$, whereas $\left.c_{1,\bm{a}}\right\vert_{A_0 \wedge A_0} = 0$ for all $\bm{a} \in \mathbf{A}$ (which is reminiscent of the fact that $\{\cdot,\cdot\}_{\bm{a}}$ is a $2$-shifted Poisson structure on $\gr(A) \cong A_0$).
\end{Remark}

\begin{Remark}A different viewpoint on the deformation parameter $\hslash$ is obtaining by tur\-ning~$I_1(\bm{a})$ into a homogeneous ideal inside the $\mathbb{C}\llbracket \hslash \rrbracket$-algebra $\Tens(\mathbb{M}^*)\llbracket \hslash \rrbracket$, which yields the definition of the homogenised Weyl algebra $W_{\hslash}(\mathbb{M}^*,\{\cdot,\cdot\}_{\bm{a}}) \subseteq A\llbracket \hslash \rrbracket$ (see~\cite[Section~2.6]{schedler_2012_deformations_of_algebras_in_noncommutative_geometry}). In what follows only the formulation with the Rees algebra~\eqref{eq:rees_algebra} will be used, for the sake of a uniform notation.
\end{Remark}

This material can be adapted to the universal enveloping algebra $U(\mathfrak{g})$ of a Lie algebra $\mathfrak{g}$: there is a (polynomial) Rees algebra $\Rees(U(\mathfrak{g}),\mathcal{B}_{\mathfrak{g}}) \subseteq U(\mathfrak{g})[\hslash]$ with an extension $\widehat{U}(\mathfrak{g}) \subseteq U(\mathfrak{g})\llbracket \hslash \rrbracket$ to power series following Definition~\ref{def:rees_algebra} and~\eqref{eq:completion_rees}, respectively.

\begin{Remark}
The construction of the universal simply-laced quantum connections of Section~\ref{sec:slqc} only involves polynomials in $\hslash$, but it is sometimes important to allow for power series. For example, a nonhomogeneous ideal deforming an ideal $I \subseteq \Sym(\mathfrak{g})$ may exist in $\widehat{U}(\mathfrak{g})$ and contain power series which do not converge for particular values of $\hslash$ (if such an ideal exists in the uncompleted Rees algebra this is, however, implied; cf.\ Remark~\ref{remark:orbit_quantisation} about the deformation of ideals vanishing on semisimple coadjoint orbits).
\end{Remark}

\subsection{Quantisation of linear symplectic actions}\label{sec:quantisation_of_action}

The following abstract fact will be essential in what follows. Let $(V,\omega_V)$ and $(U,\omega_U)$ be finite-dimensional symplectic vector spaces, denote $\{\cdot,\cdot\}_V$ and $\{\cdot,\cdot\}_U$ the Poisson brackets associated to $\omega_V$ and $\omega_U$~-- respectively~-- and suppose $\varphi \colon V \to U$ is a linear Poisson map.

\begin{Lemma}[Contravariant functoriality]\label{lemma:linear_functoriality}
There exists a unique morphism of associative algebras
\begin{equation*}
 \widehat{\varphi}^* \colon \ W\big(U^*,\{\cdot,\cdot\}_U\big) \longrightarrow W(V^*,\{\cdot,\cdot\}_V),
\end{equation*}
whose associated graded equals the pull-back $\varphi^* \colon \Sym(U^*) \to \Sym(V^*)$ along $\varphi$. It is defined by the equality $\widehat{\varphi}^* \circ \pi_1 = \pi_1 \circ \widetilde{\varphi}^*$, where $\widetilde{\varphi}^* \colon \Tens(U^*) \to \Tens(V^*)$ is the morphism associated to the restriction of the pull-back to linear maps. Moreover the association $\varphi \mapsto \widehat{\varphi}^*$ is compatible with composition.
\end{Lemma}

\begin{proof}The pull-back of linear functions yields a linear map $\varphi^* \colon U^* \to V^*$, since $\varphi$ itself is linear, so $\widetilde{\varphi}^*$ is the unique morphism of associative algebras defined on monomials by
\begin{equation*}
 \widetilde{\varphi}^*\Big(\bigotimes_i f_i\Big) = \bigotimes_i \varphi^* f_i, \qquad \text{where} \quad f_i \in U^* \text{ for } i = 1, \dotsc, n.
\end{equation*}

By the universal property of quotients it is then enough to show that $\widetilde{\varphi}(I_{1,U}) \subseteq I_{1,V}$ inside $\Tens(V^*)$, where $I_{1,U}$ and $I_{1,V}$ are the nonhomogenous ideals defining the Weyl algebras as in~\eqref{eq:weyl_algebra}. But this follows from the fact that $\varphi$ is Poisson, that is from the identity
\begin{equation*}
 \varphi^*\{f,g\}_U = \{\varphi^*f,\varphi^*g\}_V, \qquad \text{for} \quad f,g \in \mathbb{C}[U] \cong \Sym(U^*).
\end{equation*}

Next one must prove that the associated graded map of the induced morphism $\widehat{\varphi}^*$ on the Weyl algebras coincide with the pull-back $\varphi^*$, i.e., that $\sigma_k \circ \widehat{\varphi}^* = \varphi^* \circ \sigma_k$ on all filtration spaces.

By linearity it is enough to show that
\begin{equation*}
 \sigma_k \circ \pi_1 \circ \widetilde{\varphi}^* \bigg(\bigotimes_i X_i\bigg) = \varphi^* \circ \sigma_k \bigg(\prod_i \widehat{X_i}\bigg),
\end{equation*}
for all $k \geq 0$ and all $X_i \in U^*$ such that $\bigotimes_i X_i$ is homogeneous of degree $k$, where we denote $\pi_1(\bigotimes_i X_i) = \prod_i \widehat{X_i}$ the product in the Weyl algebra. The left-hand side equals $\sigma_k\left(\prod_i \widehat{\varphi^*X_i}\right)$, which is the function $\prod_i \varphi^*X_i \in \Sym(U^*)$, in the identification $\Sym(U^*) \cong \gr \big(W(U^*,\{\cdot,\cdot\}_U)\big)$. This is equal to the right-hand side since
\begin{equation*}
 \prod_i \varphi^*X_i = \varphi^*\bigg(\prod_i X_i \bigg) = \varphi^* \circ \sigma_k\bigg(\prod_i \widehat{X}_i \bigg).
\end{equation*}
Finally, the compatibility with the composition follows from uniqueness. If $\psi \colon W \to V$ is another morphism of linear Poisson spaces then $\widehat{\psi^* \circ \varphi^*}$ and $\widehat{\psi}^* \circ \widehat{\varphi}^*$ are both defined, and since their associated graded both coincide with $\psi^* \circ \varphi^*$ then they must be equal.
\end{proof}

The morphisms of Lemma~\ref{lemma:linear_functoriality} also preserve the Bernstein filtration, hence induces a morphism of Rees algebras~-- or the formal power series version~-- by $\mathbb{C} \llbracket \hslash \rrbracket$-linearity:
\begin{equation*}
 \widehat{\varphi}^*_{\hslash} \colon \sum_k f_k \hslash^k \longmapsto \sum_k \widehat{\varphi}^*(f_k) \hslash^k.
\end{equation*}
This is the unique morphism satisfying $\sigma \circ \widehat{\varphi}^* = \varphi^* \circ \sigma$, i.e., such that $\widehat{\varphi}^*_{\hslash} = \varphi^* \pmod \hslash$.

Lemma~\ref{lemma:linear_functoriality} can be used to conclude that there is a quantum $\SL_2(\mathbb{C})$-action lifting the classical action of Section~\ref{sec:action_on_algebras}. Choose an element $g \in \SL_2(\mathbb{C})$ and fix an embedding $\bm{a} \in \mathbf{A}$.

\begin{Theorem}\label{theorem:quantum_action}There exists a unique isomorphism of $\mathbb{C}\llbracket \hslash \rrbracket$-algebras $\widehat{\varphi}^*_{g,\hslash}(\bm{a}) \colon\! \big(\widehat{A},\ast_{\bm{a}.g}\big) \!\to\! \big(\widehat{A},\ast_{\bm{a}}\big)$ which intertwines the pull-back $\varphi^*_g(\bm{a}) \colon (A_0,\{\cdot,\cdot\}_{\bm{a}.g}) \to (A_0,\{\cdot,\cdot\}_{\bm{a}})$ with the semiclassical limits:
\[
\begin{tikzpicture}
 \node (a) at (0,0) {$\big(\widehat{A},\ast_{\bm{a}.g}\big)$};
 \node (b) at (4,0) {$\big(\widehat{A},\ast_{\bm{a}}\big)$};
 \node (c) at (0,-2) {$(A_0,\{\cdot,\cdot\}_{\bm{a}.g})$};
 \node (d) at (4,-2) {$(A_0,\{\cdot,\cdot\}_{\bm{a}})$.};
 \node (e) at (2,-1) {$\circlearrowleft$};
\path
 (a) edge node[above]{$\widehat{\varphi}^*_{g,\hslash}(\bm{a})$} (b)
 (a) edge node[left]{$\sigma$} (c)
 (b) edge node[right]{$\sigma$} (d)
 (c) edge node[below]{$\varphi^*_g(\bm{a})$} (d);
\end{tikzpicture}
\]
Moreover if $g' \in \SL_2(\mathbb{C})$ is another element then $\widehat{\varphi}^*_{gg',\hslash}(\bm{a}) = \widehat{\varphi}^*_{g,\hslash}(\bm{a}) \circ \widehat{\varphi}^*_{g',\hslash}(\bm{a}.g)$.
\end{Theorem}

\begin{proof}It follows from the discussion of Section~\ref{sec:action_on_algebras} that for all $g \in \SL_2(\mathbb{C})$ there is a linear symplectic map $\varphi_g \colon (\mathbb{M},\omega_{\bm{a}}) \to (\mathbb{M},\omega_{\bm{a}.g})$. Using Lemma~\ref{lemma:linear_functoriality} one concludes that there exists a~morphisms of associative algebras $\widehat{\varphi}^*_{g,\hslash}(\bm{a}) \colon W(\mathbb{M}^*,\{\cdot,\cdot\}_{\bm{a}.g}) \to W(\mathbb{M}^*,\{\cdot,\cdot\}_{\bm{a}})$ which restricts to the pull-back $\varphi^*_g(\bm{a})$ on linear functions, and then an induced morphism on the completion $\widehat{A}$. Finally, by linearity one finds
\begin{gather*}
 \sigma \circ \widehat{\varphi}^*_{g,\hslash}(\bm{a}) \bigg(\sum_k f_k \hslash^k\bigg)  = \sum_k \sigma_k\big(\widehat{\varphi}^*_{g,\hslash}(\bm{a})(f_k)\big) = \varphi^*_g(\bm{a})\bigg(\sum_k \sigma_k(f_k) \bigg)\\
 \hphantom{\sigma \circ \widehat{\varphi}^*_{g,\hslash}(\bm{a}) \bigg(\sum_k f_k \hslash^k\bigg)}{}
 = \varphi^*_g(\bm{a}) \circ \sigma \bigg(\sum_k f_k \hslash^k\bigg),
\end{gather*}
whence indeed $\sigma \circ \widehat{\varphi}^*_{g,\hslash}(\bm{a}) = \varphi^*_g(\bm{a}) \circ \sigma$. The compatibility with the product follows from the uniqueness of Lemma~\ref{lemma:linear_functoriality} and from~\eqref{eq:cocycle_classical_action}, since
\begin{equation*}
 \sigma \circ \widehat{\varphi}^*_{gg',\hslash}(\bm{a}) = \varphi^*_{gg'}(\bm{a}) \circ \sigma = \big(\varphi^*_{g}(\bm{a}) \circ \varphi^*_{g'}(\bm{a}.g)\big) \circ \sigma = \big(\widehat{\varphi}^*_{g,\hslash}(\bm{a}) \circ \widehat{\varphi}^*_{g',\hslash}(\bm{a}.g)\big) \circ \sigma.\tag*{\qed}
\end{equation*}\renewcommand{\qed}{}
\end{proof}

Moreover, there is an explicit formula for the quantum action on monomials. To write it choose bases of $W^j \subseteq V$ for $j \in J$ and denote $B^{ij}_{kl} \in \mathbb{M}^*$ the associated linear coordinates which take the components of $B^{ij} \colon W^j \to W^i$.

\begin{Definition}The Weyl quantisation of a linear function on $\mathbb{M}$ is its image in the universal embedding $\mathbb{M}^* \hookrightarrow A$.
\end{Definition}

Denote the Weyl quantisation by $f \mapsto \widehat{f} \in \mathcal{B}_1 \subseteq A$, for $f \in \mathbb{M}^*$. Then~\eqref{eq:classical_action_decomposition} yields the following formula on monomials:
\begin{equation}\label{eq:quantum_action}
 \widehat{\varphi}^*_{g,\hslash}(\bm{a})\bigg(\prod_p \widehat{B}^{i_p j_p}_{k_pl_p}\bigg) \cdot \hslash^l = \eta \cdot \prod_p \widehat{B}^{i_pj_p}_{k_pl_p} \cdot \hslash^l, \qquad \text{where} \quad \eta \coloneqq \prod_p \eta_{i_p}(g,\bm{a}),
\end{equation}
with the functions $\eta_{i_p} \colon \SL_2(\mathbb{C}) \times \mathbf{A} \to \mathbb{C}^*$ defined in~\eqref{sec:action_on_weyl_modules}.

Hence on the whole the action of $\SL_2(\mathbb{C})$ on the embeddings $\bm{a} \colon J \hookrightarrow \mathbb{C} \cup \{\infty\}$ is lifted to an action on the bundle of noncommutative algebras $\widehat{\mathcal{A}} \to \mathbf{A}$, whose fibre over $a$ is by definition $\big(\widehat{A},\ast_{\bm{a}}\big)$. In contrast with $\mathcal{A}_0$, $\widehat{\mathcal{A}}$ is not given as a trivial bundle of associative algebras, but there is a global trivialisation as a bundle of $\mathbb{C} \llbracket \hslash \rrbracket$-modules provided by the canonical $\bm{a}$-independent identifications $\widehat{A} \cong A_0 \llbracket \hslash \rrbracket$ (see Remark~\ref{remark:star_product}). Further, taking fibrewise semiclassical limit defines a map of bundles $\widehat{\mathcal{A}} \to \mathcal{A}_0$ over the identity.

Then the assignment $g \mapsto \widehat{\varphi}^*_{g,\hslash}$ for $g \in \SL_2(\mathbb{C})$ defines a $\SL_2(\mathbb{C})$-action on $\widehat{\mathcal{A}}$ covering that on the base, and the semiclassical limit intertwines it with the classical action on $\mathcal{A}_0$.

This is a quantisation of the $\SL_2(\mathbb{C})$-action of~\cite{boalch_2012_simply_laced_isomonodromy_systems}.

\section{Classical isomonodromy system}\label{sec:slims}

The definition of the simply-laced isomonodromy systems of~\cite{boalch_2012_simply_laced_isomonodromy_systems} will be recalled in this section.

As in Section~\ref{sec:weyl_modules} let $V$ be a finite-dimensional vector space, and consider a differential operator $M = \alpha \partial + \beta z - \gamma \in \End(V) \otimes A_1$, where $A_1 = \mathbb{C}[\partial,z]$ is the one-dimensional Weyl algebra~\eqref{eq:weyl_algebra_dimension_one}. Suppose $M$ is put in normal form, that is either the degenerate form~\eqref{eq:normal_degenerate_form} or the generic form~\eqref{eq:normal_generic_form}. Then the system $Mv = 0$ for a holomorphic $V$-valued function $v$ can be written
\begin{equation*}
 \partial_z u = \big(Az + B + T + Q(z - T^{\infty})^{-1}P\big)u
\end{equation*}
in the case of a degenerate normal form, and
\begin{equation*}
 \partial_z u =  (Az + B + T )u
\end{equation*}
in the case of a generic form, where $u$ is the component of $v$ taking values in $U^{\infty} \subseteq V$ (whence $u = v$ in the generic case). This is a first-order system of linear differential equations with rational coefficients for the function $u$, which can be extended to a meromorphic connection on the trivial holomorphic vector bundle $U^{\infty} \times \mathbb{C}P^1 \to \mathbb{C}P^1$. Then~$u$ becomes a local section of the vector bundle, and the above differential equations express the fact that it is covariantly constant.

The next step is to introduce isomonodromic deformations of such meromorphic connections, as follows. Recall from Section~\ref{sec:weyl_modules} that $U^{\infty} = \bigoplus_{j \in J \setminus \{\infty\}} W^j$ is graded by a finite set~$J$, and that the diagonal part $\widehat{T} = (T^{\infty},T) \in \End(V)$ of $\gamma$ consists of semisimple endomorphisms
\begin{equation*}
 T^{\infty} \in \End\big(W^{\infty}\big), \qquad T = \bigoplus_{j \in J \setminus \{\infty\}} T^j \in \bigoplus_{j \in J \setminus \{\infty\}} \End\big(W^j\big) \subseteq \End\big(U^{\infty}\big).
\end{equation*}
Introduce accordingly the finer decomposition of $V$ by splitting $W^j$ into eigenspaces for $T^j$:
\begin{equation*}
 W^j = \bigoplus_{i \in I^j} V_i,
\end{equation*}
where $I^j$ is a finite set indexing the spectrum of $T^j$. Hence $V = \bigoplus_{i \in I} V_i$, with $I \coloneqq \coprod_{j \in J} I^j$, and one has
\begin{equation}\label{eq:isomonodromy_times}
 T^j = \bigoplus_{i \in I^j} t_i \Id_{V_i},
\end{equation}
where $\{t_i\}_{i \in I^j}$ is the spectrum of $T$. The admissible variations of these spectra give the isomonodromy times.

Define then the space of isomonodromy times $\mathbf{B} \subseteq \mathbb{C}^I$ as the open set corresponding to variations of the spectral types of $T$ and $T^{\infty}$ such that the eigenspace decomposition of each $W^j$ does not change. This means one is allowed to vary the eigenvalues so that $t_i \neq t_k$ for $i \neq k \in I^j$, whence
\begin{equation*}
 \mathbf{B} \coloneqq \prod_{j \in J} \mathbb{C}^{I^j} \setminus \big\{\{t_i\}_{i \in I^J} \,|\, t_i \neq t_k \text{ for } i \neq k\big\} \subseteq \mathbb{C}^I.
\end{equation*}
If the $I$-grading $V = \bigoplus_{i \in I} V_i$ is fixed then the data of the semisimple endomorphisms~\eqref{eq:isomonodromy_times} corresponds to giving a~point of~$\mathbf{B}$.

\begin{Theorem}[\cite{boalch_2012_simply_laced_isomonodromy_systems}]\label{theorem:slims}
There exists a time-dependent Hamiltonian system $H \colon \mathbb{M} \times \mathbf{B} \to \mathbb{C}^I$ whose flow controls the isomonodromic deformations of meromorphic connections of the form
\begin{gather}\label{eq:meromorphic_connections}
 \nabla = {\rm d} - \big(Az + B + T + Q(z - T^{\infty})^{-1}P\big){\rm d}z = {\rm d} - \left(Az + B + T + \sum_{i \in I^{\infty}} \frac{R_i}{z - t^{\infty}_i}\right){\rm d}z,\!\!\!
\end{gather}
defined on the trivial holomorphic vector bundle $U^{\infty} \times \mathbb{C}P^1 \to \mathbb{C}P^1$, where $R_i = Q^iP^i \in \End(U^{\infty})$. Moreover, the system is strongly flat, which means that
\begin{equation*}
 \frac{\partial H_i}{\partial t_j} - \frac{\partial H_j}{\partial t_i} = 0 = \{H_i,H_j\}_{\bm{a}},
\end{equation*}
for all $i, j \in I$, where $H_i$ is the $i$-th component of~$H$.
\end{Theorem}

In the statement one takes the Poisson bracket of the fibrewise restriction of the time-dependent Hamiltonians to the symplectic phase-space $\mathbb{M}$. Note~\eqref{eq:meromorphic_connections} simplifies to $\nabla = {\rm d} - (Az + B + T){\rm d}z$ in the case of a generic form.

\begin{Definition}\label{def:slims}The time-dependent Hamiltonian system $H$ is the simply-laced isomonodromy system attached to the data of the partitioned set $I = \coprod_{j \in J} I^j$, the $I$-graded space $V = \bigoplus_{i \in I} V_i$, and the embedding $\bm{a} \colon J \hookrightarrow \mathbb{C} \cup \{\infty\}$ of the set of parts of $I$. The components of $H$ are the simply-laced Hamiltonians.
\end{Definition}

Setting $\mathbb{F}_{\bm{a}} \coloneqq (\mathbb{M},\omega_{\bm{a}}) \times \mathbf{B}$ the canonical projection $\pi_{\bm{a}} \colon \mathbb{F}_{\bm{a}} \to \mathbf{B}$ defines a trivial symplectic fibration which constitutes the total space of the nonautonomous Hamiltonian system. The system itself is encoded in the horizontal $1$-form
\begin{equation}\label{eq:slims}
 \varpi = \sum_{i \in I} H_i \, {\rm d}t_i \in \Omega^0(\mathbb{F}_{\bm{a}},\pi_{\bm{a}}^*T^*\mathbf{B}),
\end{equation}
which also admits an intrinsic description (see~\cite[equation~(5.3)]{boalch_2012_simply_laced_isomonodromy_systems} and below for a coordinate independent formula, which is not needed here).

\subsection{Moduli space of meromorphic connections on the sphere}\label{sec:moduli_spaces}

To understand the relation of this Hamiltonian system with the moduli space of meromorphic connections one must first set the stage for the Hamiltonian reduction of $\mathbb{M}$.

The group $\widehat{H} \coloneqq \prod_{i \in I} \GL(V_i)$ acts on $(\mathbb{M},\omega_{\bm{a}})$ by simultaneous conjugation in Hamiltonian fashion, and one can take the symplectic reduction $\mathbb{M} \sslash_{\breve{\mathcal{O}}} \widehat{H}$ at an adjoint orbit
\begin{equation*}
 \breve{\mathcal{O}} \subseteq \mathfrak{h} \coloneqq \Lie(\widehat{H}) = \prod_{i \in I} \mathfrak{gl}(V_i),
\end{equation*}
identifying adjoint orbits with coadjoint ones via the symmetric nondegenerate $\widehat{H}$-invariant pairing $\mathfrak{h} \otimes \mathfrak{h} \to \mathbb{C}$ provided by the trace.

This in turn gives a finite-dimensional presentation for a moduli space of meromorphic connections with irregular singularities on the sphere, as follows. Let $H \coloneqq \prod_{i \in I \setminus I^{\infty}} \GL(V_i)$, subgroup of $\widehat{H}$, and choose adjoint orbits $\breve{\mathcal{O}}_H \in \Lie(H)$ and $\mathcal{O}_i \in \End(U^{\infty})$ for $i \in I^{\infty}$. Set also $\mathcal{O} \coloneqq \{\mathcal{O}_i\}_{i \in I^{\infty}}$, and fix a semisimple endomorphism $\widehat{T}$ corresponding to a point of $\mathbf{B}$ as in~\eqref{eq:isomonodromy_times}.

\begin{Definition}\label{def:wild_moduli_spaces}The space $\mathcal{M}_{\dR}^* = \mathcal{M}^*_{\dR}\big(\widehat{T},\breve{\mathcal{O}}_H,\mathcal{O}\big)$ is the moduli space of isomorphism classes of connections~\eqref{eq:meromorphic_connections} admitting local normal forms
\begin{equation*}
 {\rm d} - \frac{\Lambda_i}{z_i}{\rm d}z_i + \text{holomorphic terms}
\end{equation*}
around $t^{\infty}_i \in \mathbb{C}$, for some $\Lambda_i \in \mathcal{O}_i$ and some local coordinate $z_i$ vanishing at $t^{\infty}_i$, and
\begin{equation*}
 {\rm d} - \left(\frac{A}{w^3} + \frac{T}{w^2} + \frac{\Lambda}{w}\right){\rm d}w + \text{holomorphic terms}
\end{equation*}
around infinity, for some $\Lambda \in \breve{\mathcal{O}}_H$.
\end{Definition}

Conceptually the point of the base space $\mathbf{B}$ singles out a \textit{wild Riemann surface} structure
\begin{equation}\label{eq:wild_riemann_surface_structure}
 \bm{\Sigma} = \bm{\Sigma}\big(\widehat{T}\big) = \big(\mathbb{C}P^1,\{\infty,t^{\infty}_i\}_{i \in I^{\infty}},\{Q_{\infty},Q_i\}_{i \in I^{\infty}}\big)
\end{equation}
on the Riemann sphere, with $Q_{\infty} = \frac{Az^2}{2} + Tz$ and $Q_i = 0$ for $i \in I^{\infty}$, and where~$z$ is a~holo\-morphic coordinate identifying $\mathbb{C}P^1 \cong \mathbb{C} \cup \{\infty\}$. Recall that a~wild Riemann surface is the data of a~Riemann surface with marked points and irregular types at those points. To define them in the case at hand let~$m$ be the dimension of $U^{\infty}$ and choose a basis so that $\GL(U^{\infty}) \cong \GL_m(\mathbb{C})$.

\begin{Definition}
\label{def:irregular_type}

An unramified irregular type at the point $p \in \mathbb{C}P^1$ for the group $\GL_m(\mathbb{C})$ is an element $Q \in \mathfrak{t} (\!( z_p )\!) \big\slash \mathfrak{t} \llbracket z_p \rrbracket$, where $z_p$ is a local holomorphic coordinate vanishing at $p$ and $\mathfrak{t} \subseteq \mathfrak{gl}_m(\mathbb{C})$ the standard Cartan subalgebra of diagonal matrices.
\end{Definition}
This means $Q$ is the germ of a $\mathfrak{t}$-valued meromorphic function around $p$, defined up to holomorphic terms:
\begin{equation*}
 Q = \sum_{j = 1}^{k_p} T^p_j z_p^{-j},
\end{equation*}
where $T^p_j \in \mathfrak{t}$ for all $j \in \{1, \dotsc, k_p\}$. Then for all choice of $\Lambda \in \mathfrak{gl}_m(\mathbb{C})$ the differential operator ${\rm d} - \frac{\Lambda}{z_p}{\rm d}z_p + dQ$ is the germ of a meromorphic connection defined on the trivial holomorphic vector bundle $\mathbb{C}^m \times \mathbb{C}P^1 \to \mathbb{C}P^1$, having a pole of order $k_p + 1$ at~$p$ with residue $\Lambda$.

\begin{Remark}Definition~\eqref{def:irregular_type} is given in the case of a compact Riemann surface of genus zero since this is what is needed in this paper; this definition however extends verbatim to higher genera. See instead~\cite{boalch_2011_riemann_hilbert_for_tame_complex_parahoric_connections, boalch_yamakawa_2015_twisted_wild_character_varieties} for a coordinate-free generalisation of this notion to other complex reductive groups $G$, and to \textit{ramified} irregular types which are not conjugated to elements in the standard Cartan subalgebra $\mathfrak{t}(\!( z_p )\!) \subseteq \Lie(G) (\!( z_p )\!)$.
\end{Remark}

Hence one considers here the Riemann sphere with marked points at $z = \infty$ and $z = t^{\infty}_i$, and the only nonvanishing irregular type is put at $\infty$. Then the spectrum of the semisimple endomorphism $T^{\infty}$ selects the position of the simple poles in the finite part~-- the regular isomonodromy times~--, whereas that of $T$ selects the irregular type at infinity~-- the irregular isomonodromy times. Thus $\mathbf{B}$ is a space of admissible deformations of wild Riemann surface structures~\eqref{eq:wild_riemann_surface_structure} on~$\mathbb{C}P^1$, generalising the space of variations of pointed Riemann surface structures on the sphere, and even further the moduli space of deformation of ordinary Riemann surface structures, which is trivial in genus zero.

There now exists a symplectic fibration $\widetilde{\mathcal{M}}^*_{\dR} = \mathcal{M}^*_{\dR}\big(\bullet,\breve{\mathcal{O}}_H,\mathcal{O}\big) \to \mathbf{B}$ whose fibre over the wild Riemann surface $\bm{\Sigma}$ of~\eqref{eq:wild_riemann_surface_structure} is the moduli space of Definition~\ref{def:wild_moduli_spaces}. Isomonodromic families of meromorphic connections inside this fibration define the leaves of an integrable nonlinear/Ehresmann symplectic connection: the \textit{isomonodromy connection}.

The symplectic geometry of these isomonodromic deformations admits in this case a Hamiltonian interpretation, because of the existence of a preferred global trivialisation of the bundle $\widetilde{\mathcal{M}}^*_{\dR} \to \mathbf{B}$.

\begin{Theorem}[\cite{boalch_2012_simply_laced_isomonodromy_systems}]\label{theorem:isomorphism_weyl_meromorphic}Choose a wild Riemann surface structure $\bm{\Sigma}$ on $\mathbb{C}P^1$, as in~\eqref{eq:wild_riemann_surface_structure}. One can match up the choice of $\breve{\mathcal{O}} \in \mathfrak{h}$ with a choice of $\mathcal{O} \in \End(U^{\infty})^{I^{\infty}}$ and $\breve{\mathcal{O}}_H \in \Lie(H)$ so that there is an identification
\begin{equation*}
 \mathbb{M} \sslash_{\breve{\mathcal{O}}} \widehat{H} \cong \mathcal{M}^*_{\dR}\big(\bm{\Sigma},\breve{\mathcal{O}}_H,\mathcal{O}\big),
\end{equation*}
of symplectic algebraic varieties.

Moreover, stable points of $\mathbb{M}$ for the base-changing $\widehat{H}$-action correspond to stable connections, i.e., connections with no proper subconnections living on trivial holomorphic vector bundles. Restricting to the stable locus, the above identification becomes an isomorphism of holomorphic symplectic manifolds.
\end{Theorem}

The isomorphism of Theorem~\ref{theorem:isomorphism_weyl_meromorphic} yields a global trivialisation of the bundle $\widetilde{\mathcal{M}}^*_{\dR} \to \mathbf{B}$, since the symplectic reduction of the space $\mathbb{M}$ is independent of the base point in $\mathbf{B}$. Equivalently, associating to each point of $\big(\mathbb{M} \sslash_{\breve{\mathcal{O}}} \widehat{H}\big) \times \mathbf{B}$ the isomorphism class of the meromorphic connection~\eqref{eq:meromorphic_connections} yields an identification $\varphi \colon \big(\mathbb{M} \sslash_{\breve{\mathcal{O}}} \widehat{H}\big) \times \mathbf{B} \to \widetilde{\mathcal{M}}^*_{\dR}$.

Furthermore, the bundle $\big(\mathbb{M} \sslash_{\breve{\mathcal{O}}} \widehat{H}\big) \times \mathbf{B} \to \mathbf{B}$ carries an integrable nonautonomous Hamiltonian system: the Hamiltonian reduction of the simply-laced isomonodromy system, which is shown to be $\widehat{H}$-invariant (cf.\ Section~\ref{sec:slqc}). The integral manifolds of this system define a flat symplectic Ehresmann connection in the trivial symplectic bundle, and Theorem~\ref{theorem:slims} states that $\varphi$ is a flat isomorphism: the push-forward of the reduced simply-laced isomonodromy system along~$\varphi$ yields a nonautonomous Hamiltonian system which integrates the isomonodromy connection on $\widetilde{\mathcal{M}}^*_{\dR} \to \mathbf{B}$.

\section{Classical Hamiltonian reduction}\label{sec:classical_hamiltonian_reduction}

To state the invariance of the reduced simply-laced Hamiltonians the algebraic formalism of (classical) Hamiltonian reduction will be recalled in this section (see~\cite{etingof_2007_calogero_moser_systems_and_representation_theory}).

\subsection{General theory}\label{sec:general_theory_classical_reduction}

Consider a commutative Poisson algebra $(B_0,\{\cdot,\cdot\})$ equipped with an Hamiltonian action of a~Lie algebra $\mathfrak{g}$. There is then a morphism $\rho \colon \mathfrak{g} \to \Der(B_0)$ of Lie algebras (the action) together with a~lift $\mu^* \colon \Sym(\mathfrak{g}) \to B_0$ (the comoment map) through the adjoint action of $B_0$ on itself:
\begin{equation*}
 \rho(x).b = \{\mu^*(x),b\}, \qquad \text{for} \quad x \in \mathfrak{g}, \ b \in B_0.
\end{equation*}

\begin{Remark}The action is uniquely determined by the comoment map. Moreover the action is equivalently given by a Poisson morphism $\rho \colon \Sym(\mathfrak{g}) \to \Der(B_0)$, where $\Sym(\mathfrak{g}) \cong \mathbb{C}[\mathfrak{g}^*]$ is equipped with the Poisson--Lie bracket.
\end{Remark}

Choose now an ideal $\mathcal{I} \subseteq \Sym(\mathfrak{g})$.

\begin{Definition}\label{def:classical_reduction}The classical Hamiltonian reduction of $B_0$ with respect to the comoment map~$\mu^*$ and the ideal $\mathcal{I}$ is the quotient ring:
\begin{equation*}
 R(B_0,\mathfrak{g},\mathcal{I}) \coloneqq B_0^{\mathfrak{g}} \big\slash \mathcal{J}^{\mathfrak{g}},
\end{equation*}
where $B_0^{\mathfrak{g}} \subseteq B_0$ is the ring $\mathfrak{g}$-invariants, $\mathcal{J} \coloneqq B_0 \mu^*(\mathcal{I}) \subseteq B_0$ is the ideal generated by $\mu^*(\mathcal{I})$, and $\mathcal{J}^{\mathfrak{g}} \coloneqq \mathcal{J} \cap B_0^{\mathfrak{g}}$.
\end{Definition}

Note one can show that $\mathcal{J}^{\mathfrak{g}} \subseteq B_0^{\mathfrak{g}}$ is a Poisson ideal, and thus the reduction in Definition~\ref{def:classical_reduction} is canonically a Poisson algebra.

\begin{Remark}[geometric viewpoint]\label{remark:geometric_symplectic_reduction}
This is the algebraic counterpart of the usual Marsden--Weinstein reduction of a symplectic manifold.

Suppose indeed that $B_0$ is the Poisson algebra of functions on a symplectic manifold $M$, and that the Lie group $G$ acts on $M$ with moment map $\mu \colon M \to \mathfrak{g}^*$, where $\mathfrak{g} \coloneqq \Lie(G)$. Then there is a comoment map $\mu^* \colon \mathfrak{g} \to B_0$ as above, and if $\mathcal{O} \subseteq \mathfrak{g}^*$ is a coadjoint orbit then one considers the ideal of regular functions on $\mathfrak{g}^*$ vanishing on the coadjoint orbit:
\begin{equation*}
 \mathcal{I}_{\mathcal{O}} \coloneqq \{x \in \Sym(\mathfrak{g}) \,|\, x|_{\mathcal{O}} = 0 \}.
\end{equation*}

If $\mathcal{J}_{\mathcal{O}} \coloneqq B_0\mu^*(\mathcal{I}_{\mathcal{O}})$ then the quotient $B_0 \big\slash \mathcal{J}_{\mathcal{O}}$ is the ring of functions on the level set $\mu^{-1}(\mathcal{O})$. Taking $G$-invariant parts is the same as taking $\mathfrak{g}$-invariant parts: a function is fixed under the pull-back along the $G$-action if and only if it is annihilated by the infinitesimal $\mathfrak{g}$-action by vector fields. Hence the invariant ring $(B_0 \big\slash \mathcal{J}_{\mathcal{O}})^{\mathfrak{g}}$ is canonically the ring of functions on the quotient
\begin{equation*}
 \mu^{-1}(\mathcal{O}) \big\slash G = M \sslash_{\mathcal{O}} G.
\end{equation*}
Finally, if $\mathfrak{g}$ is reductive then $(B_0 \big\slash \mathcal{J}_{\mathcal{O}})^{\mathfrak{g}} \cong R(B_0,\mathfrak{g},\mathcal{I}_{\mathcal{O}})$.
\end{Remark}

Suppose now to have two sets of data $(B_0,\mathfrak{g},\mathcal{I})$ and $(B'_0,\mathfrak{g}',\mathcal{I}')$ defining classical Hamiltonian reductions, and let $\varphi \colon B_0 \to B'_0$ be a ring morphism such that $\varphi(B_0^{\mathfrak{g}}) \subseteq B_0'^{\mathfrak{g}'}$ and $\varphi(\mathcal{J}) \subseteq \mathcal{J}'$. Then by the universal property of quotients there exists a unique reduced morphism closing the following diagram:
\[
\begin{tikzpicture}
 \node (a) at (0,0) {$B_0^{\mathfrak{g}}$};
 \node (b) at (3.6,0) {$B_0'^{\mathfrak{g'}}$};
 \node (c) at (0,-2) {$R(B_0,\mathfrak{g},\mathcal{I})$};
 \node (d) at (3.6,-2) {$R(B_0',\mathfrak{g}',\mathcal{I}')$,};
 \node (e) at (1.8,-1) {$\circlearrowleft$};
\path
 (a) edge node[above]{$\varphi$} (b)
 (a) edge node[left]{$\pi_{\mathcal{I}}$} (c)
 (b) edge node[right]{$\pi_{\mathcal{I}'}$} (d)
 (c) edge [dashed] node[below]{$R\varphi$} (d);
\end{tikzpicture}
\]
where $\pi_{\mathcal{I}}$ and are the canonical projections. Further, if $\varphi$ is Poisson then so is $R\varphi$, since by definition the Poisson bracket of the Hamiltonian reduction is the quotient one.

This situation applies in particular to the case of Remark~\ref{remark:geometric_symplectic_reduction}. Assume that $B_0$, $B_0'$ are the algebras of functions on symplectic manifolds $M$, $M'$ equipped with moment maps $\mu \colon M \to \mathfrak{g}^*$, $\mu' \colon M' \to (\mathfrak{g}')^*$ for Hamiltonian actions of reductive Lie groups $G$, $G'$ with Lie algebras $\mathfrak{g}$, $\mathfrak{g}'$, respectively. Choose coadjoint orbits $\mathcal{O} \subseteq \mathfrak{g}^*$, $\mathcal{O}' \subseteq (\mathfrak{g}')^*$ and define the ideals $\mathcal{I}_{\mathcal{O}}$, $\mathcal{I}_{\mathcal{O}'}$, $\mathcal{J}_{\mathcal{O}}$ and~$\mathcal{J}_{\mathcal{O}'}$ as in Remark~\ref{remark:geometric_symplectic_reduction}.

\begin{Lemma}\label{lemma:reduction_classical_morphisms}If $\varphi = f^*$ is the pull-back along a smooth symplectic map $f \colon M' \to M$ sending $G'$-orbits inside $G$-orbits and such that $f\big((\mu')^{-1}(\mathcal{O}')\big) \subseteq \mu^{-1}(\mathcal{O})$ then the reduced morphism $R\varphi$ is well defined. Moreover, $R\varphi$ coincides the pull-back along the smooth symplectic map
\begin{equation*}
 Rf \colon \ M' \sslash_{\mathcal{O}'} G' \longrightarrow M \sslash_{\mathcal{O}} G
\end{equation*}
induced on the symplectic reductions in the identifications
\begin{equation*}
 R(B_0,\mathfrak{g},\mathcal{I}_{\mathcal{O}}) \cong \mathbb{C}\big[M \sslash_{\mathcal{O}} G\big], \qquad R(B_0',\mathfrak{g}',\mathcal{I}_{\mathcal{O}'}) \cong \mathbb{C}\big[M' \sslash_{\mathcal{O}'} G' \big].
\end{equation*}
Hence in brief $R(f^*) = (Rf)^*$.
\end{Lemma}

\begin{proof}
To see that $\varphi = f^*$ sends $B_0^{\mathfrak{g}}$ to $B_0'^{\mathfrak{g}'}$, let $b \in B_0^{\mathfrak{g}}$ be a function which is constant on $G$-orbits. Then for all $x \in M'$ and $g \in G'$ one has
\begin{equation*}
 \varphi (b) (g.x) = b\big(f(g.x)\big) = b\big(f(x)\big) = \varphi (b)(x),
\end{equation*}
using that $f(G'.x) \subseteq G.f(x)$ and that $b$ is constant on $G.f(x) \subseteq M$. This proves that $\varphi(b)$ lies in the ring of $G'$-invariants.

Let then $x \in \mathcal{I}_{\mathcal{O}}$, so that $\left.x\right\vert_{\mathcal{O}} = 0$. Then $\mu^*x$ vanishes on $\mu^{-1}(\mathcal{O})$, and $\varphi\mu^*(x) = f^*\mu^*(x)$ vanishes on $f^{-1}\big(\mu^{-1}(\mathcal{O})\big)$. By hypothesis $(\mu')^{-1}(\mathcal{O}') \subseteq f^{-1}\big(\mu^{-1}(\mathcal{O})\big)$, and thus $\varphi(\mu^*(x)) \in \mathcal{J}_{\mathcal{O'}}$. Since $x$ was arbitrary, and since by definition the set $\{\mu^*(x)\}_{x \in \mathcal{I}_{\mathcal{O}}}$ generates $\mathcal{J}_{\mathcal{O}}$, one sees that $\varphi(\mathcal{J}_{\mathcal{O}}) \subseteq \mathcal{J}_{\mathcal{O}'}$.

Finally, the reduced map $Rf$ is (well) defined by $Rf(G'.x) = G.f(x)$ for all $x \in (\mu')^{-1}(\mathcal{O}')$, and by construction the reduced morphism $R\varphi$ acts on the class of a function by pulling back a~representative along $f$. Thus indeed $R(f^*) = (Rf)^*$.
\end{proof}

\subsection{Classical reduction of symplectic quiver varieties}

We now apply the material of Section~\ref{sec:general_theory_classical_reduction} to the $\SL_2(\mathbb{C})$-action of Section~\ref{sec:action_on_algebras}, after introducing some insightful graph-theoretic notation.

Consider as in Definition~\ref{def:slims} the data of a partitioned set $I = \coprod_{j \in J} I^j$, an $I$-graded finite-dimensional vector space $V = \bigoplus_{i \in I} V_i$ and an embedding $\bm{a} \colon J \hookrightarrow \mathbb{C} \cup \{\infty\}$. Let then $\mathcal{G}$ be the complete $k$-partite graph on nodes $I$, where $k \coloneqq \lvert J \rvert$. This means $\mathcal{G}$ has exactly one edge between every pair of nodes lying in different parts of $I$, so that in particular it is simply-laced (without double edges or loop edges). Replacing each edge with a pair of opposite arrows yields a~quiver, also denoted by $\mathcal{G}$, which plays a~central role in what follows.

\begin{Remark}The results of this section could be extended to an arbitrary simply-laced quiver, not necessarily $k$-parted. We will consider this type of generality only for quantum Hamiltonian reduction in Section~\ref{sec:quantum_reduction_quiver_varieties}.
\end{Remark}

Now by definition $\mathbb{M} = \bigoplus_{i \neq j \in I} \Hom(V_i,V_j)$ is the space of representations of $\mathcal{G}$ in $V$:
\begin{equation*}
 \mathbb{M} \cong \Rep(\mathcal{G},V) \coloneqq \bigoplus_{\alpha \in \mathcal{G}_1} \Hom(V_{s(\alpha)},V_{t(\alpha)}),
\end{equation*}
where $\mathcal{G}_1$ is the set of arrows of $\mathcal{G}$ and $s,t \colon \mathcal{G}_1 \to I$ the source and target maps with values in the set of nodes, respectively.

The embedding $\bm{a} \in \mathbf{A}$ yields a symplectic form~\eqref{eq:symplectic_form_alternative} on $\mathbb{M}$ which can be written in terms of the adjacency of the quiver $\mathcal{G}$ after introducing some further notation. For $\alpha \in \mathcal{G}_1$ let $\alpha^*$ be the arrow opposite to $\alpha$ in $\mathcal{G}$, and denote $B_{\alpha} \colon V_{s(\alpha)} \to V_{t(\alpha)}$ the linear map defined by a~representation. The embedding $\bm{a}$ extends to a map $\bm{a} \colon I \to \mathbb{C} \cup \{\infty\}$ by $a_i \coloneqq a_j$ for $j \in J$ and $i \in I^j$, and we define functions $\varepsilon_{\alpha} \colon \mathbf{A} \to \mathbb{C}^{\times}$ by $\varepsilon_{\alpha} + \varepsilon_{\alpha^*} = 0$, and following~\eqref{eq:symplectic_form_weights}:
\begin{equation*}
 \varepsilon_{\alpha}(\bm{a}) \coloneqq
\begin{cases}
 \dfrac{1}{a_i - a_j},   & i = s(\alpha), \  j = t(\alpha) \not\in I^{\infty}, \\
 1,   & s(\alpha) \in I^{\infty}.
\end{cases}
\end{equation*}
With this notation introduced one has
\begin{equation}\label{eq:symplectic_form_quiver}
 \omega_{\bm{a}} = \sum_{\alpha \in \mathcal{G}_1} \frac{\varepsilon_{\alpha}(\bm{a})}{2} \Tr \big({\rm d}B_{\alpha} \wedge {\rm d}B_{\alpha^*}\big).
\end{equation}

Now there is a moment map $\mu_{\bm{a}} \colon \mathbb{M} \to \mathfrak{h}^* \cong \mathfrak{h}$ for the $\widehat{H}$-action by simultaneous base changing~-- using the componentwise trace duality to identify $\mathfrak{h}^*$ with $\mathfrak{h} = \bigoplus_{i \in I} \mathfrak{gl}(V_i)$. Writing the symplectic form as in~\eqref{eq:symplectic_form_quiver} the $\bm{a}$-dependent moment map admits the following formula (see, e.g.,~\cite[Theorem~10.10]{kirillov_2016_quiver_representations_and_quiver_varieties}):
\begin{equation*}
 \mu_{\bm{a}}(B) = \bigoplus_{i \in I} \left(\sum_{\alpha \in t^{-1}(i)} \varepsilon_{\alpha}(\bm{a}) B_{\alpha}B_{\alpha^*}\right), \qquad \text{where} \quad B = (B_{\alpha})_{\alpha \in \mathcal{G}_1}.
\end{equation*}
This defines a map of Poisson bundles $\mu \colon \widetilde{\mathbb{M}} \to \mathfrak{h} \times \mathbf{A}$, as $\mu_{\bm{a}}$ is a Poisson map for $\bm{a} \in \mathbf{A}$ where~$\mathfrak{h}$ is given the Poisson structure coming from $\mathfrak{h}^*$ under the trace duality.

This bundle-theoretic moment map is compatible with the $\SL_2(\mathbb{C})$-actions on $\widetilde{\mathbb{M}}$ and $\mathbf{A}$.

\begin{Lemma}\label{lemma:compatibility_moment_classical_action}
One has $\mu_{\bm{a}.g} \circ \varphi_g(\bm{a}) = \mu_{\bm{a}}$ for $\bm{a} \in \mathbf{A}$ and $g \in \SL_2(\mathbb{C})$.
\end{Lemma}

\begin{proof}Looking at~\eqref{eq:action_on_fibration} one sees that
\begin{equation*}
 \varphi_g(\bm{a}) \colon \ B_{\alpha} \longmapsto \eta_{t(\alpha)}(g,\bm{a}) \cdot B_{\alpha},
\end{equation*}
for $\alpha \in \mathcal{G}_1$, where $\eta_{t(\alpha)} \coloneqq \eta_j$ for $j \in J$ and $t(\alpha) \in I^j$.\footnote{Indeed by~\eqref{eq:action_on_fibration} the action of $\varphi_g(\bm{a})$ only depends on the coarser decomposition $V = \bigoplus_{j \in J} W^j$.} Hence
\begin{equation*}
 \mu_{\bm{a}.g} \circ \varphi_g(\bm{a}) \colon \ B \longmapsto \bigoplus_{i \in I} \left(\sum_{\alpha \in t^{-1}(i)} \varepsilon_{\alpha}(\bm{a}.g) \cdot \eta_{t(\alpha)}(g,\bm{a}) \eta_{s(\alpha)}(\bm{a}.g) \cdot B_{\alpha}B_{\alpha^*}\right).
\end{equation*}
Then the conclusion follows from Lemma~\ref{lemma:action_cocycle_symplectic}, replacing the elements $i,j \in J$ with the nodes $s(\alpha), t(\alpha) \in I$.
\end{proof}

Lemma~\ref{lemma:compatibility_moment_classical_action} can be summarised by stating that the $\SL_2(\mathbb{C})$-action ``preserves'' the values of the moment map (cf.~\cite[Proposition~7.5]{boalch_2012_simply_laced_isomonodromy_systems}).

The last ingredient needed for the reduction of the classical $\SL_2(\mathbb{C})$-action is its compatibility with the fibrewise $\widehat{H}$-action on $\widetilde{\mathbb{M}}$: since the $\SL_2(\mathbb{C})$-action does not depend on a choice of basis for $V$ the map $\varphi_g(\bm{a})$ commutes with the $\widehat{H}$-action. This implies that the symplectomorphism $\varphi_g(\bm{a}) \colon (\mathbb{M},\omega_{\bm{a}}) \to (\mathbb{M},\omega_{\bm{a}.g})$ sends $\widehat{H}$-orbits to $\widehat{H}$-orbits.

If further an orbit $\breve{\mathcal{O}} \subseteq \mathfrak{h}$ is chosen then Lemma~\ref{lemma:compatibility_moment_classical_action} yields $\varphi_g(\bm{a})\big(\mu_{\bm{a}}^{-1}(\breve{\mathcal{O}})\big) \subseteq \mu_{\bm{a}.g}^{-1}(\breve{\mathcal{O}})$, since $\mu_{\bm{a}}^{-1} = \varphi_g(\bm{a})^{-1} \circ \mu_{\bm{a}.g}^{-1}$. Hence by Lemma~\ref{lemma:reduction_classical_morphisms} there exists a reduced Poisson morphism
\begin{equation*}
 R\varphi^*_g(\bm{a}) \colon \ R(A_0,\{\cdot,\cdot\}_{\bm{a}.g}) \longrightarrow R(A_0,\{\cdot,\cdot\}_{\bm{a}}),
\end{equation*}
keeping track of the Poisson bracket on the classical algebras in the notation~-- but omitting the Lie algebra and the ideal $\mathcal{I}_{\breve{\mathcal{O}}} \subseteq \Sym(\mathfrak{h})$. This morphism coincide with the pull-back along the symplectomorphism
\begin{equation*}
 R\varphi_g(\bm{a}) \colon \ (\mathbb{M},\omega_{\bm{a}}) \sslash_{\breve{\mathcal{O}}} \widehat{H} \longrightarrow (\mathbb{M},\omega_{\bm{a}.g}) \sslash_{\breve{\mathcal{O}}} \widehat{H}
\end{equation*}
induced on the symplectic reductions, keeping track of the symplectic forms in the notation.

Moreover, since $g \in \SL_2(\mathbb{C})$ and $\bm{a} \in \mathbf{A}$ are arbitrary it follows that the morphisms $R\varphi^*_{gg'}(\bm{a})$ and $R\varphi^*_g(\bm{a}) \circ R\varphi^*_{g'}(\bm{a}.g)$ are defined, and by Proposition~\ref{prop:classical_action} they coincide since they close the same commutative diagram. In particular the morphism $R\varphi^*_g(\bm{a})$ is an isomorphism with inverse $R\varphi^*_{g^{-1}}(\bm{a}.g)$, as seen by taking $g' = g^{-1}$.

Hence we have defined the classical Hamiltonian reduction of the Poisson action of Proposition~\ref{prop:classical_action}. One may consider the bundle of commutative Poisson algebras $R\big(\mathcal{A}_0,\breve{\mathcal{O}}\big) \to \mathbf{A}$ whose fibre over the embedding $\bm{a}$ is the classical Hamiltonian reduction of $(A_0,\{\cdot,\cdot\}_{\bm{a}})$ at the orbit $\breve{\mathcal{O}} \subseteq \Sym(\mathfrak{h})$. It is given as a trivial bundle of graded commutative algebras with a Poisson structure that depends on the point on the base, and the assignment $(g,\bm{a}) \mapsto R\varphi^*_g(\bm{a})$ lifts the action on the base to an $\SL_2(\mathbb{C})$-action on the total space.

\subsection{Classical invariance}
\label{sec:classical_invariance}

Since the simply-laced Hamiltonians $H_i \colon \mathbb{M} \times \mathbf{B} \to \mathbb{C}$ are $\widehat{H}$-invariant, their fibrewise restrictions to $\mathbb{M}$ live in $A_0^{\mathfrak{h}}$, and their canonical projection $\pi_{\breve{\mathcal{O}}} \colon A_0^{\mathfrak{h}} \to R\big(A_0,\mathfrak{h},\mathcal{I}_{\breve{\mathcal{O}}}\big)$ can be taken.

\begin{Definition}
The element $RH_i \coloneqq \pi_{\breve{\mathcal{O}}}(H_i)$ is the reduced simply-laced Hamiltonian at the node $i \in I$.
\end{Definition}

The invariance under the classical $\SL_2(\mathbb{C})$-action is stated as follows.

\begin{Theorem}[{\cite[Corollary~9.4]{boalch_2012_simply_laced_isomonodromy_systems}}]\label{theorem:classical_invariance}
For all $i \in I$ and all $g \in \SL_2(\mathbb{C})$ there exists a constant $c_i \in \mathbb{C}$ such that the reduced Hamiltonian at the node $i$ transforms as
\begin{equation}\label{eq:classical_shift}
 R\varphi^*_g(\bm{a}) RH_i = RH_i + c_i.
\end{equation}
In particular the reduced isomonodromy equations are invariant under the $\SL_2(\mathbb{C})$-action.
\end{Theorem}

\begin{Remark}The second statement follows from the fact that the reduced isomonodromy equations are the dynamical equations for the time-evolution of classical observables with respect to the reduced simply-laced Hamiltonians~-- in the Hamiltonian picture of motion in classical mechanics, where the classical state is fixed. These equations do not change under constant shifts.

More precisely, if $H$ is a local section of the bundle $R\big(A_0,\mathfrak{h},\mathcal{I}_{\breve{\mathcal{O}}}\big) \times \mathbf{B} \to \mathbf{B}$ at the point $\bm{a} \in \mathbf{A}$, then the reduced equations read
\begin{equation}\label{eq:reduced_classical_imd_equations}
 \partial_{t_i} H = \{RH_i,H\}_{\bm{a}},
\end{equation}
where one uses the reduced Poisson bracket on the Hamiltonian reduction. Consider then the extended bundle
\begin{equation*}
 R\big(\mathcal{A}_0,\breve{\mathcal{O}}\big) \times \mathbf{B} \longrightarrow \mathbf{A} \times \mathbf{B},
\end{equation*}
which is by construction trivial along the variations in~$\mathbf{B}$. If $H$ is a section of this bundle defined in a neighbourhood of the $\SL_2(\mathbb{C})$-orbit of $\bm{a} \in \mathbf{A}$, then we want to relate the isomonodromy equations~-- along the variations in $\mathbf{B}$~-- of $H$ and $\varphi^*_g(\bm{a}) H$ for $g \in \SL_2(\mathbb{C})$. But if $H$ is a solution of~\eqref{eq:reduced_classical_imd_equations} at $\bm{a}.g$ then Theorem~\ref{theorem:classical_invariance} yields
\begin{align}
 \partial_{t_i} (R\varphi^*_g(\bm{a})H)& = R\varphi^*_g(\bm{a}) (\partial_{t_i} H) = R\varphi^*_g(\bm{a}) \big\{RH_i,H\big\}_{\bm{a}.g} \nonumber\\
 &= \big\{R\varphi^*_g RH_i,R\varphi^*_g(\bm{a}) H \big\}_{\bm{a}} = \big\{RH_i,R\varphi^*_g(\bm{a}) H \big\}_{\bm{a}},\label{eq:classical_invariance}
\end{align}
because the $\SL_2(\mathbb{C})$-action does not depend on the point in the base space $\mathbf{B}$, using~\eqref{eq:classical_shift} for the last equality, and since $R\varphi^*_g(\bm{a})$ is a Poisson morphism. Hence $\varphi^*_g(\bm{a}) H$ is a solution of~\eqref{eq:reduced_classical_imd_equations} at $\bm{a} \in \mathbf{A}$, so that a single set of isomonodromy equation at one point controls the time evolution along the whole of the $\SL_2(\mathbb{C})$-orbit of $\bm{a}$.

Conversely, if $H$ is a solution of~\eqref{eq:reduced_classical_imd_equations} in a neighbourhood of $\bm{a} \in \mathbf{A}$ then dragging it along the $\SL_2(\mathbb{C})$-action extends it to a solution in a neighbourhood of the orbit of $\bm{a}$.
\end{Remark}

Geometrically this means that $R\varphi^*_g(\bm{a})$ is a flat isomorphism of vector bundles, as follows. For a fixed embedding $\bm{a} \in \mathbf{A}$ consider the trivial bundle $(A_0,\{\cdot,\cdot\}_{\bm{a}}) \times \mathbf{B} \to \mathbf{B}$ of Poisson algebras. It carries the (simply-laced) isomonodromy connection
\begin{equation*}
 \nabla_{\bm{a}} = {\rm d} - \varpi,
\end{equation*}
with $\varpi$ is as in~\eqref{eq:slims} and the simply-laced Hamiltonians act via their $\bm{a}$-dependent adjoint action,
i.e., via their Hamiltonian vector field $\{H_i,\cdot\}_{\bm{a}}$. The strong flatness of this connection~-- that is the identities ${\rm d}\varpi = 0 = [\varpi,\varpi]$~-- is equivalent to that of the simply-laced isomonodromy system.

Now take fibrewise classical Hamiltonian reduction at an orbit $\breve{\mathcal{O}} \subseteq \mathfrak{h}$ to get a new trivial bundle $R(A_0,\{\cdot,\cdot\}_{\bm{a}}) \times \mathbf{B} \to \mathbf{B}$ of Poisson algebras, keeping the notation for the Poisson bracket and dropping the Lie algebra and the ideal. It carries the reduced (simply-laced) isomonodromy connection $R\nabla_{\bm{a}}$, that is the flat connection defined by
\begin{equation*}
 R\nabla_{\bm{a}} \coloneqq {\rm d} - R\varpi, \qquad \text{where} \quad R\varpi \coloneqq \sum_{i \in I} RH_i \,{\rm d}t_i,
\end{equation*}
where the reduced simply-laced Hamiltonians act via their (reduced) Poisson bracket. Then~\eqref{eq:classical_invariance} shows that the map
\begin{equation*}
 R\varphi^*_g(\bm{a}) \colon \ R\big(A_0,\{\cdot,\cdot\}_{\bm{a}.g}\big) \times \mathbf{B} \longrightarrow R\big(A_0,\{\cdot,\cdot\}_{\bm{a}}\big) \times \mathbf{B}
\end{equation*}
is a flat isomorphism of vector bundles equipped with (flat) connections.

In Section~\ref{sec:quantum_invariance} we will provide a quantum analogue of this statement.

\section{Universal simply-laced quantum connection}
\label{sec:slqc}

In this section we recall the definition of the universal simply-laced quantum connection, for which the first step is recognising the classical Hamiltonians as traces of potentials on the quiver~$\mathcal{G}$ (see~\cite{rembado_2019_simply_laced_quantum_connections_generalising_kz}).

Let then $\mathbb{C}\mathcal{G}_{\cycl}$ be the vector space generated by oriented cycles in $\mathcal{G}$, defined up to cyclic permutations of their arrows. Elements of $\mathbb{C}\mathcal{G}_{\cycl}$ are called potentials, and one can take trace of them to define $\widehat{H}$-invariant functions on~$\mathbb{M}$. Namely, if $C = \alpha_n \cdots \alpha_1$ is a cycle with arrows $\alpha_i \in \mathcal{G}_1$ then a representation of~$\mathcal{G}$ defines an endomorphism $B_{\alpha_n} \cdots B_{\alpha_1} \in \End(V_{s(\alpha_1)})$, and one sets
\begin{equation*}
 \Tr(C) \coloneqq \Tr(B_{\alpha_n} \cdots B_{\alpha_1}) \colon \ \mathbb{M} \longrightarrow \mathbb{C},
\end{equation*}
which is extended by $\mathbb{C}$-linearity to a map $\Tr \colon \mathbb{C}\mathcal{G}_{\cycl} \to A_0$. The fact that $\Tr(C)$ is $\widehat{H}$-invariant follows from the fact that the trace is a~class function.

Now fix a node $i \in I$, and recall $H_i$ denotes the associated component of the simply-laced isomonodromy system of Definition~\ref{def:slims}.

\begin{Proposition}[{\cite[Section~4]{rembado_2019_simply_laced_quantum_connections_generalising_kz}}]
There exist a time-dependent potential $W_i \colon \mathbf{B} \to \mathbb{C}\mathcal{G}_{\cycl}$ such that $H_i = \Tr(W_i) \colon \mathbf{B} \to A_0$. Moreover, the potential $W_i$ is a $\mathbb{C}$-linear combination of the following four types of cycles in the quiver $\mathcal{G}$:
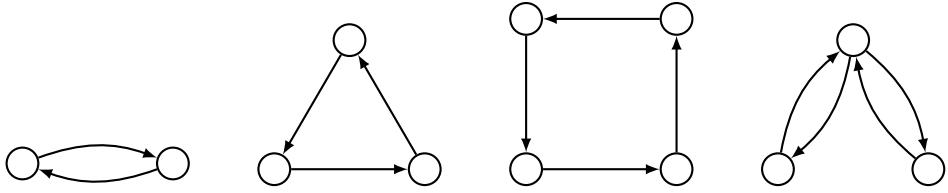
\begin{figure}[h]\centering
\begin{tikzpicture}
 \vertex (a) at (0,0) {};
 \vertex (b) [right of = a] {};
\path
 (a) edge [bend left = 20] (b)
 (b) edge [bend left = 20] (a);
\end{tikzpicture} \qquad \begin{tikzpicture}
			 \vertex (a) at (0,0) {};
			 \vertex (b) [right of = a] {};
			 \vertex (c) at (1,1.7172) {};
			 \path
			 (a) edge (b)
			 (b) edge (c)
			 (c) edge (a);
			 \end{tikzpicture} \qquad \begin{tikzpicture}
						 \vertex (a) at (0,0) {};
						 \vertex (b) [right of = a] {};
						 \vertex (c) [above of = b] {};
						 \vertex (d) [above of = a] {};
						 \path
						 (a) edge (b)
						 (b) edge (c)
						 (c) edge (d)
						 (d) edge (a);
						 \end{tikzpicture} \qquad
		\begin{tikzpicture}
		 \vertex (a) at (0,0) {};
		 \vertex (b) [right of = a] {};
		 \vertex (c) at (1,1.7172) {};
		 \path
		 (a) edge [bend left = 20] (c)
		 (c) edge [bend left = 20] (a)
		 (b) edge [bend left = 20] (c)
		 (c) edge [bend left = 20] (b);
		\end{tikzpicture}
\caption{Isomonodromy cycles. \label{fig:imd_cycles}}
\end{figure}
\end{Proposition}

In particular, since the invariant part $A_0^{\widehat{H}} = A_0^{\mathfrak{h}}$ coincides with the algebra generated by traces of cycles \cite[Theorem~1]{lebruyn_procesi_1990_semisimple_representations_of_quivers}, one sees that the simply-laced Hamiltonian descend to a time-dependent system on the symplectic fibration of moduli spaces of meromorphic connections, as stated in Section~\ref{sec:slims}.

\begin{Definition}The potential $W_i$ is the isomonodromy potential at the node $i \in I$. The cycles of Fig.~\ref{fig:imd_cycles} are the isomonodromy cycles; the rightmost cycle of Fig.~\ref{fig:imd_cycles} is a degenerate 4-cycle.
\end{Definition}

\subsection{Traces of quantum potentials}\label{sec:quantum_trace}

The quantisation of the simply-laced Hamiltonians is constructed as follows. Let $\mathcal{G}$ be the complete $k$-partite quiver of the previous section.

\begin{Definition}An \textit{anchored cycle} in $\mathcal{G}$ is as an oriented cycle in $\mathcal{G}$ with the choice of a starting arrow.
\end{Definition}

The goal is to define the trace of such a cycle in a given representation of $\mathcal{G}$ in $V = \bigoplus_{i \in I} V_i$, so that this trace lives in the Weyl algebra $A$ of Section~\ref{sec:filtered_quantisation}. Further, this will be upgraded to an $\hslash$-deformed taking values in the Rees algebra~\eqref{eq:rees_algebra}.

Choose then bases of $V_i$ for $i \in I$, or equivalently take a dimension vector $d = (d_i)_{i \in I} \in \mathbb{Z}_{\geq 0}^I$ for the quiver and let $V_i \coloneqq \mathbb{C}^{d_i}$.
Then the linear coordinates $B \longmapsto B_{\alpha}^{ij} \in \mathbb{C}$ of a representation
\begin{equation*}
 B = (B_{\alpha})_{\alpha \in \mathcal{G}_1} \in \bigoplus_{\alpha \in \mathcal{G}_1} \Hom\big(V_{s(\alpha)},V_{t(\alpha)}\big) = \mathbb{M}
\end{equation*}
are defined. Putting this linear functions together in a matrix yields elements inside
\begin{equation*}
 \Hom\big(V_{s(\alpha)},V_{t(\alpha)}\big) \otimes \mathbb{M}^* \subseteq \Hom(V_{s(\alpha)},V_{t(\alpha)}) \otimes A_0,
\end{equation*} and the coefficient-wise Weyl quantisation $\mathbb{M}^* \hookrightarrow A$ yields matrices with coefficients in the Weyl algebra, that is elements
\begin{equation*}
 \widehat{B}_{\alpha} \in \Hom\big(V_{s(\alpha)},V_{t(\alpha)}\big) \otimes A.
\end{equation*}

Now suppose $\widehat{C} = \alpha_n \cdots \underline{\alpha_1}$ is a cycle anchored at $\alpha_1$, where the anchor is underlined~-- composing arrows as linear maps, from right to left. Then the matrix product
\begin{equation*}
 \widehat{B}_{\alpha_n} \dotsm \widehat{B}_{\alpha_1} \in \End\big(V_{s(\alpha_1)}\big) \otimes A,
\end{equation*}
is well defined by
\begin{equation*}
 \big(\widehat{B}_{\alpha} \cdot \widehat{B}_{\beta}\big)_{kl} \coloneqq \sum_m \widehat{B}_{\alpha}^{km} \ast_{\bm{a}} \widehat{B}_{\beta}^{ml} \in A,
\end{equation*}
using the coefficient-wise product of $A$ for $\alpha, \beta \in \mathcal{G}_1$. Then one can take a trace:
\begin{equation}\label{eq:quantum_trace}
 \Tr_a(\widehat{C}) \coloneqq \Tr\big(\widehat{B}^{\alpha_n} \dotsm \widehat{B}^{\alpha_1}\big) \in A.
\end{equation}

\begin{Remark}It would be equivalent to do the following. Consider matrices with coefficients in the ring $\Tens(\mathbb{M}^*)$, define their product, take their trace (now an element of $\Tens(\mathbb{M}^*)$) and then compose with the canonical projection $\pi_1 \colon \Tens(\mathbb{M}^*) \to A$.
\end{Remark}

The trace~\eqref{eq:quantum_trace} is not invariant under all cyclic permutations, since the coefficients of~$\widehat{B}_{\alpha}$ and~$\widehat{B}_{\beta}$ need not commute inside~$A$. Indeed they commute if and only if~$\alpha$ is not the arrow opposite to $\beta$, which leads to the following definition.

\begin{Definition}An admissible permutation of the arrows of an anchored cycle $\widehat{C} = \alpha_n \cdots \underline{\alpha_1}$ consists of splitting the cycle in two paths $A_1 = \alpha_n \cdots \alpha_i$, $A_2 = \alpha_{i-1} \cdots \alpha_1$ such that no arrow of $A_1$ has its opposite in $A_2$, and swap them to obtain the new anchored cycle $C' = \alpha_{i-1} \cdots \alpha_1 \cdot \alpha_n \cdots \underline{\alpha_i}$ after concatenation.
\end{Definition}

\begin{Example}[admissible permutation for degenerate 4-cycles]\label{example:anchor_degenerate_4_cycles}If $\alpha$ and $\beta$ are arrows based at the same node, with opposite arrows $\alpha^*$ and $\beta^*$ respectively, then $C = \beta^*\beta\alpha^*\alpha$ is a degenerate $4$-cycle (as the rightmost cycle of Fig.~\ref{fig:imd_cycles}). Then $C' = \alpha^*\alpha\beta^*\beta$ is an admissible permutation, whereas $C'' = \alpha\beta^*\beta\alpha^*$ is not.
\end{Example}

Let $\widehat{\mathbb{C}\mathcal{G}}_{\cycl}$ be the complex vector space generated by anchored cycles, defined up to admissible permutations of their arrows. The elements of $\widehat{\mathbb{C}\mathcal{G}}_{\cycl}$ are called quantum potentials, with terminology suggested by the fact that~\eqref{eq:quantum_trace} defines a map $\Tr_{\bm{a}} \colon \widehat{\mathbb{C}\mathcal{G}}_{\cycl} \to A$, which is the analogue to $\Tr \colon \mathbb{C}\mathcal{G}_{\cycl} \to A_0$.

Moreover this notion of quantum trace admits an $\hslash$-deformed version given by
\begin{equation*}
 \Tr_{\bm{a},\hslash} \colon \ \widehat{\mathbb{C}\mathcal{G}}_{\cycl} \longrightarrow \Rees(A,\mathcal{B}) \subseteq \widehat{A}, \qquad \Tr_{\bm{a},\hslash}\big(\widehat{C}\big) \coloneqq \Tr_a\big(\widehat{C}\big) \cdot \hslash^{l(C)},
\end{equation*}
where $l(C) \geq 0$ is the length of the oriented cycle underlying $\widehat{C}$. Then forgetting the anchor provides a surjection $\sigma_{\mathcal{G}} \colon \widehat{\mathbb{C}\mathcal{G}}_{\cycl} \to \mathbb{C}\mathcal{G}_{\cycl}$, and the semiclassical limit $\sigma \colon \widehat{A} \to A_0$ of equation~\eqref{eq:semiclassical_limit} intertwines it with taking traces.

The upshot is that one can quantise linear combinations of traces of oriented cycles in $\mathcal{G}$ by taking suitable anchors for each cycle (and this construction generalises to any quiver).

To explain the choice which produces an integrable quantum system out of~\eqref{eq:slims} consider again an oriented cycle $C$ in $\mathcal{G}$. The cycle defines a subquiver $C =(C_0,C_1)$ of $\mathcal{G}$, and the (even) degree of each node $i \in C_0$ is defined as the number of arrows $e \in C_1$ adjacent to it:
\begin{equation*}
 \deg(i) \coloneqq \Card \{e \in C_1 \,|\, i \in \{s(e),t(e)\} \} \geq 2.
\end{equation*}
Let then
\begin{equation*}
 \max(C) \coloneqq \{ i \in C_0 \,|\, \deg(i) \text{ is maximal} \} \subseteq C_0
\end{equation*}
be the subset of nodes of maximal degree, set
\begin{equation*}
 m(C) \coloneqq \Card  (\max(C) )
\end{equation*}
to be the number of nodes of maximal degree, and finally
\begin{equation*}
 \deg(C) \coloneqq \max_{i \in C_0} \deg(i) = \max_{i \in \max(C)} \deg(i)
\end{equation*}
the maximal degree of a node in $C$.

\begin{Definition}\label{def:quantisation_of_potentials}The quantisation of the cycle $C$ is the quantum potential $\widehat{C} \in \widehat{\mathbb{C}\mathcal{G}}_{\cycl}$ defined by
\begin{equation}\label{eq:quantisation_of_potentials}
 \widehat{C} \coloneqq w_{C} \sum_{i \in \max(C)}\sum_{e \in s^{-1}(i)} \widehat{C}_e,
\end{equation}
up to admissible permutations, where $\widehat{C}_e$ denotes the cycle obtained by anchoring $C$ at $e$ and setting
\begin{equation*}
 w_{C} \coloneqq \frac{2}{\deg(C)m(C)} \in \mathbb{Q}_{> 0}.
\end{equation*}
\end{Definition}

\begin{Remark}In words one anchors $C$ at all possible arrows whose source is a node of maximal degree, takes the sum of that, and multiplies by the weight $w_{C}$. The number of arrows starting at a node of maximal degree is $\frac{\deg(C)}{2}$, so that $w_{C}^{-1} = \frac{\deg(C)m(C)}{2}$ is the overall number of arrows starting at nodes of maximal degree; thus dividing by it assures that $\sigma_{\mathcal{G}}\big(\widehat{C}\big) = C$.
\end{Remark}

Then Definition~\ref{def:quantisation_of_potentials} extends to the whole of $\mathbb{C}\mathcal{G}_{\cycl}$ by $\mathbb{C}$-linearity, so that there is a preferred section $\mathcal{Q}_{\mathcal{G}} \colon \mathbb{C}\mathcal{G}_{\cycl} \to \widehat{\mathbb{C}\mathcal{G}}_{\cycl}$ of the anchor forgetting map.

If in particular a cycle does not contain a pairs of opposite arrows then Definition~\ref{def:quantisation_of_potentials} reduces to take an anchor at any nodes, since all possible anchors are related by admissible permutations. This is the case of the middle cycles of Fig.~\ref{fig:imd_cycles}, whereas the 2-cycles and degenerate 4-cycles are explicitly quantised as follows.

\begin{Definition}[from Definition~\ref{def:quantisation_of_potentials}]\label{def:quantisation_imd_potentials}
The quantisation of a degenerate 4-cycles is the quantum cycle having the same underlying classical cycle, anchored at either of the two arrows coming out of its centre.\footnote{This follows from Example~\ref{example:anchor_degenerate_4_cycles}.} The quantisation of a two cycle $C = $~\begin{tikzpicture}[scale=0.8]
 \vertex (a) at (0,0) {};
 \vertex (b) at (1,0) {};
 \path
 (a) edge [bend left = 20] (b)
 (b) edge [bend left = 20] (a);
 \end{tikzpicture}
 is by definition the quantum potential
\[
 \widehat{C} = \frac{1}{2}\bigg($\begin{tikzpicture}
 \vertex (a) at (0,0) [fill = black] {};
 \vertex (b) at (1,0) {};
 \path
 (a) edge [bend left = 20] (b)
 (b) edge [bend left = 20] (a);
 \end{tikzpicture} $+$ \begin{tikzpicture}
 \vertex (a) at (0,0) {};
 \vertex (b) at (1,0)[fill = black] {};
 \path
 (a) edge [bend left = 20] (b)
 (b) edge [bend left = 20] (a);
 \end{tikzpicture}$\bigg),
\]
where the black nodes correspond to the source of the anchors.
\end{Definition}

Now if $C$ is an oriented cycle one can quantise $\Tr(C) \in A_0$ by
\begin{equation*}
 \mathcal{Q}_{\bm{a},\hslash} \Tr(C) \coloneqq \Tr_{\bm{a},\hslash} \big(\mathcal{Q}_{\mathcal{G}} (C)\big) \in \Rees(A,\mathcal{B}) \subseteq \widehat{A},
\end{equation*}
since indeed
\begin{equation*}
 \sigma \circ \mathcal{Q}_{\bm{a},\hslash} \big(\Tr(C)\big) = \Tr \big(\sigma_{\mathcal{G}} \circ \mathcal{Q}_{\mathcal{G}} (C)\big) = \Tr(C) \in A_0.
\end{equation*}
Hence $\mathcal{Q}_{\bm{a},\hslash}$ is extended by $\mathbb{C}$-linearity to a quantisation map defined on $\Tr(\mathbb{C}\mathcal{G}_{\cycl}) \subseteq A_0$, and a quantisation of the simply-laced isomonodromy Hamiltonian $H_i = \Tr(W_i) \in \mathbb{C}[\mathbf{B}] \otimes A_0$ is provided by
\begin{equation*}
 \widehat{H}_i \coloneqq \mathcal{Q}_{\bm{a},\hslash} \Tr(W_i) = \Tr_{\bm{a},\hslash} (\mathcal{Q}_{\mathcal{G}} \circ W_i) \in \mathbb{C}[\mathbf{B}] \otimes \Rees(A,\mathcal{B}) \subseteq \mathbb{C}[\mathbf{B}] \otimes \widehat{A}.
\end{equation*}

\begin{Theorem}[\cite{rembado_2019_simply_laced_quantum_connections_generalising_kz}]\label{theorem:slqc}
The time-dependent quantum Hamiltonian system $\widehat{H} \in \mathbb{C}[\mathbf{B}] \otimes \widehat{A}^I$ with components $\widehat{H}_i$ is strongly integrable, which means that
\begin{equation}\label{eq:quantum_flatness}
 \frac{\partial \widehat{H}_i}{\partial t_j} - \frac{\partial \widehat{H}_j}{\partial t_i} = 0 = \big[\widehat{H}_i,\widehat{H}_j\big], \qquad \text{for all } i,j \in I.
\end{equation}
\end{Theorem}

The final output is thus a bundle $\widehat{A} \times \mathbf{B} \to \mathbf{B}$ of noncommutative $\mathbb{C} \llbracket \hslash \rrbracket$-algebras equipped with a strongly flat connection $\widehat{\nabla} = {\rm d} - \widehat{\varpi}$, where $\widehat{\varpi} = \sum_{i \in I} \widehat{H}_i \,{\rm d}t_i$ and the quantum Hamiltonians act on $\widehat{A}$ by their adjoint action, i.e., their commutator. The strong flatness of the connection is expressed by the identities ${\rm d}\widehat{\varpi} = 0 = [\widehat{\varpi},\widehat{\varpi}]$, which are a compact version of~\eqref{eq:quantum_flatness}.

\begin{Definition}\label{def:slqc}The connection $\widehat{\nabla}$ is the universal simply-laced quantum connection attached to the data of the partitioned set $I = \coprod_{j \in J} I^j$, the $I$-graded space $V = \bigoplus_{i \in I} V_i$, and the embedding $\bm{a} \colon J \hookrightarrow \mathbb{C}^2 \cup \{\infty\}$ of the set of parts of $I$. The Hamiltonians $\widehat{H}_i$ are the universal simply-laced quantum Hamiltonians.
\end{Definition}

The quantum connection defines linear differential equations for $\widehat{A}$-valued functions on $\mathbf{B}$, which are by definition the \textit{quantum isomonodromy equations}.

\begin{Remark}The connection of Definition~\ref{def:slqc} is ``universal'' since one can replace the quantum algebra $\widehat{A}$ with any left $\widehat{A}$-module $\rho \colon \widehat{A} \to \End(\mathcal{V})$, and let $\widehat{H}_i$ act on $\mathcal{V}$ in the given representation. The resulting connection in the vector bundle $\mathcal{V} \times \mathbf{B} \to \mathbf{B}$ is strongly flat, and computing its monodromy provides representations of $\pi_1(\mathbf{B})$ on $\mathcal{V}$~-- i.e., representations of arbitrary finite products of pure braid groups.

A particular important example is obtained from the quotient modulo the ideal generated by $\hslash - 1$, which yields a surjective morphism $\widehat{A} \to A$. In this case one obtains a strongly flat connection in the bundle $A \times \mathbf{B} \to \mathbf{B}$: it is the simply-laced quantum connection, genera\-li\-sing the Knizhnik--Zamolodchikov connection~\cite{knizhnik_zamolodchikov_1984_current_algebra_and_wess_zumino_model_in_two_dimensions}, the Casimir connection~\cite{millson_toledanolaredo_2005_casimir_operators_and_monodromy_representations_of_generalised_braid_groups} and the FMTV connection~\cite{felder_markov_tarasov_varchenko_2000_differential_equations_compatible_with_kz_equations}, in the sense explained in~\cite{rembado_2019_simply_laced_quantum_connections_generalising_kz}.

\end{Remark}

\subsection{Compatibility of action and quantisation}
\label{sec:action_quantisation_compatibility}

We concluce this section by showing that $\mathcal{Q}_{\bm{a},\hslash}$ is compatible with the classical and quantum actions of $\SL_2(\mathbb{C})$-action of Sections~\ref{sec:classical_symmetries} and~\ref{sec:quantum_action}~-- respectively.

\begin{Lemma}
\label{lemma:quantisation_and_action_commute}

The quantisation map intertwines the $\SL_2(\mathbb{C})$-actions on traces of cycles:
\begin{equation*}
 \widehat{\varphi}^*_{g,\hslash}(\bm{a}) \circ \mathcal{Q}_{\bm{a},\hslash} = \mathcal{Q}_{\bm{a},\hslash} \circ \varphi^*_g(\bm{a}) \quad \text{on} \ \Tr(\mathbb{C}\mathcal{G}_{\cycl}), \qquad \text{for} \quad g \in \SL_2(\mathbb{C}), \  \bm{a} \in \mathbf{A}.
\end{equation*}
\end{Lemma}

\begin{proof}Let $C$ be an oriented cycle in $\mathcal{G}$, and $C_1 \subseteq \mathcal{G}_1$ its set of arrows as subquiver. Then~\eqref{eq:classical_action_decomposition} implies that
\begin{equation*}
 \varphi^*_g(\bm{a}) \Tr(C) = \eta \Tr(C), \qquad \text{where} \quad \eta \coloneqq \prod_{e \in C_1} \eta_{t(e)}(g,\bm{a}) \in \mathbb{C}^*.
\end{equation*}
Hence the subspace $\Tr(\mathbb{C}\mathcal{G}_{\cycl}) \subseteq A_0$ is preserved by the action, and the composition $\mathcal{Q}_{\bm{a},\hslash} \circ \varphi^*_g(\bm{a})$ is defined. Then the linearity of the quantisation map yields
\begin{equation*}
 \mathcal{Q}_{\bm{a},\hslash} \circ \varphi^*_g(\bm{a}) \big(\Tr(C)\big) = \eta \mathcal{Q}_{\bm{a},\hslash} \Tr(C) = \eta \Tr_{\bm{a},\hslash}\big(\mathcal{Q}_{\mathcal{G}}(C)\big).
\end{equation*}

Similarly, if $e \in C_1$ is fixed and $\widehat{C}_e$ is the quantum cycle anchored at $e$ then~\eqref{eq:quantum_action} yields
\begin{equation*}
 \widehat{\varphi}^*_{g,\hslash}(\bm{a}) \Tr_{\bm{a},\hslash} \big(\widehat{C}_e\big) = \eta \Tr_{\bm{a},\hslash} \big(\widehat{C}_e\big),
\end{equation*}
with the same number $\eta$ independently of the choice of the anchor. Hence by $\mathbb{C}$-linearity~\eqref{eq:quantisation_of_potentials} gives
\begin{equation*}
 \widehat{\varphi}^*_{g,\hslash}(\bm{a}) \circ \mathcal{Q}_{\bm{a},\hslash} \big(\Tr(C)\big) = \widehat{\varphi}^*_{g,\hslash}(\bm{a}) \Tr_{\bm{a},\hslash} \big(\mathcal{Q}_{\mathcal{G}}(C)\big) = \eta \Tr_{\bm{a},\hslash}\big(\mathcal{Q}_{\mathcal{G}}(C)\big).\tag*{\qed}
\end{equation*}\renewcommand{\qed}{}
\end{proof}

\section{Quantum Hamiltonian reduction}
\label{sec:quantum_hamiltonian_reduction}

We introduce in this section the algebraic formalism of quantum Hamiltonian reduction, first in filtered quantisation and then in formal deformation quantisation (see~\cite{etingof_2007_calogero_moser_systems_and_representation_theory}).

\subsection{Reduction for filtered quantisation}\label{sec:quantum_hamiltonian_reduction_filtered}

Let $B$ be an associative algebra equipped with an action of a Lie algebra $\mathfrak{g}$, i.e., a morphism $\rho \colon \mathfrak{g} \to \Der(B)$ of Lie algebras.

\begin{Definition}\label{def:quantum_comoment}A quantum comoment map for the action is a morphism $\widehat{\mu}^* \colon U(\mathfrak{g}) \to B$ of associative algebras lifting $\rho$ through the adjoint action of $B$ on itself:
\begin{equation*}
 \rho(x).b = \big[\widehat{\mu}^*(x),b\big], \qquad \text{for} \quad x \in \mathfrak{g}, \  b \in B.
\end{equation*}
\end{Definition}

Let now $\widehat{\mathcal{I}} \subseteq U(\mathfrak{g})$ be a two-sided ideal.

\begin{Definition}\label{def:filtered_quantum_hamiltonian_reduction}The quantum Hamiltonian reduction of $B$ with respect to the quantum comoment $\widehat{\mu}^*$ and the ideal $\widehat{\mathcal{I}}$ is the quotient
\begin{equation*}
 R\big(B,\mathfrak{g},\widehat{\mathcal{I}}\big) \coloneqq B^{\mathfrak{g}} \big\slash \widehat{\mathcal{J}}^{\mathfrak{g}},
\end{equation*}
where $B^{\mathfrak{g}}$ is ring of $\mathfrak{g}$-invariants, $\widehat{\mathcal{J}} \subseteq B$ is the two-sided ideal generated by $\widehat{\mu}^*\big(\widehat{\mathcal{I}}\big)$, and $\widehat{\mathcal{J}}^{\mathfrak{g}} \coloneqq \widehat{\mathcal{J}} \cap B^{\mathfrak{g}}$.
\end{Definition}

One can show that $\widehat{\mathcal{J}}^{\mathfrak{g}} \subseteq B^{\mathfrak{g}}$ is a two-sided ideal, and thus the reduction is canonically an associative algebra.

This construction can then be related to filtered quantisation, as follows. Assume there exist a~grading $\mathcal{B}_0$ on $B_0$ and a~filtration $\mathcal{B}$ on $B$ such that the associative filtered algebra $(B,\mathcal{B})$ is a~filtered quantisation of the graded commutative Poisson algebra $(B_0,\mathcal{B}_0)$, and that there is a~classical comoment map $\mu^* \colon \Sym(\mathfrak{g}) \to B_0$.

\begin{Definition}\label{def:filtered_quantisation_comoment}A filtered quantisation of $\mu^*$ is a quantum comoment $\widehat{\mu}^* \colon U(\mathfrak{g}) \to B$ whose associated graded equals $\mu^*$ in the identifications $\gr(B) \cong B_0$ and $\gr \big( U(\mathfrak{g})\big) \cong \Sym(\mathfrak{g})$.
\end{Definition}

This implies that the action $\rho = \{\mu^*,\cdot\}$ on $B_0$ is quantised by the $\mathfrak{g}$-action $\widehat{\rho} \coloneqq \big[\widehat{\mu}^*,\cdot\big]$ on $B$, i.e., $\gr \widehat{\rho}(x) = \rho(x) \in \Der(B_0)$ for $x \in \mathfrak{g}$.

\subsection{Reduction for formal deformation quantisation}\label{sec:quantum_hamiltonian_reduction_formal}

Suppose $\widehat{B}$ is a deformation algebra, that is a topologically free $\mathbb{C} \llbracket \hslash \rrbracket$ algebra (a formal deformation of $\widehat{B} \big\slash \hslash\widehat{B}$), equipped with an action $\rho \colon \mathfrak{g} \to \Der \big(\widehat{B}\big)$ of the Lie algebra $\mathfrak{g}$.

\begin{Definition}\label{def:quantum_comoment_formal}A quantum comoment map for the action is a morphism
\begin{equation*}
 \widehat{\mu}^* \colon U(\mathfrak{g}) \longrightarrow \widehat{B}\big[\hslash^{-1}\big]
\end{equation*}
of associative algebras lifting $\rho$ through the adjoint action of $\widehat{B}$ on itself.
\end{Definition}

One allows for negative powers of $\hslash$ since $\big[\widehat{B},\widehat{B}\big] \subseteq \hslash^k \widehat{B}$ for some minimal integer $k > 0$ when $B_0 \coloneqq \widehat{B} \big\slash \hslash \widehat{B}$ is commutative. This is precisely the relevant case for deformation quantisation (for example with the algebra~\eqref{eq:completion_rees} one has $k = 2$).

Assume then further that $\widehat{B}$ is a deformation quantisation of the commutative Poisson algebra $(B_0,\{\cdot,\cdot\})$, which means that
\begin{equation*}
 \sigma\big([f,g] \cdot \hslash^{-k}\big) = \big\{\sigma(f),\sigma(g)\big\} \in B_0, \qquad \text{for}\quad f,g \in \widehat{B},
\end{equation*}
where $\sigma \colon \widehat{B} \to B_0$ is the semiclassical limit (the projection modulo the ideal generated by $\hslash$), and let $\mu^* \colon \Sym(\mathfrak{g}) \to B_0$ a classical comoment map.

\begin{Definition}A quantum comoment map $\widehat{\mu}^* \colon U(\mathfrak{h}) \to \widehat{B}\big[\hslash^{-1}\big]$ is said to be a quantisation of the classical comoment map $\mu^*$ if
\begin{equation*}
 \sigma\big(\hslash^k\widehat{\mu}^*(x)\big) = \mu^*(x), \qquad \text{for} \quad x \in \mathfrak{g}.
\end{equation*}
\end{Definition}

This implies that the action $\rho = \{\mu^*,\cdot\}$ on $B_0$ is quantised by the $\mathfrak{g}$-action $\widehat{\rho}_{\hslash} \coloneqq \big[\widehat{\mu}^*,\cdot\big]$, since for $x \in \mathfrak{g}$ and $b \in \widehat{B}$ one has
\begin{equation*}
 \sigma \big(\widehat{\rho}_{\hslash}(x)\big).b = \sigma \big(\big[\hslash^k\widehat{\mu}^*(x),b\big] \hslash^{-k}\big) = \big\{\sigma\big(\hslash^k \widehat{\mu}^*(x)\big),\sigma(b)\big\} = \big\{\mu^*(x),\sigma(b)\big\}.
\end{equation*}

In this case the ring of $\mathfrak{g}$-invariants $\widehat{B}^{\mathfrak{g}} \subseteq \widehat{B}$ is the centraliser of the image of $\widehat{\mu}^*$, and an analogue to Definition~\ref{def:filtered_quantum_hamiltonian_reduction} is obtained by a $\hslash$-deformation of the quantum comoment. To introduce it consider the algebras $\Rees(B,\mathcal{B}) \subseteq \widehat{B}$ as in Section~\ref{sec:rees_construction}, and define $\widehat{\mu}^*_{\hslash}$ on $\widehat{U}(\mathfrak{g})$ by imposing
\begin{equation}\label{eq:deformed_quantum_comoment}
 \widehat{\mu}^*_{\hslash}(\hslash) \coloneqq \hslash, \qquad \widehat{\mu}^*_{\hslash}(x \hslash) \coloneqq \widehat{\mu}^*(x) \hslash^k, \qquad \text{for} \quad x \in \mathfrak{g}.
\end{equation}

Choose now a two-sided ideal $\widehat{\mathcal{I}}_{\hslash} \subseteq \widehat{U}(\mathfrak{g})$

\begin{Definition}The quantum Hamiltonian reduction of $\widehat{B}$ with respect to the $\hslash$-deformed quantum comoment $\widehat{\mu}^*_{\hslash}$ and the ideal $\widehat{\mathcal{I}}_{\hslash}$ is the quotient
\begin{equation*}
 R_{\hslash}\big(\widehat{B},\mathfrak{g},\widehat{\mathcal{I}}_{\hslash}\big) \coloneqq \widehat{B}^{\mathfrak{g}} \big\slash \widehat{\mathcal{J}}_{\hslash}^{\mathfrak{g}},
\end{equation*}
where $\widehat{\mathcal{J}}_{\hslash} \coloneqq \widehat{B} \widehat{\mu}^*_{\hslash}\big(\widehat{\mathcal{I}}_{\hslash}\big)$ is the ideal generated by the image of $\widehat{\mathcal{I}}_{\hslash}$ for the quantum comoment, and $\widehat{\mathcal{J}}_{\hslash}^{\mathfrak{g}} \coloneqq \widehat{\mathcal{J}}_{\hslash} \cap \widehat{B}^{\mathfrak{g}}$.
\end{Definition}

Finally, the material of Sections~\ref{sec:filtered_quantisation} and~\ref{sec:rees_construction} can be adapted to pass from the filtered setting to the deformation one in quantum Hamiltonian reduction.

Suppose then the formal deformation quantisation $\widehat{B}$ of $B_0$ is obtained from the filtered quantisation $(B,\mathcal{B})$ via the Rees construction, and let $\widehat{\mu}^* \colon U(\mathfrak{g}) \to B$ be a quantisation of the classical comoment $\mu^* \colon \Sym(\mathfrak{g}) \to B_0$ as in Definition~\ref{def:filtered_quantisation_comoment}. If the $\mathfrak{g}$-action $\widehat{\rho} = \big[\widehat{\mu}^*,\cdot\big]$ on $B$ preserves the filtration $\mathcal{B}$ then there is a natural induced action on $\widehat{B}$ by $\mathbb{C} \llbracket \hslash \rrbracket$-linearity:
\begin{equation*}
 \widehat{\rho}_{\hslash} \colon \ \sum_{k \geq 0} f_k \cdot \hslash^k \longmapsto \sum_k \widehat{\rho}(f_k) \cdot \hslash^k.
\end{equation*}
Moreover there is a natural inclusion $B[\hslash] \hookrightarrow \widehat{B}\big[\hslash^{-1}\big]$, writing
\begin{equation*}
 f \cdot \hslash^k = f \cdot \hslash^{\lvert f \rvert} \hslash^{k - \lvert f \rvert}, \qquad \text{for} \quad f \in B, \ k \geq 0,
\end{equation*}
and restricting it to $B$ turns the quantum comoment into a map $\widehat{\mu}^* \colon U(\mathfrak{g}) \to \widehat{B}\big[\hslash^{-1}\big]$ which generates $\widehat{\rho}_{\hslash}$ via the adjoint action. Hence this is a quantum comoment in deformation quantisation as in Definition~\ref{def:quantum_comoment_formal}, which can be $\hslash$-deformed to $\widehat{\mu}^*_{\hslash}$ as in~\eqref{eq:deformed_quantum_comoment}. This deformed quantum comoment then takes values in $\widehat{B}$, because the minimal integer $k \geq 1$ such that $\big[\widehat{B},\widehat{B}\big] \subseteq \hslash^k \widehat{B}$ coincides with the minimal integer $k$ such that $\widehat{\mu}^*(\mathfrak{g}) \subseteq \mathcal{B}_{\leq k}$~-- else $\widehat{\rho} = \big[\widehat{\mu}^*,\cdot\big]$ would not preserve the filtration.

Hence on the whole $\widehat{\mu}^*_{\hslash} \colon \widehat{U}(\mathfrak{g}) \to \widehat{B}$ is defined by
\begin{equation*}
 \widehat{\mu}^*_{\hslash} (\hslash) = \hslash, \qquad \widehat{\mu}^*_{\hslash} (x \hslash) = \widehat{\mu}^*(x) \cdot \hslash^{\lvert \widehat{\mu}^*(x)\rvert}, \qquad \text{for} \quad x \in \mathfrak{g},
\end{equation*}
and it intertwines the semiclassical limits $\sigma_{\mathfrak{g}} \colon \widehat{U}(\mathfrak{g}) \to \Sym(\mathfrak{g})$ and $\sigma \colon \widehat{B} \to B_0$ because
\begin{equation*}
 \sigma \circ \widehat{\mu}^*_{\hslash} (x\hslash) = \sigma_{\lvert \widehat{\mu}^*(x)\rvert} \circ \widehat{\mu}^*(x) = \mu^*(x) = \mu^*\big(\sigma_{\mathfrak{g}}(x \hslash)\big), \qquad \text{for} \quad x \in \mathfrak{g},
\end{equation*}
using~\eqref{eq:deformed_quantum_comoment}. So this is also a deformation quantisation of the classical comoment $\mu^*$.

Suppose finally to have two sets of data $\big(B,\mathfrak{g},\widehat{\mathcal{I}}\big)$ and $\big(B',\mathfrak{g}',\widehat{\mathcal{I}}'\big)$ defining quantum Hamiltonian reduction in the filtered setting, let $\widehat{B}$ and $\widehat{B}'$ be the $\mathbb{C} \llbracket \hslash \rrbracket$-algebras given by the Rees construction and choose further ideals $\widehat{\mathcal{I}}_{\hslash} \subseteq \widehat{U}(\mathfrak{g})$ and $\widehat{\mathcal{I}}'_{\hslash} \in \widehat{U}(\mathfrak{g'})$.

\begin{Lemma}\label{lemma:reduction_quantum_morphisms}Let $\widehat{\varphi} \colon B \to B'$ be a ring morphism such that $\widehat{\varphi}\big(B^{\mathfrak{g}}\big) \subseteq B'^{\mathfrak{g}'}$ and $\widehat{\varphi}\big(\widehat{\mathcal{J}}\big) \subseteq \widehat{\mathcal{J}}'$. Then:
\begin{enumerate}\itemsep=0pt
 \item[$1.$] There exists a unique ring morphism $R\widehat{\varphi} \colon R\big(B,\mathfrak{g},\widehat{\mathcal{I}}\big) \to R\big(B',\mathfrak{g}',\widehat{\mathcal{I}}'\big)$ induced by the universal property of quotients.

 \item[$2.$] If $\widehat{\varphi}$ is a filtered map then the morphism $\widehat{\varphi}_{\hslash} \colon \widehat{B} \to \widehat{B}'$ extended from $\widehat{\varphi}$ by $\mathbb{C}\llbracket \hslash \rrbracket$-linearity satisfies $\widehat{\varphi}_{\hslash}\big(\widehat{B}^{\mathfrak{g}}\big) \subseteq \widehat{B}'^{\mathfrak{g}'}$. If further $\widehat{\varphi}_{\hslash} \big(\widehat{\mathcal{J}}_{\hslash}\big) \subseteq \widehat{\mathcal{J}}'_{\hslash}$ then there exists a unique ring morphism $R\widehat{\varphi}_{\hslash} \colon R_{\hslash}\big(\widehat{B},\mathfrak{g},\widehat{\mathcal{I}}_{\hslash}\big) \to R_{\hslash}\big(\widehat{B}',\mathfrak{g}',\widehat{\mathcal{I}}'_{\hslash}\big)$ induced by the universal property of quotients.
\end{enumerate}
\end{Lemma}

\begin{proof}The first item follows directly from the hypotheses.

For the former statement of the second item, note the Lie algebra actions on $\widehat{B}$ and $\widehat{B}'$ are obtained by imposing $\mathbb{C} \llbracket \hslash \rrbracket$-linearity, and the morphism $\widehat{\varphi}_{\hslash}$ preserves power series with invariant coefficients; but these are precisely the invariant elements inside the deformation quantisation. The latter statement of the second item follows directly from the hypotheses.
\end{proof}

This is as far as one can go with general constructions, since there is no global correspondence between ideals of $B$ and $\widehat{B}$. A partial correspondence is given by replacing elements $b \in B$ with their homogeneous version $b \cdot \hslash^{\lvert b \rvert} \in \widehat{B}$, which yield ideals in~$\widehat{B}$ whose semiclassical limit is homogeneous. In particular the quantisation of nonhomogeneous ideals in $B_0$ in power series in~$\hslash$ will not correspond to any of those (and such ideals do not admit filtered quantisations, cf.\ Remark~\ref{remark:orbit_quantisation}).

This is another motivation for introducing deformation quantisation instead of just working with filtered quantisation~-- apart from formalising the semiclassical limit.

\subsection{Quantum Hamiltonian reduction of symplectic quiver varieties}\label{sec:quantum_reduction_quiver_varieties}

The material of the previous Sections~\ref{sec:quantum_hamiltonian_reduction_filtered} and~\ref{sec:quantum_hamiltonian_reduction_formal} will now be applied to the case of the representation variety of a simply-laced graph/quiver. We will keep the notation that was used in the particular case of a complete $k$-partite quiver, since all this material will specialise to it.

Let then $\mathcal{G} = (\mathcal{G}_0 = I,\mathcal{G}_1)$ be a (finite) simply-laced graph, that is a graph having at most one edge between any two distinct nodes and no loop edges. Consider it equivalently as the quiver obtained by replacing every edge by a pair of opposite arrows, also denoted~$\mathcal{G}$.

The we attach finite-dimensional vector spaces $V_i$ to the nodes $i \in I$ and let $\mathbb{M} = \Rep(\mathcal{G},V)$ be the representation space in the $I$-graded vector space $V = \bigoplus_{i \in I} V_i$. To make it into a~symplectic vector space choose a skew-symmetric function $\alpha \mapsto \varepsilon_{\alpha} \colon \mathcal{G}_1 \to \mathbb{C}^*$, that is a function satisfying $\varepsilon_{\alpha} + \varepsilon_{\alpha^*} = 0$ for all $\alpha \in \mathcal{G}_1$ where $\alpha^*$ is the (unique) opposite arrow to $\alpha$, and define
\begin{equation*}
 \omega_{\varepsilon} \coloneqq \sum_{\alpha \in \mathcal{G}_1} \frac{\varepsilon_{\alpha}}{2} \Tr\big({\rm d}B_{\alpha} \wedge {\rm d}B_{\alpha^*}\big) \in \Omega^2(\mathbb{M},\mathbb{C}),
\end{equation*}
where $B_{\alpha} \in \Hom(V_{s(\alpha)},V_{t(\alpha)})$ is the linear map defined by a representation and $s,t \colon \mathcal{G}_1 \to I$ are the source and target maps.

\begin{Remark}\label{remark:lagrangian_splitting}Choosing an orientation for the arrows of $\mathcal{G}$, i.e., defining a partition $\mathcal{G}_1 = \mathcal{G}_1^+ \coprod \mathcal{G}_1^-$ with the two parts swapped by the involution $\alpha \mapsto \alpha^*$, yields a symplectic identification
\begin{equation*}
 \mathbb{M} \cong T^*\mathbb{L}, \qquad \text{where} \quad \mathbb{L} \coloneqq \bigoplus_{\alpha \in \mathcal{G}^+_1} \Hom(V_{s(\alpha)},V_{t(\alpha)}).
\end{equation*}
The subspace $\mathbb{L}$ is Lagrangian for the canonical symplectic structure $\omega = \sum_{\alpha \in \mathcal{G}^+_1} \Tr\big({\rm d}B^{\alpha} \wedge {\rm d}B^{\alpha^*}\big)$ of the cotangent bundle. This form corresponds to defining the skew-symmetric function $\varepsilon$ by imposing $\varepsilon = 1$ on $\mathcal{G}_1^+$.

A suitable change of Darboux coordinates on $\mathbb{M}$ turns the starting situation into this one. However this is non canonical (and breaks the $\SL_2(\mathbb{C})$-symmetries in the case of a complete $k$-partite quiver).
\end{Remark}

The group $\widehat{H} = \prod_{i \in I} \GL(V_i)$ acts on $(\mathbb{M},\omega_{\varepsilon})$ with a moment map $\mu_{\varepsilon} \colon \mathbb{M} \to \mathfrak{h}^* \cong \mathfrak{h}$, where $\mathfrak{h} = \Lie\big(\widehat{H}\big)$ is the Lie algebra and the duality is provided by the nondegenerate trace pairings $\mathfrak{gl}(V_i) \otimes \mathfrak{gl}(V_i) \to \mathbb{C}$. The moment map admits the following formula:
\begin{equation}\label{eq:classical_moment}
 \mu_{\varepsilon} \colon B \longmapsto \bigoplus_{i \in I} \left( \sum_{\alpha \in t^{-1}(i)}\varepsilon_{\alpha} B_{\alpha}B_{\alpha^*}\right), \qquad \text{where} \quad B = (B_{\alpha})_{\alpha \in \mathcal{G}_1}.
\end{equation}

Now we split~\eqref{eq:classical_moment} into components after introducing some notation. For $i \in I$ denote $\mu_{\varepsilon,i}$ the component of $\mu$ taking values in $\mathfrak{h}_i \coloneqq \mathfrak{gl}(V_i)$ and write $U^i \coloneqq \bigoplus_{\alpha \in t^{-1}(i)} V_{s(\alpha)}$ the direct sum of spaces on nodes adjacent to $i$. Then $\mu_i$ only depends on the linear maps inside the $\GL(V_i)$-Hamiltonian subspace
\begin{equation*}
 \mathbb{M}_i \coloneqq \bigoplus_{\alpha \in t^{-1}(i)} \Hom\big(V_{s(\alpha)},V_{t(\alpha)}\big) \oplus \Hom\big(V_{t(\alpha)},V_{s(\alpha)}\big) = \Hom\big(U^i,V_i\big) \oplus \Hom\big(V_i,U^i\big),
\end{equation*}
which is the space of representation of the full subquiver $\mathcal{G}_i \subseteq \mathcal{G}$ on nodes
\begin{equation*}
 (\mathcal{G}_i)_0 \coloneqq \{i\} \cup \big\{s(\alpha) \,|\, \alpha \in t^{-1}(i) \big\} \subseteq I.
\end{equation*}
This is the (full) star-shaped subquiver centred at $i$, e.g., Fig.~\ref{fig:star} where an edge stands for a pair of opposite arrows.
\begin{figure}[h]\centering
\begin{tikzpicture}
 \vertex [label=above right:{$i$}] (a) at (0,0) {};
 \vertex (b) [above of = a] {};
 \vertex (c) [right of = a] {};
 \vertex [label=left:{$\mathcal{G}_i =$}] (d) [left of = a] {};
 \vertex (e) [below of = a] {};
\path
 (a) edge [-] (b)
 (a) edge [-] (c)
 (a) edge [-] (d)
 (a) edge [-] (e);
\end{tikzpicture}
\caption{A simply-laced star-shaped graph centred at $i$ (here with 4 peripheral nodes). \label{fig:star}}
\end{figure}
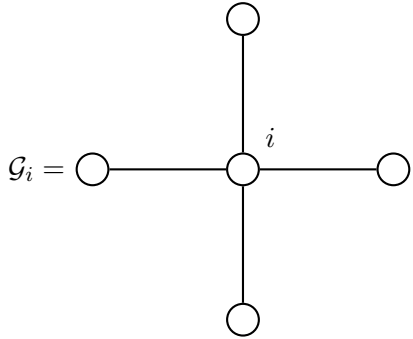

Now consider the pull-back $\mu^*_{\varepsilon,i} \colon \Sym\big(\mathfrak{h}_i^*\big) \to \Sym(\mathbb{M}_i^*) \subseteq A_0$ along $\mu_{\varepsilon,i}$, where $A_0 = \mathbb{C}[\mathbb{M}]$. This map is the restriction of the global comoment $\mu^* \colon \Sym(\mathfrak{h}^*) \to A_0$ along the inclusion $\mathfrak{h}_i^* \hookrightarrow \mathfrak{h}^*$ induced by the canonical projection $\mathfrak{h} \to \mathfrak{h}_i$, but it also the classical comoment for the Hamiltonian $\GL(V_i)$ action on $\mathbb{M}_i$ equipped with the restriction of $\omega_{\varepsilon}$. To give a formula for it choose bases for the spaces $V_i$~-- or equivalently take a dimension vector $d = (d_i)_{i \in I} \in \mathbb{Z}_{\geq 0}^I$ and let $V_i \coloneqq \mathbb{C}^{d_i}$~-- and denote $\big(\Lambda^i_{kl}\big)_{kl}$ the set of linear coordinates on $\mathfrak{h}_i$ in the basis. Similarly define the linear coordinates $\big(B_{\alpha}^{kl}\big)_{kl}$ on the subspace $\Hom\big(V_{s(\alpha)},V_{t(\alpha)}\big)$ for $\alpha \in \mathcal{G}_1$.

Then~\eqref{eq:classical_moment} yields the following formula:
\begin{equation}\label{eq:classical_local_comoment}
 \mu^*_{\varepsilon,i}\big(\Lambda^i_{kl}\big) = \sum_{\alpha \in t^{-1}(i)} \varepsilon_{\alpha} \sum_m B_{\alpha}^{km}B_{\alpha^*}^{ml},
\end{equation}
and we want to quantise this comoment map. To this end denote $(A,\ast_{\varepsilon})$ the Weyl algebra of the vector space $\mathbb{M}^*$ equipped with the alternating bilinear form defined by the Poisson bracket $\{\cdot,\cdot\}_{\varepsilon} = \omega_{\varepsilon}^{-1}$, and abusively keep the same notation for the restriction of the bracket to $\mathbb{M}^*_i$. Then there is a canonical embedding $W(\mathbb{M}^*_i,\{\cdot,\cdot\}_{\varepsilon}) \hookrightarrow A$ induced on the quotients from the ring morphism $\Tens(\iota)$, where $\iota \colon \mathbb{M}_i^* \hookrightarrow \mathbb{M}^*$ is the dual of the canonical projection $\mathbb{M} \to \mathbb{M}_i$.

\begin{Lemma}\label{lemma:local_quantum_comoment}
The following formula defines a morphism $\widehat{\mu}^*_{\varepsilon,i} \colon U(\mathfrak{h}_i^*) \to W(\mathbb{M}^*_i,\{\cdot,\cdot\}_{\varepsilon}) \subseteq A$ of associative algebras whose associated graded coincides with~\eqref{eq:classical_local_comoment}:
\begin{equation}\label{eq:quantum_comoment_local}
 \widehat{\mu}^*_{\varepsilon,i}\big(\widehat{\Lambda}^i_{kl}\big) \coloneqq \frac{1}{2}\sum_{\alpha \in t^{-1}(i)} \varepsilon_{\alpha} \sum_m \big(\widehat{B}_{\alpha}^{km} \ast_{\varepsilon} \widehat{B}_{\alpha^*}^{ml} + \widehat{B}_{\alpha^*}^{ml} \ast_{\varepsilon} \widehat{B}_{\alpha}^{km}\big),
\end{equation}
where $\widehat{\Lambda}^i_{kl}$ and $\widehat{B}_{\alpha}^{kl}$ are the Weyl quantisations of the corresponding coordinate functions.
\end{Lemma}

\begin{proof}By the universal property of the quotient, it will be enough to show that the morphism $\widetilde{\mu}^*_{\varepsilon,i} \colon \Tens\big(\mathfrak{h}_i^*\big) \to \Tens(\mathbb{M}_i^*) \subseteq \Tens(\mathbb{M}^*)$ defined by
\begin{equation*}
 \qquad \widetilde{\mu}^*_{\varepsilon,i} \colon \ \Lambda^i_{kl} \longmapsto \frac{1}{2}\sum_{\alpha \in t^{-1}(i)} \varepsilon_{\alpha} \sum_m \big(B_{\alpha}^{km} \otimes B_{\alpha^*}^{ml} + B_{\alpha^*}^{ml} \otimes B_{\alpha}^{km}\big)
\end{equation*}
sends the Lie bracket of two elements of $X,Y \in \mathfrak{h}^*_i$ to the Poisson bracket $\big\{\widetilde{\mu}^*_i(X),\widetilde{\mu}^*_i(Y)\big\}_{\varepsilon}$. Since the Lie bracket of the dual Lie algebra is by definition induced from that of $\mathfrak{h}_i$ by the trace duality $\Tr \colon \mathfrak{h}_i \to \mathfrak{h}_i^*$, it is easier to prove that $\alpha \coloneqq \widetilde{\mu}^*_i \circ \Tr$ sends $[X,Y] \in \mathfrak{h}_i$ to $\{\alpha(X),\alpha(Y)\}_{\varepsilon}$ for $X,Y \in \mathfrak{h}_i$. This can be checked in the chosen coordinates using the commutation relations
\begin{equation*}
 \big[e^i_{kl},e^i_{k'l'}\big] = \delta_{k'l}e^i_{kl'} - \delta_{kl'}e^i_{k'l}
\end{equation*}
inside $\mathfrak{h}_i$, where $e_{kl} \in \mathfrak{h}_i$ is the vector sent to $\Tr(e^i_{kl} \cdot) = \Lambda^i_{lk}$ by the trace duality, as well as the Poisson-commutation relations
\begin{equation}\label{eq:commutation_relations_classical}
 \big\{B_{\alpha}^{kl},B_{\beta}^{k'l'}\big\}_{\varepsilon} = \varepsilon_{\alpha}^{-1} \delta_{\alpha^* \beta} \delta_{kl'}\delta_{k'l},
\end{equation}
inside $A_0$.\footnote{These Poisson-commutation relations hold because $\{\cdot,\cdot\}_{\varepsilon} = \omega_{\varepsilon}^{-1}$, and since by definition
\begin{equation*}
 \omega_{\varepsilon}\left(\frac{\partial}{\partial (B_{\alpha})_{kl}},\frac{\partial}{\partial (B_{\beta})_{k'l'}}\right) = \varepsilon_{\alpha} \delta_{\alpha^* \beta} \delta_{kl'}\delta_{k'l},
\end{equation*}
where we take the vector fields associated to the coordinate functions $(B_{\alpha})_{kl}$.}

Then formula~\eqref{eq:quantum_comoment_local} holds. The statement about the associated graded follows from the identity
\begin{equation*}
 \sigma_2 \big(\widehat{B}_{\alpha}^{km} \ast_{\varepsilon} \widehat{B}_{\alpha^*}^{ml} + \widehat{B}_{\alpha^*}^{ml} \ast_{\varepsilon} \widehat{B}_{\alpha}^{km}\big) = 2 B_{\alpha}^{km}B_{\alpha^*}^{ml} \in \Sym(\mathbb{M}_i^*) \subseteq A_0.\tag*{\qed}
\end{equation*}\renewcommand{\qed}{}
\end{proof}

Now collect these morphisms together into a map $\widehat{\mu}_{\varepsilon}^* \colon U(\mathfrak{h}^*) \to A$ by imposing
\begin{equation}\label{eq:quantum_comoment_global}
 \widehat{\mu}_{\varepsilon}^* \colon \ \bigotimes_{i \in I} \widehat{\Lambda}^i \longmapsto \prod_{i \in I} \widehat{\mu}^*_{\varepsilon,i} \big(\widehat{\Lambda}^i\big), \qquad \text{for} \quad \widehat{\Lambda}^i \in \mathfrak{h}^*_i,
\end{equation}
in the identification $U(\mathfrak{h}^*) \cong \bigotimes_{i \in I} U(\mathfrak{h}_i^*)$, and using the product of $A$ on the right-hand side.

\begin{Lemma}\label{lemma:global_quantum_comoment}The map $\widehat{\mu}_{\varepsilon}^*$ is a morphism of associative algebras.
\end{Lemma}

\begin{proof}The point is showing that the images of $\widehat{\mu}^*_{\varepsilon,i}$ and $\widehat{\mu}^*_{\varepsilon,j}$ commute inside $A$ for $i \neq j$ in $I$, as this implies that{\samepage
\begin{equation*}
 \widehat{\mu}_{\varepsilon}^* \left(\bigotimes_{i \in I} \widehat{\Lambda}^i_1 \cdot \widehat{\Lambda}^i_2\right) = \prod_{i \in I} \widehat{\mu}_{\varepsilon,i}^*\big(\widehat{\Lambda}^i_1\big) \ast_{\varepsilon} \prod_{i \in I} \widehat{\mu}_{\varepsilon,i}^*\big(\widehat{\Lambda}^i_2\big), \qquad \text{for} \quad \widehat{\Lambda}^i_1, \widehat{\Lambda}^i_2 \in \mathfrak{h}^*_i,
\end{equation*}
where $\widehat{\Lambda}^i_1 \cdot \widehat{\Lambda}^i_2$ is the product in $U(\mathfrak{h}^*_i)$.}

Importantly, the commutator of the product $\ast_{\varepsilon}$ is determined by the adjacency of $\mathcal{G}$ by means of the relations
\begin{equation}\label{eq:commutation_relations_quantum}
 \big[\widehat{B}_{\alpha}^{kl},\widehat{B}_{\beta}^{k'l'}\big] = \varepsilon_{\alpha}^{-1} \delta_{\alpha^* \beta}\delta_{kl'}\delta_{k'l},
\end{equation}
which in turn follow from~\eqref{eq:commutation_relations_classical}. In particular only the quantum variables attached to opposite arrows do not commute.

The commutativity is clear if there is no arrow between $i$ and $j$: in that case one has the stronger statement that $\big[W(\mathbb{M}^*_i,\{\cdot,\cdot\}_{\varepsilon}),W(\mathbb{M}^*_j,\{\cdot,\cdot\}_{\varepsilon})\big] = (0)$ in the Weyl algebra, since there are no arrows in the full star-shaped subquiver $\mathcal{G}_i$ centred at $i$ with its opposite in the full star-shaped subquiver $\mathcal{G}_j$ centred at $j$.

Suppose instead to have a (unique) pair of antiparallel arrows between the nodes $i$ and $j$, and consider the subquiver $\mathcal{G}_{ij} \subseteq \mathcal{G}$ obtained by glueing together the full star-shaped quivers centred at $i$ and $j$ at their common edge/double arrow, e.g., Fig.~\ref{fig:stars_glued}.
\begin{figure}[h]\centering
\begin{tikzpicture}
 \vertex [label={$i$}] (a) at (0,0) {};
 \vertex (b) [above left of = a] {};
 \vertex [label=left:{$\mathcal{G}_{ij} =$}] (c) [left of = a] {};
 \vertex (d) [below left of = a] {};
 \vertex [label={$j$}] (e) [right of = a] {};
 \vertex (f) [above right of = e] {};
 \vertex (g) [right of = e] {};
 \vertex (h) [below right of = e] {};
\path
 (a) edge [-] (b)
 (a) edge [-] (c)
 (a) edge [-] (d)
 (a) edge [-] (e)
 (e) edge [-] (f)
 (e) edge [-] (g)
 (e) edge [-] (h);
\end{tikzpicture}
\caption{The glueing of the star-shaped subgraphs centred at $i$ and $j$ (here with 4 peripheral nodes each). \label{fig:stars_glued}}
\end{figure}
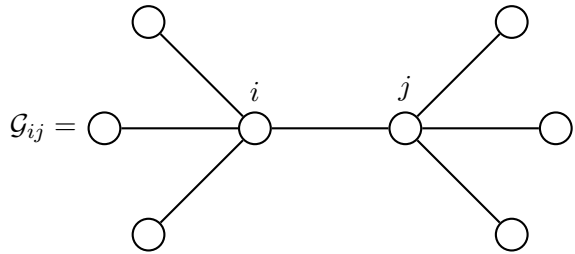

Then the morphisms $\widehat{\mu}^*_i$ and $\widehat{\mu}^*_j$ take values in the Weyl algebra for the representation space of $\mathcal{G}_{ij}$. Reasoning as above, the only commutators to consider are those among the components of $\widehat{\mu}^*_i$ and $\widehat{\mu}^*_j$ taking values inside the Weyl subalgebra
\begin{equation*}
 W(\mathbb{M}^*_{ij},\{\cdot,\cdot\}_{\varepsilon}) \subseteq W(\mathbb{M}^*_i,\{\cdot,\cdot\}_{\varepsilon}) \cap W(\mathbb{M}^*_j,\{\cdot,\cdot\}_{\varepsilon}),
\end{equation*}
where $\mathbb{M}_{ij} \coloneqq \Hom(V_i,V_j) \oplus \Hom(V_j,V_i)$. Keeping the notation $\widehat{\mu}^*_i$, $\widehat{\mu}^*_j$ for these components and writing $\alpha$ the arrow from $i$ to $j$ yields
\begin{equation*}
 \widehat{\mu}^*_i\big(\widehat{\Lambda}^i_{kl}\big) = \frac{\varepsilon_{\alpha^*}}{2} \sum_m \widehat{B}_{\alpha^*}^{km} \ast_{\varepsilon} \widehat{B}_{\alpha}^{ml} + \widehat{B}_{\alpha}^{ml} \ast_{\varepsilon} \widehat{B}_{\alpha^*}^{km},
\end{equation*}
and
\begin{equation*}
 \widehat{\mu}^*_j\big(\widehat{\Lambda}^j_{kl}\big) = \frac{\varepsilon_{\alpha}}{2} \sum_m \widehat{B}_{\alpha}^{km} \ast_{\varepsilon} \widehat{B}_{\alpha^*}^{ml} + \widehat{B}_{\alpha^*}^{ml} \ast_{\varepsilon} \widehat{B}_{\alpha}^{km},
\end{equation*}
looking at~\eqref{eq:quantum_comoment_local}. Then a direct computation shows that $\big[\widehat{\mu}^*_i\big(\widehat{\Lambda}^i_{kl}\big),\widehat{\mu}^*_j\big(\widehat{\Lambda}^j_{k'l'}\big)\big] = 0$ for all $k$,~$l$,~$k'$,~$l'$, using both $\varepsilon_{\alpha} + \varepsilon_{\alpha^*} = 0$ and the commutation relations~\eqref{eq:commutation_relations_quantum}.
\end{proof}

\begin{Remark}Lemmata~\ref{lemma:local_quantum_comoment} and~\ref{lemma:global_quantum_comoment} should be compared with~\cite[Propositions~9.4 and~9.6]{rembado_2019_simply_laced_quantum_connections_generalising_kz}. The case treated there was that of a star-shaped quiver with a function $\varepsilon$ such that $\omega_{\varepsilon}$ is already in canonical form. In that context there is no symmetry breaking in taking Darboux coordinates (since a Lagrangian splitting is already given), and this allows to simply formula~\eqref{eq:quantum_comoment_local} by replacing the symmetrisation with a normal ordered product.

Thus the present situation generalises that of~\cite[Section~9]{rembado_2019_simply_laced_quantum_connections_generalising_kz} to a generic simply-laced graph quiver, independently of the choice of Lagrangian splitting~-- hence compatibly with the $\SL_2(\mathbb{C})$-symmetries in the complete $k$-partite case.
\end{Remark}

Now notice that since the Hamiltonian $\widehat{H}$-action on $(\mathbb{M},\omega_{\varepsilon})$ is given by linear symplectomorphisms, one can use Lemma~\ref{lemma:linear_functoriality} to quantise it. The pull-backs along the $\widehat{H}$-action give Poisson automorphisms of the graded commutative Poisson algebra $(A_0,\{\cdot,\cdot\}_{\varepsilon})$, and one quantises them to automorphisms of the associative algebra $(A,\ast_{\varepsilon})$. It follows from the construction of the quantisation in Lemma~\ref{lemma:linear_functoriality} that this action is still given by the simultaneous conjugation
\begin{equation}\label{eq:quantum_conjugation_action}
 (g_i)_{i \in I}.\big(\widehat{B}_{\alpha}\big)_{\alpha \in \mathcal{G}_1} \coloneqq \big(g_{t(\alpha)}^{-1} \widehat{B}_{\alpha} g_{s(\alpha)}\big)_{\alpha \in \mathcal{G}_1},
\end{equation}
where $g_i \in \GL(V_i)$ for $i \in I$ and $\widehat{B}^{\alpha} \in \Hom\big(V_{s(\alpha)},V_{t(\alpha)}\big) \otimes A$ is the matrix containing the Weyl quantisation of the linear functions $B_{\alpha}^{kl} \in \mathbb{M}^*$.

\begin{Remark}\label{remark:quantum_trace_invariance}
In this viewpoint it is clear that the traces of quantum potentials of Section~\ref{sec:quantum_trace} define $\widehat{H}$-invariant functions. Indeed if $V_1$, $V_2$ are finite dimensional vector spaces and one takes elements $B \in \Hom(V_1,V_2) \otimes A$, $C \in \Hom(V_2,V_1) \otimes A$ then
\begin{equation*}
 \Tr(BC) - \Tr(CB) = \sum_{i,j} [B_{ij},C_ {ji}] \in A,
\end{equation*}
where $B_{ij}, C_{ji} \in A$ are the coefficients in given bases of $V_1$ and $V_2$. Thus in particular one can cyclically permute the factors inside the trace if all coefficients of $B$ commute with those of $C$. This is the case for the ``quantum'' base-changing action~\eqref{eq:quantum_conjugation_action}, since $\widehat{H}$ conjugates with respect to matrices with complex coefficients which lie in the centre of the Weyl algebra.
\end{Remark}

Now taking the tangent map of the action $\widehat{H} \to \Aut(A)$~\eqref{eq:quantum_conjugation_action} at the identity provides an $\mathfrak{h}$-action, and its composition with the trace duality yields a morphism $\rho \colon \mathfrak{h}^* \to \Der(A)$.

\begin{Lemma}The morphism~\eqref{eq:quantum_comoment_global} is a quantum comoment map for the quantum base-changing action~\eqref{eq:quantum_conjugation_action} on $A$.
\end{Lemma}

\begin{proof}According to Definition~\ref{def:quantum_comoment} one must show that $\rho(\Lambda).\widehat{X} = \big[\widehat{\mu}^*(\Lambda),\widehat{X}\big]$ for $\Lambda \in \mathfrak{h}^*$ and $\widehat{B} \in A$. This can be checked locally, i.e., fixing $i \in I$ and a pair $\alpha, \alpha^* \in \mathcal{G}_1$ of opposite arrows with $t(\alpha) = i$. Then using the restriction~\eqref{eq:quantum_comoment_global} of the comoment map yields
\begin{align*}
 \big[\widehat{\mu}_i^*\big(\widehat{\Lambda}^i_{kl}\big),\widehat{B}_{\alpha}^{qr}\big]& = \frac{\varepsilon_{\alpha}}{2}\sum_m \big[\widehat{B}_{\alpha}^{km} \ast_{\varepsilon} \widehat{B}_{\alpha^*}^{ml} + \widehat{B}_{\alpha^*}^{ml} \ast_{\varepsilon} \widehat{B}_{\alpha}^{km},\widehat{B}_{\alpha}^{qr}\big] \\
 &= \frac{\varepsilon_{\alpha}}{2} \sum_m \widehat{B}_{\alpha}^{km} \ast_{\varepsilon} \big[\widehat{B}_{\alpha^*}^{ml},\widehat{B}_{\alpha}^{qr}\big] + \big[\widehat{B}_{\alpha^*}^{ml},\widehat{B}_{\alpha}^{qr}\big] \ast_{\varepsilon} \widehat{B}_{\alpha}^{km} \\
 &= \varepsilon_{\alpha}\varepsilon_{\alpha^*}^{-1} \sum_m \delta_{mr}\delta_{ql} \widehat{B}_{\alpha}^{km} = -\delta_{ql}\widehat{B}_{\alpha}^{kr},
\end{align*}
and analogously
\begin{equation*}
 \big[\widehat{\mu}^*_i(\widehat{\Lambda}^i_{kl}),\widehat{B}_{\alpha^*}^{qr}\big] = \delta_{kr}\widehat{B}_{\alpha^*}^{ql}.
\end{equation*}

This must be compared with the trace dual of the derivative of the conjugation action~\eqref{eq:quantum_conjugation_action}. It reads
\begin{equation*}
 \big(\Lambda^i\big)_{i \in \mathfrak{h}_i}.\big(\widehat{B}^{\alpha}\big)_{\alpha \in \mathcal{G}_1} = \big(\widehat{B}^{\alpha}\Lambda_{s(\alpha)} - \Lambda_{t(\alpha)}\widehat{B}^{\alpha}\big)_{\alpha \in \mathcal{G}_1},
\end{equation*}
which in this local situation reduces to
\begin{equation*}
 \Lambda^i.\big(\widehat{B}_{\alpha},\widehat{B}_{\alpha^*}\big) = \big(\widehat{B}_{\alpha^*}\Lambda^i,-\Lambda^i\widehat{B}_{\alpha}\big).
\end{equation*}
Finally, if $e_{lk} \in \mathfrak{h}_i$ denotes the trace dual of $\Lambda^i_{kl} \in \mathfrak{h}^*_i$~-- i.e., the standard basis vector satisfying $\Tr(e_{lk} \cdot) = \Lambda_{kl}$~-- then
\begin{equation*}
 \big({-}e_{lk}\widehat{B}_{\alpha}\big)_{qr} = -\sum_m (e_{lk})_{qm}\widehat{B}_{\alpha}^{mr} = -\sum_m \delta_{ql}\delta_{km} \widehat{B}_{\alpha}^{mr} = -\delta_{ql}\widehat{B}_{\alpha}^{kr},
\end{equation*}
and analogously
\begin{equation*}
 \big(\widehat{B}_{\alpha^*}e_{lk}\big)_{qr} = \delta_{kr}\widehat{B}_{\alpha^*}^{ql},
\end{equation*}
as needed.
\end{proof}

\begin{Remark}Note that this proof would work verbatim for showing that~\eqref{eq:classical_local_comoment} is a comoment map for the classical $\widehat{H}$-action, replacing commutators with Poisson brackets. Indeed what is used here is that in these computations one just needs to use the Leibnitz identity of the Poisson bracket, and not the commutativity of the associative product.
\end{Remark}

So in conclusion~\eqref{eq:quantum_comoment_global} defines a quantum comoment map which moreover quantise the classical comoment for the $\widehat{H}$-action in the filtered quantisation setting. Then Section~\ref{sec:quantum_hamiltonian_reduction_formal} explains how to obtain an $\hslash$-deformed quantum comoment map $\widehat{\mu}^*_{\hslash} \colon \widehat{U}(\mathfrak{h}^*) \to \widehat{A}$ which is a deformation quantisation of $\mu^*$.

\begin{Remark}[about orbit quantisation]\label{remark:orbit_quantisation}
What is left is to choose an ideal of $\Sym(\mathfrak{h}) \cong \Sym(\mathfrak{h}^*)$ and deform it so that the quantum Hamiltonian reduction yields a deformation of the classical Hamiltonian reduction. This ideal $\mathcal{I}_{\breve{\mathcal{O}}}$ should be that of functions vanishing on some coadjoint orbit $\breve{\mathcal{O}} \subseteq \mathfrak{h}^*$~-- corresponding to an adjoint orbit under the trace duality~-- so that the algebraic Hamiltonian reduction can be related to the geometric symplectic reduction $\mathbb{M} \sslash_{\breve{\mathcal{O}}} \widehat{H}$ as in Remark~\ref{remark:geometric_symplectic_reduction}.

Concentrating on the semisimple orbits, the ideals of functions vanishing on them are not homogeneous. For example in the regular case they are generated by elements $D_i - c_i \in \Sym(\mathfrak{h})$, where $D_i \in \Sym(\mathfrak{h})^{\widehat{H}}$ is an $\widehat{H}$-invariant functions and $c_i \in \mathbb{C}$ the value of $D_i$ on the orbit. Since the associated graded of every ideal in $U(\mathfrak{h})$ is homogeneous such ideals do not admit filtered deformations, but a formal deformation exists in $\widehat{U}(\mathfrak{h})$: consider the ideal of $\widehat{U}(\mathfrak{h})$ generated by the elements $C_i \cdot \hslash^{|C_i|} - c_i$, where $C_i \in U(\mathfrak{h})$ are the Casimir generators of the centre of the universal enveloping algebra and $c_i$ the value of the principal symbol $\sigma_{\lvert C_i \rvert} (C_i) \in A_0^{\mathfrak{h}}$ on the orbit (see~\cite{donin_mudrov_2002_explicit_equivariant_quantization_on_coadjoint_orbits_of_gl_n_c,donin_mudrov_2002_quantum_coadjoint_orbits_in_gl_n} for the general semisimple case).

In the case of nilpotent orbit closures then one can a priori look for filtered deformations in~$U(\mathfrak{h})$ (see \cite[Chapter~10]{collingwood_mcgovern_1993_nilpotent_orbits_in_semisimple_lie_algebras} and the more recent \cite{losev_2016_deformation_of_symplectic_singularities_and_orbit_method_for_semisimple_lie_algebras}).
\end{Remark}

\subsection{Reduction of the quantum action}\label{sec:quantum_action_reduction_compatibility}

In what follows suppose to have chosen an ideal deforming $\mathcal{I}_{\breve{\mathcal{O}}}$, either in $U(\mathfrak{h})$ or $\widehat{U}(\mathfrak{h})$. Up to replacing ideals of $U(\mathfrak{h})$ with their homogeneous versions inside the Rees algebra the $\hslash$-deformed quantum Hamiltonian reduction $R_{\hslash}(\widehat{A},\mathfrak{h},\widehat{\mathcal{I}}_{\breve{\mathcal{O}},\hslash})$ is defined as in Section~\ref{sec:quantum_hamiltonian_reduction_filtered} and~\ref{sec:quantum_hamiltonian_reduction_formal}.

Now we consider again the particular case where $\mathcal{G}$ is a complete $k$-partite quiver on nodes $I = \coprod_{j \in J} I^j$, and where $\varepsilon$ is defined by an embedding $\bm{a} \colon J \hookrightarrow \mathbb{C} \cup \{\infty\}$ as in~\eqref{eq:symplectic_form_weights}. Then the quantum $\SL_2(\mathbb{C})$-action is defined as in Section~\ref{sec:quantisation_of_action}, and we want to reduce the morphisms $\widehat{\varphi}^*_{g,\hslash}(\bm{a}) \colon (A,\ast_{\bm{a}.g}) \to (A,\ast_{\bm{a}})$ for $g \in \SL_{2}(\mathbb{C})$, $\bm{a} \in \mathbf{A}$.

\begin{Theorem}\label{theorem:reduced_quantum_action}
The morphism $\widehat{\varphi}^*_{g,\hslash}(\bm{a})$ naturally induces a reduced morphism
\begin{equation*}
 R\widehat{\varphi}^*_{g,\hslash}(\bm{a}) \colon \ R_{\hslash}\big(\widehat{A},\ast_{\bm{a}.g}\big) \longrightarrow R_{\hslash}\big(\widehat{A},\ast_{\bm{a}}\big),
\end{equation*}
keeping track of the products in the notation~-- but omitting the Lie algebra and the ideal. Moreover if $g' \in \SL_2(\mathbb{C})$ is another element then $R\widehat{\varphi}^*_{gg',\hslash}(\bm{a}) = R\widehat{\varphi}^*_{g,\hslash}(\bm{a}) \circ R\widehat{\varphi}^*_{g',\hslash}(\bm{a}.g)$.
\end{Theorem}

\begin{proof}It is helpful to start from filtered quantisation and show that the family of morphisms $\widehat{\varphi}^*_g(\bm{a}) \colon (A,\ast_{\bm{a}.g}) \to (A,\ast_{\bm{a}})$ in filtered quantisation commutes with the $\widehat{H}$-action~\eqref{eq:quantum_conjugation_action} and ``preserves'' the values of the quantum comoment map~\eqref{eq:quantum_comoment_global}.

As for the action, the statement follows from the fact that the quantum $\SL_2(\mathbb{C})$-action does not depend on a choice of basis for $V$. This yields the inclusion $\widehat{\varphi}^*_g(\bm{a}) \big(A^{\mathfrak{h}}\big) \subseteq A^{\mathfrak{h}}$, and from Lemma~\ref{lemma:reduction_quantum_morphisms} one gets $\widehat{\varphi}^*_{g,\hslash}(\bm{a}) \big(\widehat{A}^{\mathfrak{h}}\big) \subseteq \widehat{A}^{\mathfrak{h}}$.

As for the quantum comoment, the proof of Lemma~\ref{lemma:compatibility_moment_classical_action} is easily modified as it relies on Lemma~\ref{lemma:action_cocycle_symplectic}, and the cocycles $\eta_i \colon \SL_2(\mathbb{C}) \times \mathbf{A} \to \mathbb{C}^{\times}$ of~\eqref{eq:action_cocycle} are the same for the classical and quantum action. Hence one has
\begin{equation}\label{eq:compatibility_moment_quantum_action}
 \widehat{\varphi}^*_g(\bm{a}) \circ \widehat{\mu}^*_{\bm{a}.g} = \widehat{\mu}^*_{\bm{a}}, \qquad \text{for} \quad g \in \SL_2(\mathbb{C}), \, \bm{a} \in \mathbf{A},
\end{equation}
where $\widehat{\mu}^*_a$ is the quantum comoment~\eqref{eq:quantum_comoment_global} for the above choice of $\varepsilon = \varepsilon(\bm{a})$.

Now~\eqref{eq:compatibility_moment_quantum_action} directly implies
\begin{equation*}
 \widehat{\varphi}^*_{g,\hslash}(\bm{a}) \circ \widehat{\mu}^*_{\bm{a}.g,\hslash} = \widehat{\mu}^*_{\bm{a},\hslash}, \qquad \text{for} \quad g \in \SL_2(\mathbb{C}), \  \bm{a} \in \mathbf{A},
\end{equation*}
as $\widehat{\mu}^*_{\bm{a},\hslash} \big(\widehat{\Lambda} \hslash\big) = \widehat{\mu}^*_{\bm{a}} \big(\widehat{\Lambda}\big) \cdot \hslash^2$ for $\widehat{\Lambda} \in \mathfrak{h}$ and $\widehat{\varphi}^*_{g,\hslash}(\bm{a})$ acts as $\widehat{\varphi}^*_g(\bm{a})$ on the coefficients of power series. Then if $\widehat{\mathcal{I}}_{\hslash}$ is any ideal in $\widehat{U}(\mathfrak{g})$ one has $\widehat{\varphi}^*_{g,\hslash} \big(\widehat{\mu}_{\bm{a}.g,\hslash}\big(\widehat{\mathcal{I}}_{\hslash}\big)\big) \subseteq \widehat{\mathcal{J}}_{\bm{a},\hslash}$, where $\widehat{\mathcal{J}}_{\bm{a},\hslash} \subseteq \widehat{A}$ is the two-sided ideal generated by $\widehat{\mu}^*_{\bm{a},\hslash}\big(\widehat{\mathcal{I}}_{\hslash}\big)$. It follows that the ideal generated by $\widehat{\mu}^*_{\bm{a},\hslash}\big(\widehat{\mathcal{I}}_{\hslash}\big)$ is also sent inside~$\widehat{\mathcal{J}}_{\bm{a},\hslash}$, as $\widehat{\varphi}^*_{g,\hslash}(\bm{a})$ is a ring morphism.

Then the existence of the reduced morphism is assured by the second item of Lemma~\ref{lemma:reduction_quantum_morphisms}. Further by uniqueness the identity in the statement of Theorem~\ref{theorem:reduced_quantum_action} follows from the analogous one in the statement of Theorem~\ref{theorem:quantum_action}.
\end{proof}

\begin{Remark}In particular we have shown that reduced morphisms $R\widehat{\varphi}^*_g(\bm{a})$ also exist in filtered quantum Hamiltonian reduction for every choice of ideal in $U(\mathfrak{h})$. By uniqueness they satisfy $R\widehat{\varphi}^*_{gg'}(\bm{a}) = R\widehat{\varphi}^*_g(\bm{a}) \circ R\widehat{\varphi}^*_{g'}(\bm{a}.g)$. Moreover both $R\widehat{\varphi}^*_g(\bm{a})$ and $R\widehat{\varphi}^*_{g,\hslash}(\bm{a})$ are isomorphisms inverted by $R\widehat{\varphi}_{g^{-1}}(\bm{a}.g)$ and $R\widehat{\varphi}^*_{g^{-1},\hslash} (\bm{a}.g)$, respectively.
\end{Remark}

Consider the bundle of noncommutative $\mathbb{C} \llbracket \hslash \rrbracket$-algebras $R_{\hslash}\big(\widehat{\mathcal{A}},\breve{\mathcal{O}}\big) \to \mathbf{A}$ whose fibre over the embedding $a$ is the quantum Hamiltonian reduction of $(\widehat{A},\ast_{\bm{a}})$ at the orbit $\breve{\mathcal{O}} \subseteq \Sym(\mathfrak{h})$. This is not a trivial bundle of associative algebras, but the quantum Hamiltonian reduction alwasy defines a formal deformation of the classical Hamiltonian reduction: this yields a global trivialisation of $R_{\hslash}\big(\mathcal{\widehat{A}},\breve{\mathcal{O}}\big)$ as bundle of $\mathbb{C} \llbracket \hslash \rrbracket$-modules via the canonical $a$-independent identifications $R_{\hslash}(\widehat{A}) \cong R(A_0) \llbracket \hslash \rrbracket$. Then the assignment $(g,\bm{a}) \mapsto R\widehat{\varphi}^*_{g,\hslash}(\bm{a})$ lifts the action on the base to one on the total space.

An analogous statement holds for the bundle of filtered quantum Hamiltonian reductions. Hence we have constructed the quantum Hamiltonian reduction of the action of Theorem~\ref{theorem:quantum_action}, both in filtered quantisation and in deformation quantisation.

\section{Quantum symmetries}\label{sec:quantum_invariance}

In this section we reduce the universal simply-laced quantum connection of Section~\ref{sec:slqc}, and shows the reduction is projectively invariant under the reduced quantum action of Theorem~\ref{theorem:reduced_quantum_action}.

By Remark~\ref{remark:quantum_trace_invariance} the universal simply-laced quantum Hamiltonians $\widehat{H}_i$ are invariant for the $\mathfrak{h}$-action on $\widehat{A}$ defined by~\eqref{eq:quantum_conjugation_action}, i.e., the action induced by the quantum comoment map~\eqref{eq:quantum_comoment_global}. Hence one can map them along the canonical projection $\widehat{\pi}_{\breve{\mathcal{O}}} \colon \widehat{A}^{\mathfrak{h}} \to R_{\hslash}\big(\widehat{A},\mathfrak{h},\widehat{\mathcal{I}}_{\breve{\mathcal{O}},\hslash}\big)$.

\begin{Definition}The element $R\widehat{H}_i \coloneqq \widehat{\pi}_{\breve{\mathcal{O}}}\big(\widehat{H}_i\big)$ is the reduction of the universal simply-laced quantum Hamiltonian $\widehat{H}_i$, for $i \in I$.
\end{Definition}

\begin{Remark}By construction the semiclassical limit of the reduced quantum Hamiltonians equals the reduced classical Hamiltonians, since the projections to the Hamiltonian reductions commute with semiclassical limits. Namely, if one denotes $\widehat{\mathcal{J}}_{\bm{a},\breve{\mathcal{O}},\hslash}$ (resp.~$\mathcal{J}_{\bm{a},\breve{\mathcal{O}}}$) the ideal generated by $\widehat{\mu}^*_{\bm{a},\hslash}\big(\widehat{\mathcal{I}}_{\breve{\mathcal{O}},\hslash}\big)$ inside $\widehat{A}$ (resp. by $\mu^*_{\bm{a}}(\mathcal{I}_{\breve{\mathcal{O}}})$ inside $A_0$) then the reduction of $\widehat{H}_i$ is by definition
\begin{equation*}
 R\widehat{H}_i = \widehat{H}_i + \widehat{\mathcal{J}}_{\bm{a},\breve{\mathcal{O}},\hslash} \in R_{\hslash}\big(\widehat{A},\mathfrak{h},\widehat{\mathcal{I}}_{\breve{\mathcal{O}},\hslash}\big),
\end{equation*}
whose semiclassical limit is
\begin{equation*}
 H_i + \mathcal{J}_{\bm{a},\breve{\mathcal{O}}} = RH_i \in R\big(A_0,\mathfrak{h},\mathcal{I}_{\breve{\mathcal{O}}}\big),
\end{equation*}
using the fact that $\widehat{H}_i$ and $\widehat{\mathcal{J}}_{\bm{a},\breve{\mathcal{O}},\hslash}$ quantise $H_i$ and $\mathcal{J}_{\breve{\mathcal{O}}}$, respectively.
\end{Remark}

Now the quantum analogue of Theorem~\ref{theorem:classical_invariance} can be proved for suitable choices of (co)adjoint orbits. To this end let $f \in A_0^{\widehat{H}}$ be the (finite) product of all traces of isomonodromy cycles of Fig.~\ref{fig:imd_cycles}, and consider the distinguished (Zariski dense) open subset $D(f) \coloneqq \{ B\in \mathbb{M}\,|\,f(B) \neq 0\}$, where $\mathbb{M}$ has the natural structure of complex affine space.

For $\bm{a} \in \mathbf{A}$ set then $U_{\bm{a}} \coloneqq \mu_{\bm{a}}\big(D(f)\big) \subseteq \mathfrak{h}$.

\begin{Lemma}\label{lemma:choice_good_orbits}The subspace $U_{\bm{a}}$ is $\widehat{H}$-invariant, and one has $\Tr(C) \not\in \mathcal{J}_{\bm{a},\breve{\mathcal{O}}}$ for all $\widehat{H}$-orbits $\breve{\mathcal{O}} \subseteq U_{\bm{a}}$ and for all isomonodromy cycles $C \in \mathbb{C}\mathcal{G}_{\cycl}$. Moreover $U_{\bm{a}}$ only depends on the class of the embedding $\bm{a}$ for the $\SL_2(\mathbb{C})$-action.
\end{Lemma}

\begin{proof}The nonempty set $D(f)$ is $\widehat{H}$-stable, since all traces of cycles are $\widehat{H}$-invariant functions. Then the nonempty set $U_{\bm{a}}$ is a union of $\widehat{H}$-orbits, since the moment map is $\widehat{H}$-equivariant. By construction, any $\widehat{H}$-orbit $\breve{\mathcal{O}}$ contained in $U_{\bm{a}}$ is such that $D(f) \cap \mu_{\bm{a}}^{-1}(\breve{\mathcal{O}}) \neq \varnothing$, and thus $\left.\Tr(C)\right\vert_{\mu_{\bm{a}}^{-1}(\breve{\mathcal{O}})} \neq 0$ for all isomonodromy cycles.

As for the second statement, by~\eqref{eq:classical_action_decomposition} the classical $\SL_2(\mathbb{C})$-action at $(g,\bm{a}) \in \SL_2(\mathbb{C}) \times \mathbf{A}$ multiplies the trace $\Tr(C)$ of a cycle~$C$ by the number $\eta = \prod_{e \in C_1} \eta_{t(e)}(g,\bm{a}) \in \mathbb{C}^*$, where $C_1$ is the set of arrows of~$C$~-- as subquiver of~$\mathcal{G}$. Hence the function $\varphi^*_g(\bm{a}) f$ is a nonzero multiple of~$f$, and $\varphi_g(\bm{a})$ stabilises $D(f)$ and $V(f) \coloneqq \{B \in \mathbb{M}|f(B) = 0\}$. Since $\varphi_g(\bm{a})$ is bijective and $\mathbb{M} = D(f) \coprod V(f)$ this implies $\varphi_g(\bm{a}) \big(D(f)\big) = D(f)$, and by Lemma~\eqref{lemma:compatibility_moment_classical_action}
\begin{equation*}
 U_{\bm{a}} = \mu_{\bm{a}} \big(D(f)\big) = \mu_{\bm{a}.g} \big(D(f)\big).\tag*{\qed}
\end{equation*}\renewcommand{\qed}{}
\end{proof}

Choose then $\bm{a} \in \mathbf{A}$ and an orbit $\breve{\mathcal{O}} \subseteq U_{[\bm{a}]}$, where $[\bm{a}] \in \mathbf{A} \big\slash \SL_2(\mathbb{C})$ is the class of $\bm{a}$. Let $\mathcal{Q}_{\bm{a},\hslash} \colon \Tr(\mathbb{C}\mathcal{G}_{\cycl}) \to \widehat{A}$ be the quantisation map of Section~\ref{sec:quantum_trace}, and define $\IMD \subseteq \mathbb{C}\mathcal{G}_{\cycl}$ as the finite-dimensional vector space spanned by isomonodromy cycles and the zero-length cycles at each node.

\begin{Corollary}[quantisation and reduction commute]\label{cor:quantisation_and_reduction_commute}There exists a reduced quantisation map~$R\mathcal{Q}_{\bm{a},\hslash}$ defined on the classical Hamiltonian reduction of $\Tr(\IMD)$ and taking values into the quantum Hamiltonian reduction $R_{\hslash}\big(\widehat{A},\mathfrak{h},\widehat{\mathcal{I}}_{\breve{\mathcal{O}}}\big)$. It sends the reduced classical simply-laced isomonodromy system to the reduced quantum one, and it intertwines the reduced $\SL_2(\mathbb{C})$-actions along the orbit of $\bm{a}$:
\begin{equation*}
 R\mathcal{Q}_{\bm{a},\hslash} \circ R\varphi^*_g(\bm{a}) = R\widehat{\varphi}_{g,\hslash}(\bm{a}) \circ R\mathcal{Q}_{\bm{a},\hslash}, \qquad \text{for} \quad g \in \SL_2(\mathbb{C}).
\end{equation*}
\end{Corollary}

\begin{proof}If it exists, the reduced map $R\mathcal{Q}_{\bm{a},\hslash}$ is defined by the identity $R\mathcal{Q}_{\bm{a},\hslash} \circ \pi_{\breve{\mathcal{O}}} = \widehat{\pi}_{\breve{\mathcal{O}}} \circ \mathcal{Q}_{\bm{a},\hslash}$ on $\Tr(\IMD) \subseteq A_0^{\mathfrak{h}}$. The existence follows from $\mathcal{J}_{\bm{a},\breve{\mathcal{O}}} \cap \Tr(\IMD) \subseteq \Ker(\mathcal{Q}_{\bm{a},\hslash})$, which is verified for this choice of orbit, and the defining identity assures that
\begin{equation}
 R\mathcal{Q}_{\bm{a},\hslash} (RH_i) = R\mathcal{Q}_{\bm{a},\hslash} \circ \pi_{\breve{\mathcal{O}}}(H_i) = \widehat{\pi}_{\breve{\mathcal{O}}} \circ \mathcal{Q}_{\bm{a},\hslash} (H_i) = \widehat{\pi}_{\breve{\mathcal{O}}} \big(\widehat{H}_i\big) = R\widehat{H}_i.
\end{equation}

The last statement follows from Lemma~\ref{lemma:quantisation_and_action_commute}.
\end{proof}

An analogous statement holds for the quantisation map in filtered quantisation.

\subsection{Quantum invariance}

Now consider the reduced quantum isomonodromy equations for a local section $\widehat{H}$ of the bundle $R_{\hslash}\big(\widehat{A},\ast_{\bm{a}}\big) \times \mathbf{B} \to \mathbf{B}$. They read
\begin{equation}
\label{eq:reduced_quantum_imd_equations}
 \partial_{t_i} \widehat{H} = \big[R\widehat{H}_i,\widehat{H}\big]_{\bm{a}},
\end{equation}
using the commutator for the $\bm{a}$-dependent noncommutative product of the quantum Hamiltonian reduction. The differential equations~\eqref{eq:reduced_quantum_imd_equations} are the dynamical equations for the time-evolution of the quantum observable $\widehat{H}$ in the Heisenberg picture of motion in quantum mechanics~-- where the quantum state is fixed.

\begin{Theorem}\label{theorem:quantum_invariance}
The reduced quantum isomonodromy equations~\eqref{eq:reduced_quantum_imd_equations} are invariant under the reduced quantum $\SL_2(\mathbb{C})$-action along the $\SL_2(\mathbb{C})$-orbit of $\bm{a}$.
\end{Theorem}

\begin{proof}Putting together the previous statements yields
\begin{equation}\label{eq:quantum_shift}
 R\widehat{\varphi}^*_{g,\hslash}(\bm{a}) R\widehat{H}_i = R\widehat{H}_i + c_i, \qquad \text{for} \quad i \in I, \  g \in \SL_2(\mathbb{C}),
\end{equation}
where $c_i \in \mathbb{C}$ is the same constant of Theorem~\ref{theorem:classical_invariance}. Indeed by Corollary~\ref{cor:quantisation_and_reduction_commute} and Theorem~\ref{theorem:classical_invariance}:
\begin{align*}
 R\widehat{\varphi}^*_{g,\hslash}(\bm{a}) R\widehat{H}_i &= R\widehat{\varphi}^*_{g,\hslash}(\bm{a}) \circ \widehat{\pi}_{\breve{\mathcal{O}}}\big(\mathcal{Q}_{\bm{a},\hslash}(H_i)\big) = R\widehat{\varphi}^*_{g,\hslash}(\bm{a}) \circ R\mathcal{Q}_{\bm{a},\hslash} \big(RH_i\big) \\
 &= R\mathcal{Q}_{\bm{a},\hslash} \circ R\varphi^*_g(\bm{a}) \big(RH_i\big) = R\mathcal{Q}_{\bm{a},\hslash}\big(RH_i + c_i\big)
 = R\widehat{H}_i + c_i,
\end{align*}
using the fact that the quantisation map and its reduction preserve constants, where $\mathbb{C} \hookrightarrow \IMD$ is embedded as the sum of the zero-length cycles at each node.

Now consider the extended bundle $R_{\hslash}\big(\widehat{\mathcal{A}},\breve{\mathcal{O}}\big) \times \mathbf{B} \to \mathcal{O}_{\bm{a}} \times \mathbf{B}$, where $\mathcal{O}_{\bm{a}} \subseteq \mathbf{A}$ is the $\SL_2(\mathbb{C})$-orbit of $\bm{a}$. This bundle is by construction trivial along the variations in $\mathbf{B}$, and if $\widehat{H}$ is a local section then we want to relate the reduced quantum isomonodromy equations for of $\widehat{H}$ and $R\widehat{\varphi}^*_g(\bm{a}) \widehat{H}$ for $g \in \SL_2(\mathbb{C})$. Assuming that $\widehat{H}$ is a solution of~\eqref{eq:reduced_quantum_imd_equations} at $\bm{a}$ then:
\begin{align*}
 \partial_{t_i} \big(R\widehat{\varphi}^*_{g,\hslash}(\bm{a})\widehat{H}\big) &= R\widehat{\varphi}^*_{g,\hslash}(\bm{a}) \partial_{t_i} \widehat{H} = R\widehat{\varphi}^*_{g,\hslash}(\bm{a}) \big[R\widehat{H}_i,\widehat{H}\big]_{\bm{a}.g} \\
 &= \big[\widehat{\varphi}^*_{g,\hslash}(\bm{a})R\widehat{H}_i,R\widehat{\varphi}^*_{g,\hslash}(\bm{a}) \widehat{H}\big]_{\bm{a}} = \big[R\widehat{H}_i,R\widehat{\varphi}^*_{g,\hslash}(\bm{a}) \widehat{H}\big]_{\bm{a}},
\end{align*}
because the $\SL_2(\mathbb{C})$-action does not depend on the point in the base space $\mathbf{B}$, using~\eqref{eq:quantum_shift} for the last equality, and since $R\widehat{\varphi}^*_{g,\hslash}(\bm{a})$ is a morphism of associative algebras.
\end{proof}

\begin{Remark}Hence $\widehat{\varphi}^*_{g,\hslash}(\bm{a}) \widehat{H}$ is a solution of~\eqref{eq:reduced_quantum_imd_equations} at $\bm{a} \in \mathbf{A}$, so that a single set of quantum isomonodromy equation at one point controls the time evolution along the whole of the $\SL_2(\mathbb{C})$-orbit of $\bm{a}$.

Analogous statements hold in filtered quantisation.
\end{Remark}

Theorem~\ref{theorem:quantum_invariance} can be geometrically rephrased in terms of flat isomorphisms of vector bundles, as follows. Fix a point $\bm{a} \in \mathbf{A}$, and consider the trivial bundle of noncommutative $\mathbb{C} \llbracket \hslash \rrbracket$-algebras $\big(\widehat{A},\ast_{\bm{a}}\big) \times \mathbf{B} \to \mathbf{B}$, equipped with the universal simply-laced quantum connection $\widehat{\nabla}_{\bm{a}}$ of Theo\-rem~\ref{theorem:slqc}. Taking fibrewise quantum Hamiltonian reduction at a choice of coadjoint orbit \mbox{$\breve{\mathcal{O}} \subseteq U_{[\bm{a}]}$} yields a new trivial bundle $R_{\hslash}\big(\widehat{A},\ast_{\bm{a}}\big) \times \mathbf{B} \to \mathbf{B}$. This bundle carries the reduced simply-laced quantum connection, that is the strongly flat connection defined by
\begin{equation*}
 R\widehat{\nabla}_{\bm{a}} = {\rm d} - R\widehat{\varpi}, \qquad \text{where} \quad R\widehat{\varpi} \coloneqq \sum_{i \in I} R\widehat{H}_i\,
 {\rm d}t_i,
\end{equation*}
with the reduced quantum Hamiltonian acting via their commutator. Then Theorem~\ref{theorem:quantum_invariance} states that the reduced morphism
\begin{equation*}
 R\widehat{\varphi}^*_{g,\hslash}(\bm{a}) \colon \ R_{\hslash}\big(\widehat{A},\ast_{\bm{a}.g}\big) \times \mathbf{B} \longrightarrow R_{\hslash}\big(\widehat{A},\ast_{\bm{a}}\big) \times \mathbf{B}
\end{equation*}
is a flat isomorphism of vector bundles equipped with (flat) connections.

Hence in brief the quantisation of the symmetries of the reduced isomonodromy connection yields symmetries of the reduced universal simply-laced quantum connection. A similar statement holds for the simply-laced quantum connection in filtered quantisation.

\section*{Future directions}

It was shown in~\cite{boalch_2012_simply_laced_isomonodromy_systems} that the $\SL_2(\mathbb{C})$ symmetry group can be used to obtain Kac--Moody Weyl group symmetries, generalising the Okamoto symmetries of the Painlev\'e equations~IV,~V and~VI. Indeed all Painlev\'e equations can be written as nonautonomous Hamiltonian systems, with the cases~IV,~V and~VI all within the scope of the simply-laced isomonodromy systems, and the results of this paper yield quantum version of the resulting Kac--Moody Weyl group symmetries. This generalises some results of~\cite{nagoya_yamada_2014_symmetries_of_quantum_lax_equations_for_the_painleve_equations}, and will be studied in future work.

Also, the quantisation of Section~\ref{sec:quantum_trace} goes in the direction of the quantisation of the necklace Lie algebra of a quiver, independently from~\cite{schedler_2005_a_hopf_algebra_quantizing_a_necklace_lie_algebra_canonically_associated_to_a_quiver} which considers a more refined Hopf algebra quantisation. It should in principle be possible to relate these constructions.

\subsection*{Acknowledgements}

I thank Giovanni Felder for truly helpful discussions, Pavel Etingof and Travis Schedler for truly helpful e-mail exchanges, and Philip Boalch for both.
This research began while the author was a member of the Laboratoire de Math\'ematiques d'Orsay, and was supported by a temporary research and teaching assistant contract; it was concluded while the author was a member of the Department of Mathematics at the ETH of Z\"urich, and was supported by the National Centre of Competence in Research ``SwissMAP-The Mathematics of Physics'' of the Swiss National Science Foundation.

\pdfbookmark[1]{References}{ref}
\LastPageEnding

\end{document}